\documentclass[12pt]{amsart}
\usepackage{amssymb,amsfonts,amsmath}
\usepackage{color}
\usepackage{graphicx}
\usepackage{subfigure}
\usepackage{xcolor}

\newtheorem{theorem}{Theorem}[section]
\newtheorem{thm}[theorem]{Theorem}
\newtheorem{lemma}[theorem]{Lemma}
\newtheorem{proposition}[theorem]{Proposition}

\newtheorem{assumption}{Assumption}
\newtheorem{corollary}[theorem]{Corollary}
\newtheorem{definition}[theorem]{Definition}

\newtheorem{example}[theorem]{Example}
\newtheorem{remark}[theorem]{Remark}
\newtheorem{assumptions}[theorem]{Assumptions}

\newtheorem{Lemma}[theorem]{Lemma}
\newtheorem{Corollary}[theorem]{Corollary}
\newtheorem{Definition}[theorem]{Definition}
\newtheorem{Example}[theorem]{Example}

\newtheorem{Theorem}[theorem]{Theorem}
\newtheorem{Remark}[theorem]{Remark}

\newcommand{\be}{\begin{equation}}
\newcommand{\ee}{\end{equation}}
\newcommand{\beq}{\begin{equation*}}
\newcommand{\eeq}{\end{equation*}}
\newcommand{\enq}{\end{equation}}
\newcommand{\ben}{\begin{eqnarray}}
\newcommand{\een}{\end{eqnarray}}
\newcommand{\bea}{\begin{eqnarray*}}
\newcommand{\eea}{\end{eqnarray*}}

\newcommand{\At}{ {\widetilde{A}}}
\newcommand{\Bt}{{\widetilde{B}}}
\newcommand{\Mt}{ {\widetilde{M}}}
\def\Gammat{{\widetilde{\Gamma}}}

\def\og{\overline{g}}
\def\ophi{\overline{\phi}}
\def\wt{\tilde{w}}
\newcommand{\Hc}{ {\mathcal{H}}}
\def\cK{{\mathcal K}}
\def\cH{{\mathcal H}}
\newcommand{\K}{ {\mathcal{K}}}
\newcommand{\M}{ {\mathcal{M}}}
\newcommand{\cL}{ {\mathcal{L}}}
\newcommand{\cF}{ {\mathcal{F}}}
\newcommand{\Sc}{ {\mathcal{S}}}
\newcommand{\St}{ {\widetilde{S}}}
\newcommand{\Sct}{ {\widetilde{\mathcal{S}}}}
\newcommand{\Sco}{ {\overline{\mathcal{S}}}}
\newcommand{\Scto}{ {\overline{\widetilde{\mathcal{S}}}}}
\newcommand{\Tc}{ {\mathcal{T}}}
\newcommand{\sign}{\mbox{\rm sign}}
\def\ind{{\mathrm{def\,}}}
\def\ker{{\mathrm{ker\,}}}
\def\Ran{{\mathrm{Ran\,}}}
\def\Span{{\mathrm{Span\,}}}
\newcommand{\Rr}{{\mathbb{R}}}

\newcommand{\Cc}{{\mathbb{C}}}

\newcommand{\llangle}{\left\langle}
\newcommand{\rrangle}{\right\rangle}
\renewcommand{\ll}{\left\langle}
\newcommand{\rr}{\right\rangle}
\newcommand{\eps}{\varepsilon}

\newcommand{\la}{\lambda}
\def\C{\mathbb C}
\def\R{\mathbb R}
\def\N{\mathbb N}
\def\Z{\mathbb Z}

\newcommand{\supp}{\mbox{\rm supp}}
\newcommand{\essran}{\mbox{\rm essran}}

\newcommand{\norm}[1]{\left\Vert#1\right\Vert}

\newcommand{\ian}[1]{{\color{black!80!black}#1}}

\newcommand{\change}[1]{{\color{black!80!black}#1}}
\newcommand{\changes}[1]{{\color{black!80!black}#1}}

\title[Detectable Subspace]{\change{An abstract inverse problem for boundary triples with an application to the Friedrichs Model}
}

\author{B.M.~Brown}
\address{
Cardiff School of Computer Science   and  Informatics\\ Cardiff University \\ Queen's Buildings\\ 5 The Parade\\ Cardiff CF24 3AA\\ UK}
\email{Malcolm.Brown@cs.cardiff.ac.uk}

\author{M.~Marletta}
\address{
School of Mathematics\\ Cardiff University\\ Senghennydd Road\\ Cardiff CF24 4AG\\ UK}
\email{MarlettaM@cardiff.ac.uk}

\author{S.~Naboko}
\address{
Department of Math.~Physics, Institute of Physics, St.~Petersburg State University\\ 1 Ulianovskaia, St.~Petergoff, St.~Petersburg, 198504, Russia}
   \curraddr{
   School of Mathematics, Statistics and Actuarial Science\\ University of Kent\\ Cornwallis Building\\ Canterbury CT2 7NF\\ UK}
\email{sergey.naboko@gmail.com}

\author{I.~Wood}
\address{
School of Mathematics, Statistics and Actuarial Science\\ University of Kent\\ Cornwallis Building\\ Canterbury CT2 7NF\\ UK}
\email{i.wood@kent.ac.uk}

\thanks{
M.~Marletta and S.N.~Naboko gratefully acknowledge the support of the Leverhulme
Trust, grant RPG167, and of the Wales Institute of Mathematical
and Computational Sciences. \ian{S.N.~Naboko also was partially supported by the grant  NCN t 2013/09/BST1/04319 (Poland)
and RFBR 12 - 01 - 00215-a, as well as the EC Marie Curie grant PIIF-GA-2011-299919.}
}

\begin{document}

\change{
\begin{abstract}
We discuss the detectable subspaces of an operator. We analyse the relation between the $M$-function (the abstract Dirichlet to Neumann map) and the resolvent bordered by projections onto the detectable subspaces. The abstract results are explored further by  
an extensive study of the Friedrichs model, together with illustrative applications to the Schr\"{o}dinger and Hain-L\"{u}st-type models.
\end{abstract}
\maketitle

\section{Introduction}\label{section:0}
In this paper we  consider inverse problems in a boundary triple setting involving a formally adjoint pair of operators $A$ and $\tilde{A}$, as studied in \cite{BGW09,BHMNW09,BMNW08,Lyantze,MM97,MM99,MM02}. We define, and develop formulae for, the detectable subspace \changes{(see Definition \ref{def:DS})}
 associated with the information available from the abstract Dirichlet to Neumann maps (Titchmarsh-Weyl functions) $M(\lambda)$. 
We examine the extent to which the following questions can be answered at a purely abstract level.
\begin{enumerate}
\item Is the function $M(\lambda)$ uniquely determined from a knowledge of resolvents \changes{reduced to} the detectable subspace?
\item Can the resolvent, bordered by projections onto the detectable subspaces, be determined from $M(\lambda)$?
\item What can be said about the relationship between analytic continuation of $M(\lambda)$ and analytic continuation of
 bordered resolvents?
\item What is the relationship between the rank of the jump in $M(\lambda)$ and the rank of the jump in the bordered resolvent
 across a line of essential spectrum, w.l.o.g. the real axis, when one has a limiting absorption principle?
\item \label{Q5} To what extent can the detectable subspace be explicitly described?
\end{enumerate}

Illustrative examples include the Schr\"{o}dinger operator and Hain-L\"{u}st-type models which we also examined in \cite{BHMNW09}.
However the main concrete example studied in this paper is the Friedrichs Model, discussed at length in Sections \ref{section:6}, \ref{section:7}, \ref{section:Toep}, together with the relevant appendices.  These results reveal many connections to problems in modern complex analysis, including the theory of Hankel and Toeplitz operators, and demonstrate the interplay between complex analysis and 
operator theory in the description of the detectable subspace (see e.g.~the appearance of the Riesz-Nevanlinna factorisation theorem in Theorem \ref{B2}). \changes{We consider the  Friedrichs model as a key example for the development of the theory of detectable subspaces,} because it allows a precise description of the structure of the the detectable subspace in many cases, while exhibiting such a variety of behaviours that one can hardly expect to obtain a description of the space in all cases in unique terms. It shows the problem of reconstruction of the detectable part of the operator from the $M$-function, well-known for Sturm-Liouville problems \cite{Borg49,Marchenko}, is not always possible. Our results for this example show that the detectable part of the operator can partially be recovered from the $M$-function. At least in the symmetric case we would expect this recovery to be possible up to unitary equivalence \cite{Ryzhov}.

The paper is arranged as follows. Section \ref{section:1} introduces boundary triples, $M$-functions, solution operators and the detectable subspace. Section \ref{section:2} shows the concrete realizations of these abstract objects for Schr\"{o}dinger operators, Hain-L\"ust-type operators and the Friedrichs model. Sections \ref{section:3}, \ref{section:4}, \ref{section:5} present various abstract results concerning the relationship between the bordered resolvent and the $M$-function. In particular, in Section \ref{section:3}, we prove that the $M$-function is uniquely determined by one bordered resolvent. It can be reconstructed from one bordered resolvent
and two closed solution operator ranges or by two bordered resolvents associated with different boundary conditions. We show that the bordered resolvent can be determined from the $M$-function and a family of solution operator ranges. Section \ref{section:4} examines simultaneous analytic continuation of the $M$-function and bordered resolvents, while Section \ref{section:5} deals with jumps of the $M$-function and the bordered resolvent across the essential spectrum. 

Sections \ref{section:6} onwards, including the appendices, deal with the Friedrichs model. \changes{In Section \ref{section:6} we consider the reconstruction of the $M$-function from one restricted resolvent for the Friedrichs model. Sections \ref{section:7} and \ref{section:Toep} deal with determining the detectable subspace for various combinations of the parameters of the model. In both these sections, results and techniques from complex analysis will be important; whilst in Section \ref{section:7} Hankel operators will make an appearance, the results in Section  \ref{section:Toep} rely on the theory of Toeplitz operators. Many of the proofs  from these sections can be found in the appendices.}

We conclude this introduction by mentioning that there has been an explosion of interest in boundary triples in the last decade,
in particular around their application to partial differential equations usually in the self-adjoint case (see, e.g.~\cite{AGW14,AB09,AP04,BL07,BL12,BLL13a,BLL13b,BM13,BR12,BMNW08,DM91,GMZ07,GM08,GM09,GM09b,GM11,Gru08,Gru14,HMM13,KK04,Mal10,Pos08,PR10,Ryzhov}). Some interesting ODE applications have also appeared, such as Mikhailets' and Sobolev's
study \cite{MS99} of the common eigenvalue problem. Generalisations to relations have been studied by Derkach, Hassi, Malamud and de Snoo \cite{DHMdS06,DHMdS09}. However the situation with inverse problems remains problematic: one of the striking differences between
Schr\"{o}dinger operators in dimension $d=1$ and dimension $d>1$ is that, for $d>1$, the potential can be uniquely recovered from a knowledge of the Dirichlet to Neumann map at a single value of the spectral parameter \cite{Isakov}. The fact that this is not true for
$d=1$ \cite{Borg49} already indicates that abstract techniques will generally be of limited value unless supplemented by a detailed study of the concrete operators to which they will be applied.}

\section{Definition of the detectable subspace and some properties}\label{section:1}
{We use the following assumptions and notation throughout our article.
\begin{enumerate}
  \item $A$, $\At$ are closed, densely defined operators on domains in a Hilbert space $H$.
  \item $A$ and $\At$ are an adjoint pair, i.e. $A^*\supseteq\At$ and $\At^*\supseteq A$.
%  \item Whenever considering $D(\At^*)$ as a linear space it will be equipped with the graph norm. Since $\At^*$ is closed, this
%makes $D(\At^*)$ a Hilbert space.
\end{enumerate}
\begin{proposition}\cite[(Lyantze, Storozh '83)]{Lyantze}. For each adjoint pair of closed densely defined operators on $H$,
there exist ``boundary spaces'' $\cH$, $\cK$ and ``trace operators''
  \[ \Gamma_1:D(\At^*)\to\cH,\quad \Gamma_2:D(\At^*)\to\cK,\quad \Gammat_1:D(A^*)\to\cK\quad \hbox{ and }\quad  \Gammat_2:D(A^*)\to\cH \]
  such that for $u\in D(\At^*) $ and $v\in D(A^*)$ we have an abstract Green formula
        \begin{equation}\label{Green}
          \llangle \At^* u, v\rrangle_H - \ian{\Big\langle u,A^*v \Big\rangle_H} = \llangle\Gamma_1 u, \Gammat_2 v\rrangle_\cH - \llangle \Gamma_2 u, \Gammat_1v\rrangle_\cK.
        \end{equation} 
The trace operators $\Gamma_1$, $\Gamma_2$, $\Gammat_1$ and $  \Gammat_2 $ are bounded with respect to the graph norm. The pair $(\Gamma_1,\Gamma_2)$ is 
surjective onto $\cH\times\cK$ and $(\Gammat_1,\Gammat_2)$ is surjective onto $\cK\times\cH$. Moreover, we have 
\begin{equation}\label{domains}
	D(A)= D(\At^*)\cap\ker\Gamma_1\cap \ker\Gamma_2 \quad \hbox{ and } \quad D(\At)= D(A^*)\cap\ker\Gammat_1\cap \ker\Gammat_2.
\end{equation}
The collection $\{\cH\oplus\cK, (\Gamma_1,\Gamma_2), (\Gammat_1,\Gammat_2)\}$ is called a \ian{boundary triple} for the adjoint pair $A,\At$.
\end{proposition}
}

Malamud and Mogilevskii \cite{MM02} use this setting to define Weyl $M$-functions associated with boundary 
triples. In \cite{BMNW08}, we used a slightly different setting in which the boundary conditions and Weyl function contain an additional operator $B\in\cL(\cK,\cH)$. \changes{We now summarize some results from \cite{BMNW08} for the convenience of the reader.  }

\begin{definition}\label{defmfn} Let $B\in\cL(\cK,\cH)$ and $\Bt\in\cL(\cH,\cK)$. We define extensions
of $A$ and $\At$ (respectively) by
\[ A_B:=\At^*\vert_{\ker(\Gamma_1-B\Gamma_2)} \hbox{ and } \At_\Bt:=A^*\vert_{\ker(\Gammat_1-\Bt\Gammat_2)}.\]  
In the following, we assume $\rho(A_B)\neq\emptyset$, in particular $A_B$ is a closed operator.
For $\lambda\in\rho(A_B)$, we define the $M$-function via
\[ M_B(\lambda):\Ran(\Gamma_1-B\Gamma_2)\to\cK,\ M_B(\lambda)(\Gamma_1-B\Gamma_2) u=\Gamma_2 u \hbox{ for all } u\in \ker(\At^*-\lambda)\]
and for $\lambda\in\rho(\At_\Bt)$, we define 
\[ \Mt_\Bt(\lambda):\Ran(\Gammat_1-\Bt\Gammat_2)\to\cH,\ \Mt_\Bt(\lambda)(\Gammat_1-\Bt\Gammat_2) v=\Gammat_2 v 
\hbox{ for all } v\in \ker(A^*-\lambda).\]
\end{definition}
It \ian{will follow from Lemma \ref{slamwd}} that $M_B(\lambda)$ and $\Mt_\Bt(\lambda)$ are well defined for $\lambda\in \rho(A_B)$ and $\lambda\in \rho(\At_\Bt)$,
respectively. Moreover, in our situation $\Ran(\Gamma_1-B\Gamma_2)=\cH$ and $\Ran(\Gammat_1-\Bt\Gammat_2)=\cK$, so the $M$-functions are defined on the whole spaces.

\begin{definition} (Solution Operator) For $\lambda\in\rho(A_B)$, we define the linear operator $S_{\lambda,B}:\Ran(\Gamma_1-B\Gamma_2)\to 
\ker(\At^*-\lambda)$ by
\begin{eqnarray}\label{slamdef}
  (\At^*-\lambda)S_{\lambda,B} f=0,\ (\Gamma_1-B\Gamma_2)S_{\lambda,B} f=f,
\end{eqnarray}
i.e.  $S_{\lambda,B}=\left( (\Gamma_1-B\Gamma_2)\vert_{\ker(\At^*-\lambda)}\right)^{-1}$.
For $\lambda\in\rho(\At_B^*)$, we define the linear operator $\St_{\lambda,B^*}:\Ran(\Gammat_1-B^*\Gammat_2)\to 
\ker(A^*-\lambda)$ by
\begin{eqnarray}
  (A^*-\lambda)\St_{\lambda,B^*} f=0,\ (\Gammat_1-B^*\Gammat_2)\St_{\lambda,B^*} f=f.
\end{eqnarray}
\end{definition}

All following results have a corresponding version for the quantities $\Mt_\Bt$, $\St_{\lambda,B^*}$ etc.~obtained from the formally adjoint problem. 

\begin{remark}
\begin{enumerate}
\item \ian{As we are not interested in characterising all closed extensions of $A$, in this paper we will assume for simplicity that $B\in\cL(\cK,\cH)$. A discussion of all closed extensions of $A$ in the boundary triple setting can be found in \cite{BGW09}.}
	\item Note that $M_B(\lambda)=\Gamma_2 S_{\lambda,B}$.
	\item M-functions associated with different boundary conditions are related by the Aronszajn-Donoghue formula \ian{(cf. also \ref{slam2})}
	\be\label{Mfn}M_B(\lambda)= (I+M_B(\lambda)(B-C))M_C(\lambda)=M_C(\lambda)(I+(B-C)M_B(\lambda)).\ee
 \end{enumerate}
\end{remark}
%Since we shall use solution operators quite extensively in the sequel, we include the proof of the following lemma, for completeness.

\changes{The following lemma contains the results of \cite[Lemma 2.4 and Corollary 2.5]{BHMNW09}.

\begin{lemma}\label{slamwd}
\begin{enumerate}
	\item $S_{\lambda,B}$ is well-defined for $\lambda\in\rho(A_B)$.
	\item For each
$f\in\Ran(\Gamma_1-B\Gamma_2)$ the map from $\rho(A_B)\to H$ given by $\lambda\mapsto S_{\lambda,B} f$ is analytic.
\item For $\lambda,\lambda_0\in\rho(A_B)$ we have
\be S_{\lambda,B} =  S_{\lambda_0,B} + (\lambda-\lambda_0)(A_B-\lambda)^{-1} S_{\lambda_0,B}. \label{eq:sdiff} \ee
\end{enumerate}
\end{lemma}
}

\ian{The difference of two resolvents of the operator can be related to the $M$-function by Krein-type resolvent formulae, such as 
\begin{eqnarray}\label{Krein}
(A_C-\lambda)^{-1} - (A_B-\lambda)^{-1} &=& S_{\lambda,C}(I+(B-C)M_B(\lambda))(\Gamma_1-B\Gamma_2)(A_C-\lambda)^{-1}\\ \nonumber
                   &=& S_{\lambda,C}(I+(B-C)M_B(\lambda))(C-B)\Gamma_2(A_C-\lambda)^{-1},
\end{eqnarray}
for  $B,C\in \cL(\cK,\cH)$ and $\lambda\in\rho(A_B)\cap\rho(A_C)$ (see \cite[Theorem 2.6]{BHMNW09}).}

\changes{The following formula already appears in some proofs in \cite{BHMNW09,BMNW08}, but due to its importance in our later analysis, we state and prove it here.}
\begin{lemma}\label{intbyparts}
For every $F\in D(\At^*)$ and $v\in D(A^*)$ and $\lambda\in\rho(A_B)$ we have 
\be \llangle F-(A_B-\lambda)^{-1}(\At^*-\lambda)F,(A^*-\overline{\lambda}I)v\rrangle
 = \ian{ \llangle M_B(\lambda)f,(\Gammat_1-B^*\Gammat_2)v\rrangle_{\K} -\llangle f,\Gammat_2v\rrangle_{\Hc}}
\label{eq:fund} \ee
where $f=(\Gamma_1-B\Gamma_2)F$.
\end{lemma}

\begin{proof}
Set $w:=F-(A_B-\lambda)^{-1}(\At^*-\lambda)F$. 
Then $w\in\ker(\At^*-\lambda)$, so $$M_B(\lambda)(\Gamma_1-B\Gamma_2)w=M_B(\lambda)f=\Gamma_2 w \hbox{ and } \Gamma_1 w=(\Gamma_1-B\Gamma_2+B\Gamma_2) w=(I+BM_B(\lambda))f.$$
Green's identity \ian{\eqref{Green}} for any $v\in D(A^*)$ gives
\begin{eqnarray*}%\label{greeneps}
-\Big\langle w,(A^*-\overline{\lambda})v\Big\rangle _H
&=&\llangle (\At^*-\lambda)w, v\rrangle _H -\Big\langle  w,(A^*-\overline{\lambda})v\Big\rangle _H \\ \nonumber
&=&\llangle \Gamma_1 w,\Gammat_2 v\rrangle _\cH - \llangle \Gamma_2 w,\Gammat_1 v\rrangle _\cK  \\ \nonumber
&=&\llangle (I+BM_B(\lambda))f,\Gammat_2 v\rrangle _\cH - \llangle M_B(\lambda)f,\Gammat_1 v\rrangle _\cK  \\ \nonumber
&=&\llangle f,\Gammat_2 v\rrangle _\cH - \llangle M_B(\lambda)f,(\Gammat_1 -B^*\Gammat_2)v\rrangle _\cK .
\end{eqnarray*}
\end{proof}
\change{We are now ready to define one of the main concepts of the paper, the detectable subspaces, introduced in   \cite{BHMNW09}.

\begin{definition}\label{def:DS}
Fix $\mu_0\not\in \sigma(A_B)$. We define the spaces
\be \Sc_B = \Span_{\delta\not\in \sigma(A_B)}(A_B-\delta I)^{-1}
       \mbox{Ran}(S_{\mu_0,B})\label{eq:mm1}, \ee
\be \Tc_B = \Span_{\mu\not\in \sigma(A_B)}
        \mbox{Ran}(S_{\mu,B})\label{eq:mm1b}, \ee
and similarly,
\be \widetilde{\Sc}_{B^*} = \mbox{\rm Span}_{\delta\not\in \sigma(\tilde{A}_{B^*})}
 (\tilde{A}_{B^*}-\delta I)^{-1}\mbox{Ran}(\tilde{S}_{\tilde{\mu},B^*})\label{eq:tmm1}, \ee
\be \widetilde{\Tc}_{B^*} = \mbox{\rm Span}_{\mu \not\in \sigma(\tilde{A}_{B^*})} \mbox{Ran}(\tilde{S}_{\mu,B^*}). \ee
We call $\overline{\Sc_B}$  and $\overline{\Sct_{B^*}}$ the detectable subspaces.
\end{definition}}

We now consider the dependence of these spaces on $\mu_0$ and $B$.
\begin{proposition}\label{prop:2.7}
\begin{enumerate}
	\item\label{boundedres} Let $B\in\cL(\cK,\cH)$. Assume that there is a sequence $(z_n)_{n\in\mathbb{N}}$ in $\mathbb{C}$ 
with $|z_n|\rightarrow\infty$ and $(\| z_n(A_B-z_nI)^{-1}\|)_{n\in{\mathbb N}}$ is bounded. Then we have
$ \overline{\Sc_B} = \overline{\Tc_B}. $
In particular, $\overline{\Sc_B}$ is independent of $\mu_0$.     	
	\item Let $B,C\in\cL(\cK,\cH)$. If $\rho(A_B)\cup\rho(A_C)\subseteq\overline{\rho(A_B)\cap\rho(A_C)}$, then $ \overline{\Tc_B}=\overline{\Tc_C}$.
	\item\label{Bindependence} Suppose that for all $B,C\in\cL(\cK,\cH)$, we have $\rho(A_B)\cup\rho(A_C)\subseteq\overline{\rho(A_B)\cap\rho(A_C)}$.
	Then $\overline{\Tc_B}=\overline{\Span_{\lambda\in\Lambda} \ker (\At^*-\lambda)}$, where $\Lambda=\bigcup_{C\in\cL(\cK,\cH)}\rho(A_C)$.
\end{enumerate}
\end{proposition}

\begin{proof}
\begin{enumerate}
  \item This is shown in \cite[Lemma 3.1]{BHMNW09}.
	\item From \cite[Proposition 4.5]{BMNW08} we have
\begin{equation}\label{slam2}
S_{\lambda,C}(I-(C-B)\Gamma_2 S_{\lambda,B})=S_{\lambda,B},
\end{equation}        
we note that $\Ran(S_{\lambda,B})=\Ran(S_{\lambda,C})$ whenever $\lambda\in\rho(A_B)\cap\rho(A_{C})$. \ian{Now assume $\lambda\in\rho(A_B)\cap\sigma(A_C)$.} We need to show that $\Ran(S_{\lambda,B})\subseteq \overline{ \Tc' }$ where 
\beq \Tc'= \Span_{\mu\in\rho(A_B)\cap\rho(A_C) }
        \mbox{Ran}(S_{\mu,B}).
\eeq
By assumption, there exists a sequence $(\lambda_n)_{n\in\N}$ in $\rho(A_B)\cap\rho(A_C)$ with $\lambda_n\to\lambda$. Let $u=S_{\lambda,B}f$. We have
$$S_{\lambda_n,B}-S_{\lambda,B}=(\lambda_n-\lambda)(A_B-\lambda_n)^{-1}S_{\lambda,B}.$$
Therefore,
$$\norm{S_{\lambda_n,B}f-S_{\lambda,B}f}\leq|\lambda_n-\lambda|\norm{(A_B-\lambda_n)^{-1}}\norm{S_{\lambda,B}f}.$$
As $n\to\infty$, $\norm{(A_B-\lambda_n)^{-1}}\to\norm{(A_B-\lambda)^{-1}}<\infty$, so $S_{\lambda_n,B}f\to S_{\lambda,B}f$ which completes the proof.
\item This follows immediately from the previous part of the proposition.
\end{enumerate}
\end{proof}

\begin{remark}
We note that the conditions in parts \ref{boundedres} and \ref{Bindependence} of the proposition are satisfied in many interesting cases, in particular in the case of `weak' perturbations of selfadjoint operators.

Throughout the remainder of this article, we will assume that the spaces $\overline{\Sc_B}$ and $\overline{\Tc_B}$ coincide, are independent of $B$ and equal $\overline{\Span_{\lambda\in\Lambda} \ker (\At^*-\lambda)}$. To avoid cumbersome notation, we shall denote all these spaces by $\Sco$. We shall also denote $\Sc_B$ by $\Sc$ and $\Tc_B$ by $\Tc$ when no confusion can arise. 
\change{We will generally refer to $\Sco$ as the detectable subspace.}

In \cite[Lemma 3.4]{BHMNW09}, it is shown that
 $\overline{\Sc}$ is a regular invariant space of the resolvent of the operator $A_B$: that is, 
$\overline{(A_B-\mu I)^{-1}\overline{\Sc}} = \overline{\Sc}$ for all $\mu\in \rho(A_B)$.

\end{remark}

From \ian{\eqref{eq:mm1b}} and Proposition \ref{prop:2.7}, part \ref{Bindependence}, we get \be\label{eq:Skernel}\Sc^\perp=\bigcap_{B,\lambda\in\rho(A_B)}\ker(S_{\lambda,B}^*).\ee
%and under the assumption that $\Ran(\Gamma_1-B\Gamma_2)=\Hc$ (which is satisfied for most interesting examples), 
\ian{Moreover, from \cite[Proposition 3.9]{BMNW08}} we have $$\ker(S_{\lambda,B}^*)=\ker\left(\Gammat_2(\At_{B^*}-\overline{\lambda})^{-1}\right).$$

We now assume that $h\in\Sc^\perp$. Then  we have $\Gammat_2(\At_{B^*}-\overline{\lambda})^{-1}h=0$ for all suitable $B$ and $\lambda$. Fixing $B$ and $\lambda$ and
setting \be\label{eq:yb} y_B=(\At_{B^*}-\overline{\lambda})^{-1}h, \ee 
we get $\Gammat_2 y_B=0$ and hence $\Gammat_1 y_B= B^*\Gammat_2 y_B=0$, so $y_B$ satisfies any homogeneous boundary condition and lies in the domain of the minimal operator. 

Hence,
\be\label{observable}
\Sc^\perp=\{ h\in H: \forall B^*, \lambda\in \rho(\At_{B^*}), \quad \Gammat_i (\At_{B^*}-\lambda)^{-1}h=0, \hbox{ for } i=1,2\}.
\ee

\begin{remark}\label{rem:sys}
Determining the detectable subspace $\Sc$ is closely related to the problem of observability in systems theory.  Indeed, from \eqref{eq:Skernel} and \eqref{observable} the space $\Sc^\perp$ (at least formally) coincides which the `non-observable for all time' subspace $N(\Theta)$ of a system $\Theta=(A,B,C,D)$ (see \cite{Sta05}), in which
$A=\At_{B^*}$, $B=\Sct_{\lambda_0,B^*}$ for some $\lambda_0 \in \rho (\At_{B^*})$,
\ian{$ C=\Gammat_2  (\At^*-\lambda_0)$}, $ D=0$,
though there are several differences between the notions:
\begin{enumerate}
\item
The corresponding system can be highly awkward to construct and requires involving unbounded operators.
\item
In systems theory the subspace of un-observable states is generated by the resolvent in one half plane only (corresponding to positive times $t$ only).  In our construction, the spectral parameter runs through the  whole resolvent set.  It is well known that when the resolvent set consists of several unconnected domains, developing the linear set by the resolvent essentially depends on the choice of the component.
\item
We do not require the operator \ian{$\At_{B^*}$} to be  the generator of a semigroup. In particular, the resolvent set may have a complicated geometrical structure. If \ian{$\At_{B^*}$} is a generator, the resolvent in \eqref{observable} can be replaced by the positive and negative time semigroups.
\end{enumerate}
Despite these differences, the similarity  to the observability problem may \changes{be fruitful} for analysing detectability, both in general and in particular examples. For more connections between boundary triples and systems theory, we refer to \cite{Ryz09}.
\end{remark}

\section{Example operators}\label{section:2}

In this section we shall examine three different concrete operators which will be used in the following to illustrate the power and also the limitations of the theory. For the first of these we show that $\overline{\mathcal S}$
is the whole underlying Hilbert space; for the second example we refer the reader to some previous work, where we show that $\overline{\mathcal S}$
may or may not be the whole space; for the third example, we calculate the function $M_B(\lambda)$, in preparation for the substantial work in
Sections \ref{section:7}, \ref{section:Toep} and the Appendix, which shows that the characterization of $\overline{\mathcal S}$ may be very subtle for this
seemingly innocuous model.

%For simplicity, we initially consider the symmetric case, i.e.~we need to consider functions in the kernel of $\Gamma_2(A_{B}-\lambda)^{-1}$ for any boundary condition $B$ and any $\lambda\in\rho(A_B)$. We first consider the case $w\equiv 0$, i.e.~the case of a Schr\"{o}dinger problem.

\subsection{\change{Schr\"{o}dinger} problems}\label{sec:SL}
%We first consider the case $w=\wt\equiv 0$, i.e.~the case of a Schr\"{o}dinger problem.
%Consider the problem
%\be
%\left(-\frac{d^2}{dx^2}+q\right) u=\lambda u+f \quad\hbox{on}\quad (0,1)
%\ee
For complex valued $q\in L^\infty(0,1)$, consider 
\be
L u = \left ( -\frac{d^2}{dx^2} + q \right ) u \hbox{ and } \widetilde{L} u = \left ( -\frac{d^2}{dx^2} + \overline{q} \right ) u \;\;\;\;\mbox{on}\;[0,1].
\ee
Let $Au=Lu$ and $\At u = \widetilde{L} u$ with $D(A)=D(\At)=H^2_0(0,1)$. Then $\At^*u= Lu$ and  $A^* u = \widetilde{L} u$ with $D(\At^*)=D(A^*)=H^2(0,1)$ and
for $u,v\in H^2(0,1)$
\begin{eqnarray}
\Big\langle \At^*u,
 v\Big \rangle
 -  \Big\langle u,
  A^*v\Big\rangle\nonumber & = &\Big\langle \Gamma_1u,
 \Gamma_2v\Big\rangle
 - \Big\langle \Gamma_2u,
 \Gamma_1v\Big\rangle, 
\end{eqnarray}
where 
\[ \Gamma_1u
  = \left(\begin{array}{c} -u'(1) \\ u'(0) \end{array}\right),
\;\;\;
 \Gamma_2u
  = \left(\begin{array}{c} u(1) \\ u(0) \end{array}\right).
\]
In particular, $\Gamma_1=\Gammat_1$, $\Gamma_2=\Gammat_2$ and $\cH=\cK=\C^2$.

%Consider
% \be\label{eq:SL} \At u = \left ( -\frac{d^2}{dx^2} + q \right ) u = \lambda u \;\;\;\;\mbox{on}\;[0,1].
% \ee
 Let $\theta(x,\lambda)$ and $\phi(x,\lambda)$ be solutions of \change{$\widetilde{L}u=\overline{\lambda} u$} which satisfy
$
 \theta(0,\lambda)=0$, $ \theta'(0,\lambda)=1$
 and $ 
 \phi(0,\lambda)=1$, $ \phi'(0,\lambda)=0$.
 Let $y_B$ be as in  \eqref{eq:yb}. Then by the
  variation of constants formula, there exist $C,\tilde C$ such that
 \begin{eqnarray*}
 y_B(x,\lambda)&= & \int_0^x \phi(t,\lambda)h(t)\ dt\ \theta(x,\lambda) + 
 \int_x^1 \theta(t,\lambda)h(t)\ dt\ \phi(x,\lambda)  + C \theta(x,\lambda)+\tilde C \phi(x,\lambda).
 \end{eqnarray*}
%and
%$$ y_B'(x,\lambda)=\int_0^x \phi h \ dt\ \theta' + \int_x^1 \theta h\ dt\ \phi'+ C
%\theta' + \tilde C \phi',
%$$
%where $'$ denotes differentiation with respect to $x$.

$y_B$ satisfies $\Gamma_1 y_B=0=\Gamma_2 y_B$. We choose $\lambda$ so that it is not a Dirichlet eigenvalue. Then
\ian{\begin{eqnarray*}
& y_B(0,\lambda)= \int_0^1 \theta h dt + \tilde C  =0, \quad
y_B(1,\lambda)= \left (\int_0^1 \phi h dt +C  \right )
\theta(1,\lambda)+ \tilde C  \phi(1,\lambda) =0, &\\
& y_B'(0,\lambda)= C=0, \quad
y_B'(1,\lambda)= \left (\int_0^1 \phi h dt +C  \right )
\theta'(1,\lambda)+ \tilde C  \phi'(1,\lambda)=0.&
\end{eqnarray*}}
This simplifies to
\begin{eqnarray*}
&\int_0^1 \phi h dt \ \theta(1,\lambda)-\int_0^1 \theta h dt\  \phi(1,\lambda) =0, \\
&\int_0^1 \phi h dt\ \theta'(1,\lambda)-\int_0^1 \theta h dt\  \phi'(1,\lambda)=0.
\end{eqnarray*}
As the Wronskian of $\theta$ and $\phi$ is non-zero, we have
$$\int_0^1 \theta h dt = \int_0^1 \phi h dt =0.
$$
This holds for almost all $\lambda$. Analyticity in $\lambda$ implies that these \ian{equations} hold for all $\lambda$.
Choosing $\lambda$ to run through the  Dirichlet eigenvalues shows that $\bar{h}$ is orthogonal to all Dirichlet eigenfunctions \change{and also to any possible root vectors}.  Hence, $h \equiv 0$
and we have proved the following result which is consistent with the Borg Uniqueness Theorem \cite{Borg49,Marchenko}.

\begin{proposition} For the \change{Schr\"{o}dinger} operator we have ${\overline{\mathcal S}} = L^2(0,1)$.
\end{proposition}

\subsection{Hain-L\"{u}st-type operators}
Let 
\begin{equation}
\At^* = \left(\begin{array}{cc} -\frac{d^2}{dx^2}+q(x) & \wt(x) \vspace{2pt}\\
 w(x) & u(x) \end{array}\right) \quad\hbox{ and }\quad
 A^* = \left(\begin{array}{cc} -\frac{d^2}{dx^2}+\overline{q(x)} & \overline{w(x)} \vspace{2pt}\\ 
 \overline{\wt(x)} & \overline{u(x)} \end{array}\right),
%\label{eq:hl1} 
\end{equation}
where $q$, $u$, $\wt$ and $w$ are $L^\infty$-functions, and the domain of
the operators \ian{is} given by
\begin{equation}
D(\At^*)=D(A^*) = H^2(0,1)\times L^2(0,1). 
%\label{eq:hl2} 
\end{equation}

It is then easy to see that
\begin{eqnarray}
\llangle \At^*\left(\begin{array}{c} y \\ z \end{array}\right),
 \left(\begin{array}{c} f \\ g \end{array}\right)\rrangle
 -  \llangle \left(\begin{array}{c} y \\ z \end{array}\right),
  A^*\left(\begin{array}{c} f \\ g \end{array}\right)\rrangle\nonumber & & \\
 & \hspace{-9cm} = & \hspace{-4cm}\llangle \Gamma_1\left(\begin{array}{c} y \\ z \end{array}\right),
 \Gamma_2\left(\begin{array}{c} f \\ g \end{array}\right)\rrangle
 - \llangle \Gamma_2\left(\begin{array}{c} y \\ z \end{array}\right),
 \Gamma_1\left(\begin{array}{c} f \\ g \end{array}\right)\rrangle, 
\label{eq:hl4}
\end{eqnarray}
where 
\[ \cH=\cK=\C^2, \quad\Gamma_1\left(\begin{array}{c} y \\ z \end{array}\right)
  = \left(\begin{array}{c} -y'(1) \\ y'(0) \end{array}\right),
\;\;\;
 \Gamma_2\left(\begin{array}{c} y \\ z \end{array}\right)
  = \left(\begin{array}{c} y(1) \\ y(0) \end{array}\right).
\]
Some information on $\Sc$ for these operators is available in \cite{BHMNW09}. In particular we show there that
if $w = \tilde{w}$ then $\overline{\mathcal S} \subseteq L^2(0,1)\oplus L^2((w^{-1}(\{0\})^c)$ (where $\Omega^c$ denotes the complement of \ian{a} set $\Omega$) and so if $w$ vanishes on a set of
positive measure then ${\overline{\mathcal S}}$ is not the whole underlying space.

%
%\begin{remark}
%I've cut out all Hain-Lust results for now. They're in the file HL\_notes12. We'll need to decide which of them we want here.
%\end{remark}

%We wish to show that
%\be \overline{\Sc}=\overline{\Tc} =  \left(\begin{array}{c} L^2(0,1) \\
%                                                          \chi_{\{w\neq 0\}}L^2(0,1) \end{array}\right). \label{eq:hltc}
%\ee
%Since Theorem 5.3 of the last paper shows one inclusion,
%it only remains to show the inclusion
%\beq \overline{\Sc}\supseteq  \left(\begin{array}{c} L^2(0,1) \\
%                                                          \chi_{\{w\neq 0\}}L^2(0,1) \end{array}\right). 
%\eeq
%(Note that the result in the last paper only considers the case where both off-diagonal entries are equal which is why we reprove the result below  for the modified versions of the operator that are considered there.)

\subsection{The Friedrichs model}\label{Friedrichs}

We consider in $L^2(\Rr)$ the operator $A$ with domain 
\be D(A) = \left\{ f\in L^2(\Rr) \, \Big\vert \, xf(x)\in L^2(\Rr), \;\;\;
 \lim_{R\rightarrow\infty}\int_{-R}^{R}f(x)dx \;\; \mbox{exists and is zero}\right\}, \label{eq:1} \ee
given by the expression
\be (Af)(x) = x f(x) + \langle f,\phi\rangle \psi(x), \label{eq:2} \ee
where $\phi$, $\psi$ are in $L^2(\Rr)$. Observe that since the constant
function ${\mathbf 1}$ does not lie in $L^2(\Rr)$ the domain of $A$ is 
dense in $L^2(\Rr)$. 

We first collect some results from \cite{BHMNW09} where more details and proofs can be found:

\begin{lemma}\label{lemma:1}
The adjoint of $A$ is given on the domain
\be D(A^*) = \left\{ f\in L^2(\Rr) \, | \, \exists c_f\in \Cc : xf(x)-c_f{\mathbf 1} 
\in L^2(\Rr)\right\}, \label{eq:3} \ee
by the formula
\be A^*f = x f(x) - c_f{\mathbf 1} +\langle f,\psi\rangle\phi. \label{eq:4}\ee
\end{lemma}
\ian{Since $c_f=\lim_{R\to\infty}(2R)^{-1} \int_{-R}^R xf(x)\ dx$ is uniquely determined, 
we can define} trace operators $\Gamma_1$ and $\Gamma_2$ on 
$D(A^*)$ as follows:
%\be \Gamma_1 u = \int_{\Rr} (u(x) - c_u{\mathbf 1}\sign(x)(x^2+1)^{-1/2})dx,\;\;\;
% \Gamma_2 u = c_u. \label{eq:5}\ee
\be \Gamma_1 u = \lim_{R\to\infty} \int_{-R}^R u(x) dx,\;\;\;
 \Gamma_2 u = c_u. \label{eq:5}\ee
\ian{Note that $\Gamma_1 u= \int_{\Rr} (u(x) - c_u{\mathbf 1}\sign(x)(x^2+1)^{-1/2})dx$, which is the expression used in \cite{BHMNW09}.}

\begin{lemma}\label{lemma:2}
The operators $\Gamma_1$ and $\Gamma_2$ are bounded relative to $A^*$ and the following `Green's identity' holds:
\be \langle A^*f,g\rangle - \langle f,A^*g \rangle
 = \Gamma_1 f \overline{\Gamma_2 g} - \Gamma_2 f \overline{\Gamma_1 g}
 + \langle f,\psi\rangle \langle\phi,g\rangle
 - \langle f,\phi\rangle \langle \psi,g\rangle. \label{eq:6}
\ee
\end{lemma}

We  introduce an operator $\widetilde{A}$ in which the roles of $\phi$ and
$\psi$ are swapped:
\be D(\At) = \left\{ f\in L^2(\Rr) \, \Big\vert \, xf(x)\in L^2(\Rr), \;\;\;
 \lim_{R\rightarrow\infty}\int_{-R}^{R}f(x)dx = 0\right\}, \label{eq:1t} \ee
\be (\At f)(x) = x f(x) + \langle f,\psi\rangle \phi\ian{(x)}. \label{eq:2t} \ee
In view of Lemma \ref{lemma:1} we immediately see that
$D(\At^*) = D(A^*)$ and that
\be \ian{(\At^*f)(x)} = x f(x) - c_f{\mathbf 1} +\langle f,\phi\rangle\ian{\psi(x)}. \label{eq:4t}\ee
Thus $\At^*$ is an extension of $A$,  $A^*$ is an extension of $\At$, 
and the following result is easily proved.

\begin{lemma}\label{lemma:4}
\be A = \left. \At^*\right|_{\ker(\Gamma_1)\cap\ker(\Gamma_2)}; \;\;\;
 \At = \left. A^*\right|_{\ker(\Gamma_1)\cap\ker(\Gamma_2)}; \label{eq:5t}
\ee
moreover, the Green's formula (\ref{eq:6}) can be modified to
\be \langle A^*f,g\rangle - \langle f,\At^*g \rangle
 = \Gamma_1 f \overline{\Gamma_2 g} - \Gamma_2 f \overline{\Gamma_1 g}.
 \label{eq:6t}
\ee
\end{lemma}

We finish our review from \cite{BHMNW09} with the  $M$-function and the resolvent:

\begin{lemma} Suppose that $\Im\lambda\neq 0$. Then $f\in\ker(\At^*-\lambda)$ if
\be f\ian{(x)} = \Gamma_2f \left[ \frac{1}{x-\lambda}-\frac{\langle (t-\lambda)^{-1},\phi\rangle}{D(\lambda)}\frac{\psi\ian{(x)}}{x-\lambda}\right] .\label{eq:7}\ee
Here $D$ is the function
\be D(\lambda) = 1 + \int_{\mathbb R}\frac{1}{x-\lambda}\psi(x)\overline{\phi(x)}dx. \label{eq:Ddef} \ee
Moreover the Titchmarsh-Weyl coefficient $M_B(\lambda)$ is given by
\be M_B(\lambda) = \left[ \mbox{sign}(\Im\lambda) \pi i -\frac{\ian{\langle (t-\lambda)^{-1},\overline{\psi}\rangle
 \langle (t-\lambda)^{-1},\phi\rangle} }{D(\lambda)} - B\right]^{-1}. \label{eq:9}
\ee
For the resolvent, we have that $(A_B-\lambda)f=g$ if and only if
\be f(x) = \frac{g(x)}{x-\lambda}
                                 -\frac{1}{ D(\lambda)}\frac{\psi(x)}{x-\lambda}
 \left\langle \frac{g}{t-\lambda},\phi\right\rangle +
  c_f\left[\frac{1}{x-\lambda}-\frac{1}{ D(\lambda)}\frac{\psi(x)}{x-\lambda}
 \left\langle \frac{1}{t-\lambda},\phi\right\rangle\right], \label{eq:mm1c}
\ee
in which
the coefficient $c_f$ is given by
\be c_f = M_B(\lambda)\left[-\left\langle \frac{1}{t-\lambda},\overline{g}\right\rangle
 + \frac{1}{ D(\lambda)}\left\langle \frac{g}{t-\lambda},\phi\right\rangle
 \left\langle \frac{1}{t-\lambda},\overline{\psi}\right\rangle\right]. \label{eq:mm10}\ee
\end{lemma}
These calculations will be needed in Sections \ref{section:6}, \ref{section:7} and the Appendix.

\begin{remark}\label{Fourier}
There is another approach to the Friedrichs model via the Fourier transform which may appear much more natural. It is easy to check that, denoting the Fourier transform by $\cF$ and $\ian{\cF f=\hat{f}}$, we get
$$\cF A\cF^* = i\frac{d}{dx}+\llangle\cdot,\hat\phi\rrangle \hat\psi, \quad D(\cF A \cF^*)=\{ u\in H^1(\R): \ u(0)=0\},$$
$$\cF \change{\At^*} \cF^* = i\frac{d}{dx}+\llangle\cdot,\hat\phi\rrangle \hat\psi, \quad D(\cF \change{\At^*} \cF^*)=\{ u\in L^2(\R): \ u\vert_{\R^\pm}\in H^1(\R^\pm)\},$$
and
\ian{$$\cF A_B \cF^* = i\frac{d}{dx}+\llangle\cdot,\hat\phi\rrangle \hat\psi,$$  
$$D(\cF A_B \cF^*)=\left\{ u\in L^2(\R): \ u\vert_{\R^\pm}\in H^1(\R^\pm), u(0^+)=\frac{B-i\pi}{B+i\pi} u(0^-)\right\},$$}
where $u(0^\pm)$ denotes the limit of $u$ at zero from the left and right, respectively.
Moreover, $\Gamma_1 f = \sqrt{\pi/2}(\hat f(0^+)+\hat f(0^-))$ and $\Gamma_2 f = i(2\pi)^{-1}(\hat f(0^+)-\hat f(0^-))$. 
There are similar expressions for the adjoint operators and traces.
In terms of extension theory it is much easier to use this Fourier representation. However, for our later calculations, the original model is more useful, as it facilitates the calculation of residues.
\end{remark}

\section{Relation between $M$-function and resolvent on $\overline{\Sc}$}\label{section:3}
Having introduced some concrete examples in the previous sections, we now turn our attention to what can be shown in the general setting. \ian{Our aim is to study the relation between the function $M_B$ and the bordered resolvent $P_\Scto(A_B-\lambda)^{-1}\vert_{\overline{\Sc}}$ where for any subspace $M$, $P_M$ denotes the orthogonal projection onto $M$.}

\subsection{Information on the $M$-function contained in the resolvent.}

 We first look at gaining information on the $M$-function from knowledge of the resolvent.
        
\begin{theorem}\label{Munique}  Let $\lambda\in\rho(A_B)$. Then $P_\Scto(A_B-\lambda)^{-1}\vert_{\overline{\Sc}}$ uniquely determines $M_B(\lambda)$.

In particular, if also  $\lambda\in\rho(A_C)$, then
$P_\Scto(A_B-\lambda)^{-1}\vert_{\overline{\Sc}} = P_\Scto(A_C-\lambda)^{-1}\vert_{\overline{\Sc}}$ implies that $M_B(\lambda)=M_C(\lambda)$, \change{
  and, if additionally $\lambda\in\rho(A_\infty)$, then $B=C$. Here, $A_\infty=\At^*\vert_{\ker\Gamma_2}$. }
\end{theorem}        
        
\begin{proof}
%As $(A_B-\lambda)^{-1}\vert_{\overline{\Sc}}$ is known, the l.h.s.~of (\ref{eq:fund}) is known for any $F\in \Ran(S_{\mu,B})$ and $v\in D(A^*)$. Therefore, so is the r.h.s. Let $f\in\Ran(\Gamma_1-B\Gamma_2)$ and choose $F=S_{\mu,B}f$. Then $(M_B(\lambda)f,(\widetilde{\Gamma}_1-B^*\widetilde{\Gamma}_2)v)_{\K}$ is known. Letting $v$ vary through $D(A^*)$ and using that $\Ran(\Gammat_1 -B^*\Gammat_2)$ is dense in $\cK$, $M_B(\lambda)f$ is known. As $f$ was arbitrary, this uniquely determines $M_B(\lambda)$.
%Assume $\widehat{M}_B(\lambda)$ is a second $M$-function for the problem. Then there exist $f\in\Ran(\Gamma_1-B\Gamma_2)$ and $v\in D(A^*)$ such that
%$$ \llangle M_B(\lambda)f,(\widetilde{\Gamma}_1-B^*\widetilde{\Gamma}_2)v\rrangle_{\K}\neq \llangle\widehat{M}_B(\lambda)f,(\widetilde{\Gamma}_1-B^*\widetilde{\Gamma}_2)v\rrangle_{\K}. $$
%Then there exist $h\in\Sc$ (given by $S_{\mu,B}f$) and $\widetilde{h}\in \widetilde{\Sc}$ such that
%$$ \llangle (I-(A_B-\lambda)^{-1}(\mu-\lambda))h,\widetilde{h}\rrangle$$
%has two different values, yielding a contradiction.
Assume $\widehat{M}_B(\lambda)$ is a different $M$-function for the same problem. \ian{By surjectivity of the trace operators there exist $F\in D(\At^*)$ and $v\in D(A^*)$ such that}% using Lemma \ref{intbyparts}
$$ \llangle M_B(\lambda)(\Gamma_1-B\Gamma_2)F,(\widetilde{\Gamma}_1-B^*\widetilde{\Gamma}_2)v\rrangle_{\K}\neq \llangle\widehat{M}_B(\lambda)(\Gamma_1-B\Gamma_2)F,(\widetilde{\Gamma}_1-B^*\widetilde{\Gamma}_2)v\rrangle_{\K}. $$
Setting $h=S_{\mu,B}(\Gamma_1-B\Gamma_2)F\in\Sc$ and $\widetilde{h}=\widetilde{S}_{\widetilde{\mu},B^*}(\widetilde{\Gamma}_1-B^*\widetilde{\Gamma}_2)v\in\widetilde{\Sc}$ and using \eqref{eq:fund}, we find that
$$ \llangle (I-(A_B-\lambda)^{-1}(\mu-\lambda))h,(\widetilde{\mu}-\bar{\lambda})\widetilde{h}\rrangle$$
has two different values. Therefore, \ian{$\llangle(A_B-\lambda)^{-1}h,\widetilde{h}\rrangle$} has two different values yielding a contradiction.

If $P_\Scto(A_B-\lambda)^{-1}\vert_{\overline{\Sc}} = P_\Scto(A_C-\lambda)^{-1}\vert_{\overline{\Sc}}$, then \ian{by the argument above }
$$ \llangle M_B(\lambda)(\Gamma_1-B\Gamma_2)F,(\widetilde{\Gamma}_1-B^*\widetilde{\Gamma}_2)v\rrangle_{\K}= \llangle M_C(\lambda)(\Gamma_1-B\Gamma_2)F,(\widetilde{\Gamma}_1-B^*\widetilde{\Gamma}_2)v\rrangle_{\K}. $$
Choosing $F$ and $v$ such that $\Gamma_2 F=\Gammat_2 v = 0$ and $\Gamma_1 F$  and $\Gammat_1 v$ are arbitrary, %and applying \eqref{Mfn},
 we obtain 
$$ \llangle M_B(\lambda)\Gamma_1 F,\widetilde{\Gamma}_1 v\rrangle_{\K}= \llangle M_C(\lambda)\Gamma_1 F,\widetilde{\Gamma}_1v\rrangle_{\K}. $$
Hence, $M_B(\lambda)=M_C(\lambda)$. \change{If $\lambda\in\rho(A_\infty)$, then $\Ran M_B(\lambda)=\cK$ and $\ker M_C(\lambda)=\{0\}$, so from  \eqref{Mfn} we get $B=C$.}
%$$ \llangle M_B(\lambda)\Gamma_1 F,\widetilde{\Gamma}_1 v\rrangle_{\K}= \llangle (I+M_B(\lambda)(B-C))^{-1} M_B(\lambda)\Gamma_1 F,\widetilde{\Gamma}_1v\rrangle_{\K}, $$
%from which it follows that $B=C$ and then $M_B=M_C$.
\end{proof}

%\textbf{PROBLEM:} If $A_B$ is unknown and we only know $(A_B-\lambda)^{-1}\vert_{\overline{\Sc}}$, we cannot know $S_{\mu,B}$ to find $h$ and get the contradiction.

Note that from the knowledge of the resolvent and the range of one solution operator we can explicitly reconstruct  $\overline{\Sc}$.
With some extra knowledge of the problem we can reconstruct $M_B(\lambda)$ from knowledge of the bordered resolvent on $\overline{\Sc}$.

\begin{theorem}\label{thm:Mreconst}
Assume we know $\overline{\Sc}$, $P_\Scto(A_B-\lambda)^{-1}\vert_{\overline{\Sc}}$ and $\overline{\Ran(S_{\mu,B})}$, $\overline{\Ran(\widetilde{S}_{\widetilde{\mu},B^*})}$ for some $(\mu,\widetilde{\mu})$ with $\mu,\overline{\widetilde{\mu}}\in\rho(A_B)$. 
 Then we can reconstruct $M_B(\lambda)$ uniquely if $B$ is known. 
\end{theorem}  

\begin{proof} %(Sketch - cf. notes, page 30):\\
For $h\in\overline{\Ran(S_{\mu,B})}$, $\widetilde{h}\in\overline{\Ran(\widetilde{S}_{\widetilde{\mu},B^*}})$, consider 
$$ H(h,\widetilde{h})=\llangle (I-(A_B-\lambda)^{-1}(\mu-\lambda))h,(\widetilde{\mu}-\overline{\lambda})\widetilde{h}\rrangle.$$ 
By assumption, we know $H(h,\widetilde{h})$.
\ian{Varying $h$ throughout $\overline{\Ran(S_{\mu,B})}$, we have that $(\Gamma_1-B\Gamma_2)h$ runs through the whole of $\cH$; varying $\widetilde{h}$ throughout $\overline{\Ran(\widetilde{S}_{\widetilde{\mu},B^*}})$, the values $(\Gammat_1-B^*\Gammat_2)\widetilde{h}$ run through the whole of $\cK$ and $\Gammat_2\widetilde{h}$ through the whole of $\cH$. 
 Using Lemma \ref{intbyparts} we have for a dense set of $h,\tilde{h}$ that
$$ H(h,\widetilde{h})=  \llangle M_B(\lambda)(\Gamma_1-B\Gamma_2)h,(\Gammat_1-B^*\Gammat_2)\widetilde{h}\rrangle_{\K} -\llangle (\Gamma_1-B\Gamma_2)h,\Gammat_2\widetilde{h}\rrangle_{\Hc}, $$
which allows reconstruction of the $M$-function.} %, since the ranges of $\Gamma_1-B\Gamma_2$ and $\Gammat_1-B^*\Gammat_2$ are dense.
\end{proof}

We look into the question of how strong the condition of knowledge of the closed ranges of solution operators needed in the theorem is. We first look at the Friedrichs model.

\begin{proposition}\label{prop:3.3}
Assume we know $\Ran S_{\lambda, B}$ and $\Ran \widetilde{S}_{\mu,B^*}$ for some $\lambda, \mu$ for the Friedrichs model. Moreover, assume $\Ran S_{\lambda, B}\neq\Span\{\frac{1}{x-\lambda}\}$ and $\Ran \widetilde{S}_{\mu,B^*}\neq \Span\{\frac{1}{x-\mu}\}$ (which is true for generic $\phi,\psi$). Then the operator is uniquely determined.
\end{proposition}

\begin{proof}
Assume we know $\Ran S_{\lambda, B}$ and $\Ran \widetilde{S}_{\mu,B^*}$ for some $\lambda, \mu$. Choosing the elements $u$ and $v$ in the ranges with $c_u=1$
and $c_v=1$ we have from \eqref{eq:7} that
 $$u(x)=\frac{1}{x-\lambda} - \frac{\langle (t-\lambda)^{-1},\phi\rangle}{D(\lambda)} \cdot \frac{\psi\ian{(x)}}{x-\lambda}\quad 
 \hbox{ and } \quad v(x)=\frac{1}{x-\mu} - \frac{\langle (t-\mu)^{-1},\psi\rangle}{\overline{D(\mu)}} \cdot \frac{\phi\ian{(x)}}{x-\mu}.$$
As one of the functions $\phi,\psi$ is only determined up to a scalar factor, we may normalize $\phi$ such that
$$ \frac{\langle (t-\lambda)^{-1},\phi\rangle}{D(\lambda)}=1.$$
This allows us to determine $\psi$ from the first expression, which also gives us $\langle (t-\mu)^{-1},\psi\rangle$, so $\phi\ian{(x)}/\overline{D(\mu)}$ is known. 
Solving $v(x)$ for $\phi\ian{(x)}$, we get 
\be\label{phi}\phi\ian{(x)}=(1-(x-\mu)v(x))\frac{\overline{D(\mu)}}{\langle (t-\mu)^{-1},\psi\rangle}.\ee
Multiplying by $\overline{(x-\lambda)^{-1}D(\lambda)^{-1}}$ and integrating over $x$, we can determine $D(\mu)D(\lambda)^{-1}$ from our normalisation of $\phi$.

Inserting our expression \eqref{phi} into $D(\lambda)=1+\langle(x-\lambda)^{-1},\overline{\psi}\phi\rangle$, we get a second equation relating $D(\lambda)$ and $D(\mu)$, allowing us to determine $D(\mu)$ and hence $\phi\ian{(x)}$. 
\end{proof}

\begin{remark}
Note that, if $\Ran S_{\lambda, B}=\Span\{\frac{1}{x-\lambda}\}$ or $\Ran \widetilde{S}_{\mu,B^*}= \Span\{\frac{1}{x-\mu}\}$ , it is clear from \ian{\eqref{eq:7} and} \eqref{eq:9} that the $M$-function in general does not contain sufficient information to recover $\phi$ and $\psi$.
\end{remark}

The following result shows that nothing like the result of Proposition \ref{prop:3.3}  holds for Hain-L\"ust operators and therefore no similar result
can hold in the abstract setting.
\begin{proposition}
Assume we know $\Ran S_{\lambda, B}$ and $\Ran \widetilde{S}_{\mu,B^*}$ for some $\lambda, \mu$ for the Hain-L\"ust operator
$$ \left(\begin{array}{cc} -\frac{d^2}{dx^2}+q(x) & w(x) \vspace{2pt}\\
 \widetilde{w}(x) & u(x) \end{array}\right).$$ Then the operator is not uniquely determined.
\end{proposition}

\begin{proof}
It is sufficient to show the claim for the case when the coefficients of the operator are real. Knowing the ranges of the solution operators corresponds to knowing \ian{kernels of the maximal operators and hence two linearly independent solutions each to both of the following equations:}% (this follows from using the Schur complement):
\begin{equation}
 \left(\begin{array}{cc} -\frac{d^2}{dx^2}+q(x) & w(x) \vspace{2pt}\\
 \widetilde{w}(x) & u(x) \end{array}\right) \left(\begin{array}{c} y_1 \\ z_1 \end{array}\right)= \lambda  \left(\begin{array}{c} y_1 \\ z_1 \end{array}\right)
\label{eq:hl1} 
\end{equation}
and 
\begin{equation}
 \left(\begin{array}{cc} -\frac{d^2}{dx^2}+q(x) & \widetilde{w}(x) \vspace{2pt}\\
 w(x) & u(x) \end{array}\right) \left(\begin{array}{c} y_2 \\ z_2 \end{array}\right)= \mu \left(\begin{array}{c} y_2 \\ z_2 \end{array}\right).
\label{eq:hl2} 
\end{equation}
Write the pairs of solutions as $(y_1,z_1)$, $(\widehat{y}_1,\widehat{z}_1)$, and  $(y_2,z_2)$, $(\widehat{y}_2,\widehat{z}_2)$, respectively.
Since $$ z_1=\frac{\widetilde{w}y_1}{u-\lambda},\ \widehat{z}_1=\frac{\widetilde{w}\widehat{y}_1}{u-\lambda},\ z_2=\frac{wy_2}{u-\mu},\ \widehat{z}_2=\frac{w\widehat{y}_2}{u-\mu}, $$
setting $\alpha=\frac{\widetilde{w}}{u-\lambda}$ and $\beta=\frac{w}{u-\mu}$, 
this can be written in the form
\begin{equation}
 \left(\begin{array}{cccc} y_1 & \alpha y_1 &0 &0 \vspace{2pt}\\  \widehat{y}_1 & \alpha \widehat{y}_1 &0 &0 \vspace{2pt}\\ y_2 & 0 &\beta y_2 &0 \vspace{2pt}\\ \widehat{y}_2 & 0 &\beta \widehat{y}_2 &0 \vspace{2pt}\\ 0 & 0 &y_1 &\alpha y_1 \vspace{2pt}\\ 0 & 0 &\widehat{y}_1 &\alpha \widehat{y}_1 \vspace{2pt}\\
 0& y_2&0  & \beta y_2 \vspace{2pt}\\
 0& \widehat{y}_2&0  & \beta \widehat{y}_2 \end{array}\right)
  \left(\begin{array}{c} q \vspace{2pt}\\ w\vspace{2pt}\\ \widetilde{w} \vspace{2pt}\\ u \end{array}\right)
 = \left(\begin{array}{c} \lambda y_1+y_1'' \vspace{2pt}\\ \lambda \widehat{y}_1+\widehat{y}_1'' \vspace{2pt}\\ \mu y_2+y_2''\vspace{2pt}\\ \mu \widehat{y}_2+\widehat{y}_2''\vspace{2pt}\\ \lambda z_1 \vspace{2pt}\\ \lambda \widehat{z}_1 \vspace{2pt}\\ \mu z_2 \vspace{2pt}\\ \mu \widehat{z}_2 \end{array}\right).
\label{eq:hl3} 
\end{equation}
A calculation shows that the matrix on the left hand side of the equation does not have full rank for any $\alpha$ and $\beta$, 
%\bea
%\left|\begin{array}{cccc} y_1 & z_1 &0 &0 \vspace{2pt}\\ y_2 & 0 &z_2 &0 \vspace{2pt}\\ 0 & 0 &y_1 &z_1 \vspace{2pt}\\
% 0& y_2&0  & z_2 \end{array}\right| &=&
% y_1 \left|\begin{array}{ccc}   0 &z_2 &0 \vspace{2pt}\\  0 &y_1 &z_1 \vspace{2pt}\\ y_2&0  & z_2 \end{array}\right|
% -z_1\left|\begin{array}{ccc} y_2  &z_2 &0 \vspace{2pt}\\ 0  &y_1 &z_1 \vspace{2pt}\\ 0&0  & z_2 \end{array}\right| \\
% &=& y_1y_2z_1z_2-z_1y_2y_1z_2 \ =\ 0,
%\eea
so the system is not uniquely solvable.
\end{proof}

\subsection{Reconstruction from two bordered resolvents}
Although we are primarily concerned with inverse problems, it is still interesting to consider some forward problems with
partial data arising from the restriction of the resolvent operators to the detectable subspace.

\begin{theorem}\label{thm:two}
Assume $P_{\overline{\tilde{\Sc}}}(A_B-\lambda)^{-1}\vert_{\overline{\Sc}}$ and $P_{\overline{\tilde{\Sc}}}(A_C-\lambda)^{-1}\vert_{\overline{\Sc}}$ are known. In addition, \ian{assume that \\
(i) $\mbox{$\Gamma_2(A_C-\lambda)^{-1}\overline{\Sc}$}$ and $\Gammat_2(A_C-\lambda)^{-*}\overline{\Sct}$ are known, \\
(ii)  $\Gamma_2(A_C-\lambda)^{-1}\overline{\Sc}$ is dense in $\cH$ and $\Gammat_2 (A_C-\lambda)^{-*}\overline{\Sct}$ is dense in $\cK$, \\
(iii)  $\Ran(B-C)$ is dense in $\cH$ and $\ker(B-C)=\{0\}$.}
%
%\begin{itemize}
%\item  $\Gamma_2(A_C-\lambda)^{-1}\overline{\Sc}$ is known,
%	\item $\Gammat_2(A_C-\lambda)^{-*}\vert_{\overline{\Sct}}$ is known,
%	\item $\Ran(B-C)$ is dense in $\cH$ and $\ker(B-C)=\{0\}$,
%		\item $\Gamma_2(A_C-\lambda)^{-1}\overline{\Sc}$ is dense in the boundary space,
%		\item $\Gammat_2 (A_C-\lambda)^{-*}\overline{\Sct}$ is dense in the boundary space.
%\end{itemize}
Then $M_B(\lambda)$ can be recovered.
\end{theorem}

\begin{proof}
Let $\lambda\in\rho(A_B)\cap\rho(A_C)$. Then the Krein formula \ian{\eqref{Krein}} gives
\begin{eqnarray*}%\label{Krein}
(A_B-\lambda)^{-1} - (A_C-\lambda)^{-1}=S_{\lambda,C}(I+(B-C)M_B(\lambda))(B-C)\Gamma_2(A_C-\lambda)^{-1}.
\end{eqnarray*}
Now let $f\in\overline{\Sct}$ and $g\in \overline{\Sc}$. Then we know 
$$ \llangle f, (A_B-\lambda)^{-1}g\rrangle -  \llangle f, (A_C-\lambda)^{-1}g\rrangle  . $$
Using  \ian{\eqref{Krein}}, we obtain
\bea &&
\llangle f, (A_B-\lambda)^{-1}g\rrangle -  \llangle f, (A_C-\lambda)^{-1}g\rrangle\\
  &&\hspace{50pt} =
  \llangle f,S_{\lambda,C}(I+(B-C)M_B(\lambda))(B-C)\Gamma_2(A_C-\lambda)^{-1} g\rrangle \\
  &&\hspace{50pt}= 
  \llangle \Gammat_2 (A_C-\lambda)^{-*}f,(I+(B-C)M_B(\lambda))(B-C)\Gamma_2(A_C-\lambda)^{-1}g \rrangle \\
   &&\hspace{50pt}= 
  \llangle \Gammat_2 (A_C-\lambda)^{-*}f,(B-C)\Gamma_2(A_C-\lambda)^{-1}g \rrangle + \\
  & & \hspace{65pt} \llangle (B-C)^*\Gammat_2 (A_C-\lambda)^{-*}f,M_B(\lambda)(B-C)\Gamma_2(A_C-\lambda)^{-1}g \rrangle .
\eea
%
%\bea
%\llangle f, (A_B-\lambda)^{-1}g\rrangle -  \llangle f, (A_C-\lambda)^{-1}g\rrangle &=& 
  %\llangle f,S_{\lambda,C}(I+(B-C)M_B(\lambda))(B-C)\Gamma_2(A_C-\lambda)^{-1} g\rrangle \\
  %&=& 
  %\llangle \Gammat_2 (A_C-\lambda)^{-*}f,(I+(B-C)M_B(\lambda))(B-C)\Gamma_2(A_C-\lambda)^{-1}g \rrangle \\
   %&=& 
  %\llangle \Gammat_2 (A_C-\lambda)^{-*}f,(B-C)\Gamma_2(A_C-\lambda)^{-1}g \rrangle + \\
  %& & \hspace{10pt} \llangle (B-C)^*\Gammat_2 (A_C-\lambda)^{-*}f,M_B(\lambda)(B-C)\Gamma_2(A_C-\lambda)^{-1}g \rrangle .
%\eea
Our assumptions now allow us to recover $M_B(\lambda)$.
\end{proof}

\begin{remark} 
Alternatively, knowing the \ian{projection to $\Scto$ of the} derivative of the resolvent w.r.t.~the boundary condition and \ian{one resolvent restricted to $\Sco$ will } suffice, as, for $C=B+\eps D$ we have \ian{from \eqref{Krein}}
$$ \frac{(A_B-\lambda)^{-1} - (A_C-\lambda)^{-1}}{\eps}\to S_{\lambda,B}D\Gamma_2(A_B-\lambda)^{-1} \quad\ian{\hbox{ as } \eps\to 0}.$$
%(Since $S_{\lambda,D}(I+(B-D)M_B(\lambda))=S_{\lambda,B}$.)
If we now have the assumption of density of $D\Gamma_2(A_B-\lambda)^{-1}\overline{\Sc}$, then \ian{for $f\in\Sco$, $g\in\Scto$}, since 
$$\Big\langle S_{\lambda,B}D\Gamma_2(A_B-\lambda)^{-1}f,g\Big\rangle = \Big\langle D\Gamma_2(A_B-\lambda)^{-1}f,\Gammat_2(A_B-\lambda)^{-*}g\Big\rangle ,$$
and $\langle S_{\lambda,B}D\Gamma_2(A_B-\lambda)^{-1}f,g\rangle$ is known from the derivative, while $D\Gamma_2(A_B-\lambda)^{-1}f$ is known from the \ian{restricted} resolvent, 
knowing the \ian{projection of the derivative to $\Scto$} corresponds to knowing  $\Gammat_2(A_B-\lambda)^{-*}\vert_{\ian{\overline{\Sct}}}$.
\end{remark}

\subsection{Information on the resolvent from the $M$-function}

The following result gives some insight to the inverse problem of reconstructing $A_B$ from $M_B(\lambda)$. From examples later, we will see that knowledge of the $M$-function does not allow reconstruction of the bordered resolvent (see Remark \ref{rem:Borg}).

\begin{theorem}\label{thm:recon}
Assume we know $M_B(\lambda)$ for all $\lambda\in \rho(A_B)$, $\Ran(S_{\mu,B})$ for all $\mu\in\Lambda$  and $\Ran(\widetilde{S}_{\widetilde{\mu},B^*})$  for all $\widetilde{\mu}\in\widetilde{\Lambda}$, where $\Lambda\subseteq \rho(A_B)$ and $\widetilde{\Lambda}\subseteq\rho(A_B^{\; *})$ are dense subsets. Then we can reconstruct $P_{\overline{\widetilde{\Sc}}}(A_B-\lambda)^{-1}P_{\overline{\Sc}}$ for all $\lambda\in\rho(A_B)$.
\end{theorem}

%\begin{remark}
%The projection $P_{\overline{\widetilde{\Sc}}}$ is necessary. It is impossible to reconstruct $(A_B-\lambda)^{-1}P_{\overline{\Sc}}$ from the given information (see Sergey's first order Hain-L\"ust example on 01/04/11).
%\end{remark}

\begin{proof} Let $\mu\in\Lambda$ and $\widetilde{\mu}\in\widetilde{\Lambda}$. 
Consider \eqref{eq:fund} for any $F\in\Ran(S_{\mu,B})$ and $v\in \Ran(\widetilde{S}_{\widetilde{\mu},B^*})$. Then 
\bea
\llangle F-(A_B-\lambda)^{-1}(\At^*-\lambda)F,(A^*-\overline{\lambda}I)v\rrangle
&=&  \Big\langle F-(A_B-\lambda)^{-1}(\mu-\lambda)F,(\widetilde{\mu}-\overline{\lambda}I)v\Big\rangle \\
&\hspace{-200pt}=& \hspace{-100pt} -\llangle (\Gamma_1-B\Gamma_2)F,\Gammat_2v\rrangle_{\Hc}
 + \llangle M_B(\lambda)(\Gamma_1-B\Gamma_2)F,(\Gammat_1-B^*\Gammat_2)v\rrangle_{\K}.
\eea
We know the r.h.s.~of this equation for any $F\in\Ran(S_{\mu,B})$, $v\in \Ran(\widetilde{S}_{\widetilde{\mu},B^*})$ and $\lambda\in\rho(A_B)$, so we know 
$$\llangle F-(A_B-\lambda)^{-1}(\mu-\lambda)F,(\widetilde{\mu}-\overline{\lambda}I)v\rrangle .$$
Choosing $\lambda\neq\mu$ and $\lambda\neq\overline{\widetilde{\mu}}$, we know 
$$ \Big\langle (A_B-\lambda)^{-1}F,v\Big\rangle =  \llangle P_{\overline{\widetilde{\Sc}}}(A_B-\lambda)^{-1}P_{\overline{\Sc}} F,v\rrangle $$ for any $F\in\Ran(S_{\mu,B})$, $v\in \Ran(\widetilde{S}_{\widetilde{\mu},B^*})$ and $\lambda\in\rho(A_B)\setminus\left(\{\mu\}\cup\{\overline{\widetilde{\mu}}\}\right).$
By continuity, we know it for all $\lambda\in\rho(A_B)$.

Since $\Span\{\Ran(S_{\mu,B}): \mu\in\Lambda\}$ is dense in $\Sco$ and $\Span\{\Ran(\widetilde{S}_{\widetilde{\mu},B^*}):\widetilde\mu\in\widetilde\Lambda  \}$ is dense in $\Scto$, using boundedness of $P_{\overline{\widetilde{\Sc}}}(A_B-\lambda)^{-1}P_{\overline{\Sc}}$ gives the result.

%Repeating the argument for $\lambda'\in\rho(A_B)$, $\lambda'\neq\lambda$, and using the resolvent identity, we know
%$$(\lambda-\lambda')\llangle (A_B-\lambda)^{-1}(A_B-\lambda')^{-1}F,v\rrangle.$$
%Again by continuity, we may assume we know
%$$ \llangle (A_B-\lambda)^{-1}(A_B-\lambda')^{-1}F,v\rrangle$$
%for any $\lambda, \lambda'\in\rho(A_B)$.
%
%By the same argument for $\lambda^{\prime\prime}$ instead of $\lambda$ and again using the resolvent identity and a continuity argument, we know
%$$ \llangle (A_B-\lambda^{\prime\prime})^{-1}(A_B-\lambda)^{-1}(A_B-\lambda')^{-1}F,v\rrangle$$ which equals $$ \llangle(A_B-\lambda)^{-1}(A_B-\lambda')^{-1}F, (A_B^{\; *} -\overline{\lambda^{\prime\prime}})^{-1}v\rrangle$$
%for all $\lambda,\lambda^{\prime},\lambda^{\prime\prime}\in\rho(A_B)$.

%For $F_i\in\Ran(S_{\mu_i,B})$ and $v_j\in \Ran(\widetilde{S}_{\widetilde{\mu}_j,B^*})$, $\mu_i\in\Lambda$, $\widetilde{\mu}_j\in\widetilde{\Lambda}$, we now take linear combinations 
%$$ \sum_{i=1}^n c_i F_i \hbox{ and } \sum_{j=1}^n \widetilde{c_j} v_j.
%$$
% Since 
%$\Sco=\overline{\Span\{ (A_B-\lambda_i')^{-1} F_i \} }$ and $\Scto=\overline{\Span\{ (A_B^{\; *} -\overline{\lambda_j^{\prime\prime}})^{-1} v_j \} }$, by continuity arguments, we know
%$$\llangle (A_B-\lambda)^{-1}F,v\rrangle $$
%for all $\lambda\in\rho(A_B)$, $F\in\overline{\Sc}$ and $v\in\overline{\Sct}$,
%which means that $P_{\overline{\widetilde{\Sc}}}(A_B-\lambda)^{-1}P_{\overline{\Sc}} $ is known for all $\lambda\in\rho(A_B)$.
\end{proof}

\section{Analytic continuation}\label{section:4}
In preparation for the discussion in Section \ref{section:5} of jumps in $M_B(\lambda)$ and of the bordered resolvent across
the essential spectrum, we now discuss the relationship between the analytic continuation of $M_B(\lambda)$ and analytic continuation 
of the bordered resolvent of $A_B$, both initially defined on the resolvent set of $A_B$.

\begin{theorem}
%Assume $\Ran(\Gammat_1 -B^*\Gammat_2)=\cK$.
Let $\mu,\overline{\widetilde{\mu}}\in\rho(A_B)$. Assume that for any $F\in\Ran(S_{\mu,B})$, $v\in\Ran(\widetilde{S}_{\widetilde{\mu},B^*})$,  $$\left\langle (A_B-\lambda)^{-1}\vert_{\overline{\Sc}} F,v\right\rangle$$
admits an analytic continuation to some region $D$ of the complex plane (possibly on a different Riemann sheet). Then $M_B(\cdot)$ admits an analytic continuation to the same region $D$.
\end{theorem}

\begin{proof}
Given \ian{$f\in\cH$ and $\widetilde{f}\in\cK$}, choose $F=S_{\mu,B}f$ and $v=\widetilde{S}_{\widetilde{\mu},B^*}\widetilde{f}$. Then (\ref{eq:fund}) becomes
\beq \Big\langle F-(\mu-\lambda)(A_B-\lambda)^{-1}F,(\widetilde{\mu}-\overline{\lambda})v\Big\rangle
 = -\llangle f,\Gammat_2v\rrangle_{\Hc}+ \llangle M_B(\lambda)f,\widetilde{f}\rrangle_{\K}
 \eeq
 and the l.h.s.~admits analytic continuation, so the r.h.s.~does as well. 
\end{proof}

%\begin{remark}
%If $\Ran(\Gammat_1 -B^*\Gammat_2)$ is not the whole space, we will only get that the map  $\lambda\mapsto (M_B(\lambda)f,\widetilde{f})_{\K}$ admits an analytic continuation.
%\end{remark}

\begin{lemma}\label{dS}
For $\mu\in\rho(A_B)$,
$$\left(\frac{d}{d\lambda}S_{\cdot,B}\right)(\mu)=(A_B-\mu)^{-1}S_{\mu,B}.$$
\end{lemma}

\begin{proof}
From \eqref{eq:sdiff}, we have
% We have
%$$ S_{\lambda,B} f=  S_{\mu,B}f + (\lambda-\mu)(A_B-\lambda)^{-1}S_{\mu,B}f,$$
%hence,
$$\frac{S_{\lambda,B} f-  S_{\mu,B}f }{\lambda-\mu}=(A_B-\lambda)^{-1}S_{\mu,B}f,$$
which immediately proves the result.
\end{proof}

\begin{theorem} \label{analytic}
\ian{
Assume $M_B(\cdot)$ admits an analytic continuation to some region $D$ of the complex plane (possibly on a different Riemann sheet). Let $\mu,\overline{\widetilde{\mu}}\in\rho(A_B)$. Then
$$\left\langle (A_B-\lambda)^{-1}\vert_{\overline{\Sc}} F,v\right\rangle$$
admits an analytic continuation to the same region $D$ for any $F\in\Ran(S_{\mu,B})$, $v\in\Ran(\widetilde{S}_{\widetilde{\mu},B^*})$, apart from possible simple poles at $\mu$  and $\overline{\widetilde{\mu}}$. If $\mu=\overline{\widetilde{\mu}}$, a pole of order $2$ is possible at this point.}
\end{theorem}

\begin{proof}
Let $F\in\Ran(S_{\mu,B})$, $v\in\Ran(\widetilde{S}_{\widetilde{\mu},B^*})$. By assumption the r.h.s.~of (\ref{eq:fund}) admits analytic continuation, hence so does the l.h.s., given by
$$\llangle F-(\mu-\lambda)(A_B-\lambda)^{-1}F,(\widetilde{\mu}-\overline{\lambda})v\rrangle.$$
\ian{Since $\llangle F,(\widetilde{\mu}-\overline{\lambda})v\rrangle$ is clearly analytic, we have that
$$(\mu-\lambda) (\overline{\widetilde{\mu}}-{\lambda}) \llangle (A_B-\lambda)^{-1}F,v\rrangle$$
is analytic which gives the desired result.}
\end{proof}

\begin{remark}\ian{
%The proof shows that we
We can extend the set of those vectors} for which $\left\langle (A_B-\lambda)^{-1}\vert_{\overline{\Sc}} F,v\right\rangle$ admits analytic continuation by developing vectors on both sides by taking linear combinations and using the resolvents of $A_B$ and $\widetilde{A}_{B^*}$ respectively. However, we should not expect the result to extend to the whole of $\Sco$ (or $\Scto$) and therefore the bordered resolvent will not \ian{necessarily} admit analytic continuation.
\end{remark}

\ian{It is interesting to note that poles of $\left\langle (A_B-\lambda)^{-1}\vert_{\overline{\Sc}} F,v\right\rangle$ at $\mu$ and  $\overline{\widetilde{\mu}}$ do arise in concrete examples, though they may sometimes be   cancelled by other terms. 

\begin{example}
Let $\mu \in \C^-, \widetilde{\mu} \in \C^+$.
Consider an example of the Friedrichs model from Section \ref{Friedrichs},  where $\phi \in H^-_2$, $\psi \in H^+_2$ are rational functions with poles in suitable half-planes such that
\change{$\psi(\lambda)\overline{\phi(\overline\lambda)}$} does not have poles at $\mu$ or  $\overline{\tilde \mu}$.  Then  
$$F_\mu:=\dfrac1{x-\mu}\in H_2^+\cap  \ker ( \tilde A^* -\mu) \quad\hbox{ and }\quad   v_{\tilde \mu}:=\dfrac1{x-\tilde \mu}\in H_2^-\cap   \ker  (A^*-\tilde \mu).$$
 We consider the analytic continuation of the   functions $M_B(\cdot)$ and   $\llangle (A_B-\cdot)^{-1} F_\mu, v_{ \tilde \mu}\rrangle$  from the upper to the lower half-plane.

From \eqref{eq:fund}, we get   for $\lambda \in \C^+$
\ben
(\lambda-\mu) (\lambda-\overline{ \tilde \mu}) 
\llangle (A_B-\lambda)^{-1} F_\mu, v_{ \tilde \mu}\rrangle \label{*} 
& =&  -\llangle  M_B(\lambda) (\Gamma_1- B \Gamma_2) F_\mu, (\tilde \Gamma_1-B^*\tilde \Gamma_2) v_{\tilde \mu }\rrangle  \\
&& \quad +  \llangle (\Gamma_1-B\Gamma_2 )F_\mu, \tilde \Gamma_2 v_{\tilde \mu})\rrangle + \llangle F_\mu, (\tilde\mu-\overline\lambda)v_{\tilde \mu}\rrangle. \nonumber
\een
Now, $$\llangle F_\mu, v_{\tilde \mu}\rrangle = \llangle \dfrac1{x-\mu},\dfrac1{x-\tilde \mu}\rrangle =0,$$
$$ (\Gamma_1-B\Gamma_2)F_\mu=(\Gamma_1-B\Gamma_2)\dfrac1{x-\mu}=-\pi i -B$$
$$(\tilde \Gamma_1-B^*\tilde   \Gamma_2)v_{\tilde \mu}=\pi i -B^* \quad \hbox{ and }\quad\tilde \Gamma_2 v_\mu=1.$$
Thus  (\ref{*})  gives 
\ben    (\lambda-\mu) (\lambda-\overline{ \tilde \mu}) 
\llangle (A_B-\lambda)^{-1} F_\mu, v_{ \tilde \mu}\rrangle &=& - M_B(\lambda) (-\pi i-B)(-\pi i -B) +(-\pi i -B) \nonumber \\
&= &-(\pi i + B) ( M_B(\lambda) (\pi i + B) +1),  \nonumber
\een
or 
\ben    
\llangle (A_B-\lambda)^{-1} F_\mu, v_{ \tilde \mu}\rrangle 
&= &-\frac{(\pi i + B) ( M_B(\lambda) (\pi i + B) +1)}{(\lambda-\mu) (\lambda-\overline{ \tilde \mu}) }\nonumber  \\
&= & 
\dfrac{-2\pi i    [  \dfrac{  \pi   i -B}{  \pi i + B}- 2\pi i     \psi(\lambda)\overline{\phi(\overline \lambda)}   ]  ^{-1}}{(\lambda-\mu) (\lambda-\overline{ \tilde \mu})}.  \label{**}
\een

From  \eqref{eq:9} we get 
\be M_B(\lambda) = \left[   \pi i +\frac{4\pi^2 \psi(\lambda)\overline{\phi(\overline \lambda)} }{1+2\pi i     \psi(\lambda)\overline{\phi(\overline \lambda)}  } - B\right]^{-1}
 = \left[   -\pi i   - B   +\frac{2\pi  i    }{1+2\pi i     \psi(\lambda)\overline{\phi(\overline \lambda)}  } \right]^{-1}.
 \ee
 
This $M_B(\cdot)$ admits an analytic continuation to the lower half-plane while the analytic continuation of   $\llangle (A_B-\cdot)^{-1} F_\mu, v_{ \tilde \mu}\rrangle$  given by \eqref{**} has poles at $\mu$ and $\overline {\tilde \mu}$.

In the case when $B\neq -\pi i $, to cancel the poles in \eqref{**} we need to chose poles of the analytic continuation of $\psi(\lambda)\overline{  \phi(\overline \lambda)}$ to lie at $\mu$ and $\overline{\tilde \mu}$.  Note that the poles appearing in Theorem \ref{analytic} should not be confused with resonances (poles of the analytic continuation of $M_B$).  Here  the resonances are due to zeroes of 
$     \dfrac{  \pi   i -B}{  \pi i + B}- 2\pi i     \psi(\lambda)\overline{\phi(\overline \lambda)}    $  in formula \eqref{**}.
 \end{example}
}

\section{Abstract theory: relation between jumps of $M_B$ and  bordered resolvent}\label{section:5}
We consider the case in which $A_B$ and $A_B^*$ have essential spectrum lying on the real axis; we wish to examine
in what sense $M_B(\lambda)$ jumps across the real axis, and how this may be related to a jump in the resolvent
$(A_B-\lambda)^{-1}$.

\begin{assumption}\label{assumption:new1}
We assume that there exist countable families $\{f_i\}_{i\in I}$, $\{w_j\}_{j\in\tilde{I}}$ in $\cH$ and
$\cK$ respectively, whose closed linear spans are $\cH$ and $\cK$, and that for these families the inner products
$ \llangle M_B(\lambda)f_i,w_j\rrangle $ lie in both the Nevanlinna classes $N(\mathbb{C}_{\pm})$. This implies that
they have non-tangential boundary values $\llangle M_B(k\pm i0)f_i,w_j\rrangle $ for a.e. $k\in \mathbb R$.
The class $N(\mathbb{C}_{\pm})$ consists of all meromorphic functions on $\C_\pm$ which can be represented as the quotient of two bounded analytic functions in the corresponding half-plane (see \cite{Koosis2}).
\end{assumption}

Our first result is that this assumption is equivalent to an assumption on the resolvent.
\begin{lemma}\label{lemma:extra1}
The functions $\llangle M_B(\lambda)f_i,w_j\rrangle $ lie in $N(\mathbb{C}_{\pm})$ if and only if, for every
$\mu\not\in \sigma(A_B)$ and $\tilde{\mu}\not\in \ian{\sigma(\At_{B^*})}$, the functions $\llangle (A_B-\lambda)^{-1}S_{\mu,B}f_i,
\tilde{S}_{\tilde{\mu},B^*}w_j\rrangle$ lie in $N(\mathbb{C}_{\pm})$.
\end{lemma}

\begin{proof} Starting with the fundamental identity
\[ \llangle (\At^*-\lambda)u,v\rrangle -  \llangle u, (\overset{}{A^*}-\overline{\lambda})v   \rrangle
 = \llangle \Gamma_1 u, \Gammat_2v\rrangle - \llangle \Gamma_2 u, \Gammat_1 v \rrangle \]
and making the choices $u = (A_B-\lambda)^{-1}S_{\mu,B}f_i$, $v = \tilde{S}_{\tilde{\mu},B^*}w_j$ leads to
\be \hspace{-5cm}\llangle \overset{}{S_{\mu,B}}f_i,w_j\rrangle - \llangle (A_B-\lambda)^{-1}S_{\mu,B}f_i,(\tilde{\mu}-\overline{\lambda})\tilde{S}_{\tilde{\mu},B^*}w_j
\rrangle = \label{eq:nevan} \ee
\[ \hspace{3cm} \llangle \Gamma_2(A_B-\lambda)^{-1}S_{\mu,B}f_i,\Gammat_1\tilde{S}_{\tilde{\mu},B^*}w_j\rrangle
 - \llangle \Gamma_2(A_B-\lambda)^{-1}S_{\mu,B}f_i,\Gammat_1\tilde{S}_{\tilde{\mu},B^*}w_j\rrangle . \]
If the functions $\llangle M_B(\lambda)f_i,w_j\rrangle $ lie in $N(\mathbb{C}_{\pm})$ then, thanks to the identity
\be M_B(\lambda) = \Gamma_2(I+(\lambda-\mu)(A_B-\lambda)^{-1})S_{\mu,B}, \label{eq:midentity} \ee
the terms $\llangle \Gamma_2(A_B-\lambda)^{-1}S_{\mu,B}f_i,\cdot\rrangle$ appearing in (\ref{eq:nevan}) also
lie in $N(\mathbb{C}_{\pm})$,  so  
$$\llangle (A_B-\lambda)^{-1}S_{\mu,B}f_i,(\tilde{\mu}-\overline{\lambda})\tilde{S}_{\tilde{\mu},B^*}w_j
\rrangle$$ lies in $N(\mathbb{C}_{\pm})$. This implies that $\llangle (A_B-\lambda)^{-1}S_{\mu,B}f_i,\tilde{S}_{\tilde{\mu},B^*}w_j
\rrangle$ lies in $N(\mathbb{C}_{\pm})$.

The converse result is immediate from equation (\ref{eq:midentity}): if inner products of the form 
$$\llangle (A_B-\lambda)^{-1}S_{\mu,B}f_i,\tilde{S}_{\tilde{\mu},B^*}w_j\rrangle$$
 lie in $N(\mathbb{C}_{\pm})$ then so do the inner products $\llangle M_B(\lambda)f_i,w_j\rrangle$.
%lie in $N(\mathbb{C}_{\pm})$.
\end{proof}

\begin{thm}\label{thm:5.1} Suppose that Assumption \ref{assumption:new1} holds and that 
$
\eps | \langle M_B(k\pm i \eps)f,w\rangle |\to 0$ as $\eps \searrow 0$, for a.e. $k\in \R$ for $f$, $w$ in a dense countable subset of the boundary spaces.
Choose sets $\{\mu_i\}_{i\in I},\;\{\tilde \mu_j\}_{j\in\tilde{I}}$ non-real and outside  $\sigma(A_B)$. Let $F_i=S_{\mu_i,B}f_i,\;\;v_j=\tilde S_{\tilde \mu_j,B^*}w_j$.  Then for a.e. $k\in \R$,
$$
\mbox{\rm rank} \Big ( \Big [ P_{\{v_j\}}(A_B-\lambda)^{-1}P_{\{F_i\}}\Big ]_{\lambda=k}\Big )=
\mbox{\rm rank} \Big ( \Big [  P_{\{w_j\}}M_B(\lambda)P_{\{f_i\}}\Big ]_{\lambda = k}\Big ),
$$
where by $P_{\{v_j\}}$ and $P_{\{w_j\}}$ we denote the projections onto the indicated one-dimensional spaces, and $[\cdot ]_{\lambda =k}$ denotes the jump between $\lambda=k+i\eps$ and $\lambda=k-i\eps$ as $\eps\searrow 0$.
\end{thm}

In order to prove Theorem \ref{thm:5.1} we require the following lemma.
\begin{lemma}\label{lemma:proper1} The collections $\{ F_i \}_{i\in I}$ and
$\{ v_j \}_{j\in \tilde{I}}$ are \ian{both} linearly independent.
\end{lemma}
\begin{proof} We give the proof for the collection  $\{ F_i \}_{i\in I}$; the remaining case is similar. Assume that there are some constants $\alpha_i$ such that 
$\sum_{i}\alpha_i F_i = 0$. Let $\zeta\in {\mathbb C}$: applying $(\At^*-\zeta)^k$ for some $k\in \mathbb{N}$ we get 
$\sum_{i}\alpha_i (\mu_i-\zeta)^k F_i = 0$. \ian{First, assume all  the $\mu_i$ are distinct. Then we
can choose $i_0$ and $\zeta$} such that $|\mu_{i_0}-\zeta|>|\mu_i-\zeta|$ for $i\neq i_0$. Letting $k\rightarrow\infty$ we 
deduce that $\alpha_{i_0}F_{i_0}=0$. Proceeding in this way we get $\alpha_i=0$ for all $i$ as long as the $\mu_i$ 
are distinct. 

If we have a collection of $\mu_i$ which are all equal, say for $i\in J$, where $J$ is some index set, then we 
can prove that for some appropriately chosen $\zeta$ we have $0=\sum_{i\in J}\alpha_i(\mu_i-\zeta)^kS_{\mu_i,B}f_i$ 
giving, for $\ell\in J$, $S_{\mu_\ell,B}\sum_{i\in J}\alpha_i f_i = 0$. This implies that $\sum_{i\in J}\alpha_i f_i = 0$ 
and hence that $\alpha_i=0$ for all $i\in J$, by linear independence of the $\{ f_i \}$.
\end{proof}

\begin{corollary} The collections $\{ (\mu_i-k)F_i \}$ and $\{ (\overline{\tilde \mu}_j-k)v_j \}$ are \ian{both} linearly independent
as long as $\mu_i\neq k$, $\overline{\tilde \mu_j}\neq k$, for all $i,j$.

\end{corollary}
\begin{proof}[Proof of Theorem \ref{thm:5.1}] \ian{
%Let $\{ f_i \}$ and $\{ w_j \}$ be as in Lemma \ref{lemma:proper1}; let $F_i = S_{\mu_i,B}f_i$, $v_j = {\tilde S}_{{\tilde \mu}_j,B^*}w_j$ and $\lambda= k+i\eps$.
We use the fundamental identity (\ref{eq:fund}), which yields
\be \Big\langle F_i-(A_B-\lambda)^{-1}(\At^*-\lambda)F_i,(A^*-\overline{\lambda}I)v_j\Big\rangle
 = -\llangle f_i,\Gammat_2 v_j\rrangle_{\cH}
 + \Big\langle M_B(\lambda)f_i,w_j\Big\rangle_{\cK}\label{eq:fund3}
\ee
for all $i\in I$, $j\in\widetilde{I}$ and $\lambda=k+i\eps$.}
The jump \ian{at $k$} of the right hand side is clearly given by $\left[\llangle  M_B(k)f_i,w_j\rrangle_{\K}\right]$, which for convenience we denote by $\llangle  [M_B](k)f_i,w_j\rrangle_{\K}$. By our assumptions on \ian{the $\{f_i\}$ and $\{w_j\}$,} 
%the ranges of the trace operators, 
this is nonzero if and only if $[M_B](k)\neq 0$. 

Now consider the left hand side of (\ref{eq:fund3}). \ian{Clearly, $\Big\langle F_i,(A^*-\overline{\lambda}I)v_j\Big\rangle$ has no jump.}

Defining
\be G(\lambda) := \llangle (\overset{}{A_B}-\lambda)^{-1}F_i,A^*v_j\rrangle + \llangle (A_B-\lambda)^{-1}\At^*F_i,v_j\rrangle \label{eq:Gdef2} \ee
the \ian{negative of the remaining term on the} left hand side of (\ref{eq:fund3}) is
%\begin{align}\label{eq:jump3} % requires amsmath; align* for no eq. number
    %&  \llangle (A_B-\lambda)^{-1}(\At^*-\lambda)F_i,(A^*-\lambda)v_j\rrangle = 
 %\llangle (A_B-\lambda)^{-1}\At^*F_i,A^*v_j\rrangle - \lambda G(\lambda) \\  \nonumber
 %&  + \lambda^2 \llangle (\overset{}{A_B}-\lambda)^{-1}F_i,v_j\rrangle. 
%\end{align}
\ben \label{eq:jump3}  \llangle (A_B-\lambda)^{-1}(\At^*-\lambda)F_i,(A^*-\ian{\overline{\lambda}})v_j\rrangle &=& 
 \llangle (A_B-\lambda)^{-1}\At^*F_i,A^*v_j\rrangle - \lambda G(\lambda) \\ \nonumber  && + \ian{|\lambda|^2} \llangle (\overset{}{A_B}-\lambda)^{-1}F_i,v_j\rrangle. 
\een
Observe that $[\lambda G(\lambda)] = k[G](k) + \lim_{\eps\searrow 0}i\eps(G(k+i\eps)+G(k-i\eps))$.
We shall prove later that
\be \lim_{\eps\searrow 0} \eps\llangle (A_B-\lambda)^{-1}F_i,v_j\rrangle = 0 \;\; \mbox{for a.e. $k$}.
\label{eq:weirdassumption2}
\ee
This implies $[\lambda G(\lambda)] = k[G](k)$ and that $$[\ian{|\lambda|^2}\llangle (A_B-\lambda)^{-1}F_i,v_j\rrangle ]
= k^2 \llangle [(A_B-\lambda)^{-1}]F_i,v_j\rrangle;$$ hence, from (\ref{eq:jump3}), the formula
\[ \ian{-}\left[\llangle (A_B-\lambda)^{-1}(\At^*-\lambda)F_i,(A^*-\ian{\overline{\lambda}})v_j\rrangle \right]
 = \llangle [(A_B-\lambda)^{-1}](\At^*-k)F_i,(A^*-k)v_j\rrangle \]
gives the jump of the left hand side of (\ref{eq:fund3}). We therefore have
\be\label{eq:jump} \ian{-}(\mu_i-k)(\overline{{\tilde\mu}_j}-k)\llangle [(A_B-\lambda)^{-1}]F_i,v_j\rrangle = \llangle [ M_B](k)f_i,w_j\rrangle_{\K}, \ee
from which our result follows.

It remains to establish (\ref{eq:weirdassumption2}). Returning to (\ref{eq:fund3}) we have
\beq -(\mu_i-\lambda)(\overline{{\tilde \mu}_j}-\lambda)\llangle(\overset{}{A_B}-\lambda)^{-1}F_i,v_j\rrangle
 = (\overline{{\tilde \mu}_j}-\lambda)\llangle \overset{}{F_i},v_j\rrangle-\llangle f_i,\Gammat_2 v_j\rrangle_{\Hc}
 + \llangle \overset{}{M_B}(\lambda)f_i,w_j\rrangle_{\K}\label{eq:fund4}
\eeq
whence 
\[ \left| \llangle(A_B-\lambda)^{-1}F_i,v_j\rrangle \right| \leq 
 \frac{ |\overline{{\tilde \mu}_j}-\lambda| \left|\Big\langle F_i,v_j\Big\rangle\right|
+\left|\llangle f_i,\Gammat_2 v_j\rrangle_{\Hc}\right|
 + \left|\llangle M_B(\lambda)f_i,w_j\rrangle_{\K}\right| }{|\mu_i-\lambda||\overline{{\tilde \mu}_j}-\lambda|} . \]
Thus (\ref{eq:weirdassumption2}) is an immediate consequence of the hypothesis that
$\eps | \langle M_B(k\pm i \eps)f,w\rangle |\to 0$ as $\eps \searrow 0$, for a.e.~$k\in \R$.
\end{proof}

\begin{thm}\label{thm:5.2} 
Let $\{ f_i \}_{i\in I}$ and $\{ w_j \}_{j\in \tilde{I}}$ be linearly independent vectors whose spans are dense in
$\cH$ and $\cK$ respectively. Let $\{ \mu_{\ell} \}_{\ell\in J}$ and $\{ \tilde{\mu}_{\nu}\}_{\nu\in \tilde{J}}$ be
collections of distinct strictly complex numbers dense in ${\mathbb C}\setminus \sigma(A_B)$ and 
${\mathbb C}\setminus \sigma(A_B^*)$ respectively. Define $F_{i,\ell}=S_{\mu_\ell,B}f_i$, $v_{j,\nu} = \tilde{S}_{\tilde{\mu}_\nu,B^*}w_j$. 
\begin{enumerate}
\item The collections $\{ F_{i,\ell} \}_{i\in I,\ell\in J}$ and $\{ v_{j,\nu} \}_{j\in \tilde{I},\nu\in \tilde{J}}$ are \ian{both} linearly 
 independent and their spans are dense in $\Sco$ and $\Scto$ respectively.
\item For $N,M\in \mathbb{N}$ let \ian{$P_{N,\Sco}$} and $P_{M,\Scto}$ denote projections onto $N$- and $M$-dimensional 
 subspaces of $\Sco$ and $\Scto$ respectively, spanned by $N$ of the $F_{i,\ell}$ and $M$ of the $v_{j,\nu}$ respectively,
 chosen such that ${\displaystyle \lim_{N\rightarrow\infty}P_{N,\Sco}=I}$ and ${\displaystyle \lim_{M\rightarrow\infty}P_{M,\Scto}=I}$, 
 in the sense of strong convergence.
Put $$
\ian{E_1=\left\{ k\in\R \ \Big|\ [M_B](k)\mbox{  exists in the weak topology}\right\},}
$$
$$
\ian{E_2=\left\{ k\in\R \ \Big|\ \lim_{\eps \searrow 0} \eps | \langle M_B(k\pm i\eps )f_i,w_j\rangle|=0\mbox{ for all $i$, $j$} \right\}.}
$$
%$$
%E_3^c=\cup_M\sigma(P_{M,\tilde S} A^* P_{M.\tilde S}),
%$$
%$$
%E_4^c=\cup_N\sigma(P_{N,S}\tilde A^* P_{N,S}).
%$$
For any \ian{$k\in E_1\cap E_2$}
we have that
$
[P_{M,\Scto} (A_B-\lambda)^{-1}P_{N,\Sco}](k)$ exists; moreover 
$$
\lim_{N,M\to\infty} \mbox{\rm rank}([P_{M,\Scto}(A_B-\lambda)^{-1}P_{N,\Sco}]|_{\lambda=k})
$$ 
exists and is equal to $\mbox{\rm rank}([M_B](k))$.
%\item Let $(\widehat P_{N,S}), (\widehat P_{M,\tilde S})$ be projections formed from different collections
%of vectors $\{ \hat{f}_i \}$, $\{ \hat{w}_j \}$ and different complex numbers $\{ \hat{\mu}_ell \}$ and
%$\{ \hat{\tilde{\mu}}_{\nu}\}$, obeying the same hypotheses as before. Then for almost all $k\in\mathbb R$,
%$$
%\lim _{N,M\to\infty} \mbox{\rm rank}([\widehat P_{M,\tilde S}(A_B-\lambda)^{-1}\widehat P_{N,S}]|_{\lambda=k})
%$$ exists and is also equal to rank $[M_B](k).$
\end{enumerate}
\end{thm}
\begin{proof}
\hspace{1in}

\begin{enumerate}
\item 
The fact that the closed linear spans of the sets $\{ F_{i,\ell} \}_{i\in I,\ell\in J}$ and $\{ v_{j,\nu} \}_{j\in\tilde{I},\nu\in\tilde{J}}$
are $\Sco$ and $\Scto$ follows immediately from the definitions of $\Sco$ and $\Scto$ together with the fact that the closed linear spans
of $\{ f_i \}_{i\in I}$ and $\{ w_j \}_{j\in\tilde{I}}$ are $\cH$ and $\cK$ respectively. It remains only to establish linear independence.
Assume that there exist constants $\alpha_{i,\ell}$ such that $\sum_{i,\ell}\alpha_{i,\ell}F_{i,\ell}=0$.
This means that $\sum_{i,\ell}\alpha_{i,\ell}S_{\mu_\ell,B}f_i = 0$. Applying $\At^*$ $k$ times yields
\[ \sum_{i,\ell}\alpha_{i,\ell}\mu_l^k F_{i,\ell} = \sum_{\ell}\mu_l^k \sum_{i}\alpha_{i,\ell}F_{i,\ell} = 0, \;\;\; k = 0,1,2,\ldots. \]
Since the $\mu_l$ are distinct this yields $\sum_{i}\alpha_{i,\ell}F_{i,\ell}=0$ for all $\ell$. This means
$S_{\mu_\ell,B}\sum_{i}\alpha_{i,\ell}f_i = 0$, and since $S_{\mu_\ell,B}$ has a left inverse this implies
$\sum_{i}\alpha_{i,\ell}f_i=0$. But the $\{ f_i \}$ are linearly independent, so we deduce that $\alpha_{i,\ell}=0$
for all $i$ and $\ell$. 
\item
Let $P_{N,\Sco}$ denote projection onto $N$ of the $F_{i,\ell}$ and $P_{M,\Scto}$ denote projection onto $M$ of 
the $v_{j,\nu}$, chosen in each case to be such that $P_{N,\Sco}$ and $P_{M,\Scto}$ converge strongly to the 
identity. Let $P_{N'}$ and $\tilde{P}_{M'}$ denote projections onto the spaces spanned by the corresponding $f_i$, of
which there will be $N'\leq N$, and $w_j$, of which there will be $M'\leq M$. Since \ian{$k\in E_1\cap E_2$} we may invoke \ian{\eqref{eq:jump} from the proof of}
Theorem \ref{thm:5.1} and deduce that
\[ \ian{- (\mu_\ell-k)(\overline{\tilde{\mu}}_\nu-k)}\llangle [P_{M,\Scto}(A_B-\lambda)^{-1}P_{N,\Sco}]_{\lambda=k}F_{i,\ell},v_{j,\nu}\rrangle
 = \llangle [\tilde{P}_{M'}M_B(\lambda)P_{N'}]_{\lambda=k}f_i,w_j\rrangle. \]
As $k\in \mathbb{R}$ we know that $\mu_\ell\neq k$ and $\tilde{\mu}_\nu\neq k$, so we define
\[ X_{i,\ell} = (\mu_\ell-k)F_{i,\ell}, \;\;\; Y_{j,\nu} = \ian{-} (\tilde{\mu}_\nu-k)v_{j,\nu}. \]
The vectors $\{ X_{i,\ell} \}$ and $\{ Y_{j,\nu} \}$ are \ian{both} linearly independent, and \ian{for each $\ell\in J, \nu\in\widetilde{J}$} we have
\[
\llangle [P_{M,\Scto}(A_B-\lambda)^{-1}P_{N,\Sco}]_{\lambda=k}X_{i,\ell},Y_{j,\nu}\rrangle
 = \llangle [\tilde{P}_{M'}M_B(\lambda)P_{N'}]_{\lambda=k}f_i,w_j\rrangle. \]
Define $M_1$ to be the matrix with entries $$\llangle [P_{M,\Scto}(A_B-\lambda)^{-1}P_{N,\Sco}]_{\lambda=k}X_{i,\ell},Y_{j,\nu}\rrangle,$$
ordered by incrementing $\ell$ and $\nu$ before $i$ and $j$, and $M_2$ to be the matrix with entries 
$\llangle [\tilde{P}_{M'}M_B(\lambda)P_{N'}]_{\lambda=k}f_i,w_j\rrangle$. It follows from the definition of the Kronecker product that 
\[ M_1 = M_2\otimes E \]
in which $E$ is a matrix whose entries are all equal to $1$. By consideration of the singular values of the Kronecker
product ($E$ has only one non-zero singular value) it follows that $M_1$ and $M_2$ have the same rank, and hence that
\[ \mbox{\rm rank}([P_{M,\Scto}(A_B-\lambda)^{-1}P_{N,\Sco}]_{\lambda=k}) = \mbox{\rm rank}([\tilde{P}_{M'}M_B(\lambda)P_{N'}]_{\lambda=k}). \]
If we define
\[ \mbox{\rm rank}([P_{\Scto}(A_B-\lambda)^{-1}P_{\Sco}]_{\lambda=k}) := 
 \lim_{M,N\rightarrow\infty} \mbox{\rm rank}([P_{M,\Scto}(A_B-\lambda)^{-1}P_{N,\Sco}]_{\lambda=k}) \]
and exploit the fact that the $\{ f_i \}_{i\in I}$ and $\{ v_j \}_{j\in \tilde{I}}$ exhaust $\cH$ and $\cK$ respectively, then it follows
that 
\[ \mbox{\rm rank}([P_{\Scto}(A_B-\lambda)^{-1}P_{\Sco}]_{\lambda=k}) = \mbox{\rm rank}([M_B(\lambda)]_{\lambda=k}) . \]
\end{enumerate}
\end{proof}

As a final remark, we mention that $\mbox{\rm rank}([P_{\Scto}(A_B-\lambda)^{-1}P_{\Sco}]_{\lambda=k})$ is  the multiplicity of
the absolutely continuous spectrum of $\left. A_B\right|_{\overline{\mathcal S}}$.

\section{Friedrichs model: reconstruction of $M_B(\lambda)$ from one \ian{restricted} resolvent $(A_B-\lambda)^{-1}|_{\overline{\mathcal S}}$}
\label{section:6}
We have now finished our abstract considerations   and the remainder of the paper is devoted to a detailed analysis of the Friedrichs model.

In this section we show how to reconstruct $M_B(\lambda)$ explicitly from the restricted resolvent. The fact that even the bordered resolvent
determines $M_B(\lambda)$ uniquely was proved in the abstract setting in Theorem \ref{Munique}, but of course methods
of reconstruction depend on the concrete operators under scrutiny.

We introduce the notation $\widehat{\cdot}$ for the Cauchy or Borel transform given by
\be\label{eq:mm9} \widehat{\overline \phi}(\lambda) = \ll \frac{1}{t-\lambda},\phi\rr, \;\;\;
 \widehat{\psi}(\lambda) = \ll \frac{1}{t-\lambda},\overline{\psi}\rr 
 \ee
 and   $P_\pm:L^2(\R)\to H_2^\pm(\R)$ for the Riesz projections given by 
 \be\label{Riesz}P_\pm f(k)=\pm\frac{1}{2\pi i}\lim_{\eps\to 0} \widehat{ f}(k \pm i\eps) =\pm\frac{1}{2\pi i}\lim_{\eps\to 0}\int_\R\frac{f(x)}{x-(k\pm i\eps)}dx,\ee
where the limit is to be understood in $L^2(\R)$ (see \cite{Koosis}). \ian{Here, $H^+_p(\R)$ and $H^-_p(\R)$ denote the Hardy spaces of boundary values of $p$-integrable functions in the upper and lower complex half-plane, respectively.}
To simplify notation, we also sometimes write $(\hat f)_\pm(k) =\widehat f (k\pm i0):=2\pi i P_\pm f(k)$.
%\textcolor{blue}{Remark: We should be consistent in notation and not introduce more than necessary.}

\begin{Theorem}
\label{thm:1}
For the Friedrichs model, assume that 
$(A_B-\lambda)^{-1}|_{\overline{\mathcal S}}$ is known \ian{for  all
$\lambda\in\rho(A_B)\setminus\Rr$.} Then $M_B(\lambda)$ can be recovered.
\end{Theorem}

\begin{Remark}
We assume that $(A_B-\lambda)^{-1}|_{\overline{\mathcal S}}$ is known \ian{for all
$\lambda\in \rho(A_B)\setminus\Rr$}, though it is certainly sufficient to know it at one 
point in each connected component of $\Cc\setminus\sigma(A_B)$. If $\sigma(A_B)$
does not cover all of either half-plane $\Cc_{\pm}$ then it is enough to know
$(A_B-\lambda)^{-1}|_{\overline{\mathcal S}}$ at two points, one in each of 
$\Cc_{\pm}$. If, additionally, $\sigma(A_B)$ does not cover $\Rr$, then it suffices
to know $(A_B-\lambda)^{-1}|_{\overline{\mathcal S}}$ for just one value of
$\lambda$.
\end{Remark}
\begin{proof}[Proof of Theorem \ref{thm:1}] %We first recall that the resolvent is given by \eqref{eq:mm1c}.
%Note that although this is not clear from the expression for $c_f$ given in \eqref{eq:mm10}, $c_f$ is
%determined from $f$ by the condition $xf(x)-c_f\mathbf{1}\in L^2(\Rr)$ and is therefore
%known if $f$ is known.

{\bf 1. Recovering the function $\psi$.} Take non-zero $g\in\Sco$ and $\lambda
\in \Cc\setminus(\Rr\cup\sigma(A_B))$. Observe that (\ref{eq:mm1c}) may be rewritten in the form
\begin{equation}\label{eq:mm2}
f(x) - \frac{g(x)}{x-\lambda} - \frac{c_f}{x-\lambda} = \frac{\psi(x)}{x-\lambda}A(\lambda),
\end{equation}
in which 
\[ A(\lambda) = -\frac{1}{ D(\lambda)}\left[\left\langle \frac{g}{t-\lambda},\phi\right\rangle
 + c_f \left\langle\frac{1}{t-\lambda},\phi\right\rangle\right] \]
 and $D(\lambda)$  is given by   \eqref{eq:Ddef}.
 The left hand side of (\ref{eq:mm2}) is known as a function of $\lambda$, at least for 
 $g\in\overline{\mathcal{S}}$. To determine $\psi$ {\bf up to a scalar multiple} it is therefore sufficient to
 find $g$ and $\lambda$ so that  $A(\lambda)$ is non-zero: 
 in other words, find $g$ such that the function $A(\cdot)$ is not identically zero.
 
 We proceed by contradiction. Suppose we have a non-trivial Friedrichs model (i.e.
 neither $\phi$ nor $\psi$ is identically zero). If $A(\cdot)$ is identically zero then multiplying by $M_B(\lambda)^{-1}$ \ian{from \eqref{eq:9} and using \eqref{eq:mm10}}
 we obtain
 \begin{align}
\left[i\pi\mbox{sign}(\Im\lambda)-\frac{1}{ D(\lambda)}
 \left\langle\frac{1}{t-\lambda},\phi\right\rangle\left\langle\frac{1}{t-\lambda},
 \overline{\psi}\right\rangle-B\right]\left\langle\frac{g}{t-\lambda},\phi\right\rangle
\hspace{1cm} \nonumber \\ +\left[-\left\langle\frac{1}{t-\lambda},\overline{g}\right\rangle+\frac{1}{ D
 (\lambda)}\left\langle\frac{g}{t-\lambda},\phi\right\rangle\left\langle\frac{1}{t-\lambda},
 \overline{\psi}\right\rangle\right]\left\langle\frac{1}{t-\lambda},\phi\right\rangle\equiv 0,
 \end{align}
from which it follows
\begin{equation}\label{eq:mm3}
(i\pi\mbox{sign}(\Im\lambda)-B)\left\langle \frac{g}{t-\lambda},\phi\right\rangle
 - \left\langle \frac{1}{t-\lambda},\overline{g}\right\rangle \left\langle \frac{1}{t-\lambda},\phi
 \right\rangle \equiv 0.
 \end{equation}
 For all non-real
$\mu$ such that $ D(\mu)$ is nonzero (this is true for a.e.~non-real $\mu$ by analyticity),
 there exists $g\in\Sco$ in the range of the solution operator $S_{\mu,B}$. We know \ian{from \eqref{eq:7}} that such $g$ have the
form
\begin{equation}\label{eq:mm4}
g(x) = \frac{1}{x-\mu}-\frac{1}{ D(\mu)}\left\langle\frac{1}{t-\mu},
\phi\right\rangle\frac{\psi(x)}{x-\mu}, 
\end{equation}
though we do not know the function $\psi$ or the value of 
$\frac{1}{ D(\mu)}\left\langle\frac{1}{t-\mu},
\phi\right\rangle$. Substituting (\ref{eq:mm4}) into (\ref{eq:mm3}) yields
\ben
&&(i\pi\mbox{sign}(\Im\lambda)-B)\left[\ll\frac{1}{(t-\mu)(t-\lambda)},\phi\rr
 - \frac{1}{ D(\mu)}\ll \frac{1}{t-\mu},\phi\rr\ll\frac{\psi}{(t-\mu)(t-\lambda)},
 \phi\rr\right] \nonumber \\
&& \equiv \ll \frac{1}{t-\lambda},\phi\rr\left[\ll\frac{1}{(t-\lambda)(t-\mu)},\mathbf{1}\rr
  - \frac{1}{ \ian{D(\mu)}}\ll\frac{1}{t-\mu},\phi\rr\ll\frac{1}{(t-\lambda)(t-\mu)},
  \overline{\psi}\rr\right]. \label{eq:mm5}
\een
If we use the identity
\be \frac{\lambda-\mu}{(t-\lambda)(t-\mu)} = \frac{1}{t-\lambda}-\frac{1}{t-\mu} \label{eq:mm12}\ee
and use the notations from \eqref{eq:mm9}
then \ian{multiplying by $(\lambda-\mu)$, \eqref{eq:mm5}} %(\ref{eq:mm5}) 
becomes
\begin{align}
(i\pi\mbox{sign}(\Im\lambda)-B)\left[\widehat{\overline{\phi}}(\lambda)-
\widehat{\overline{\phi}}(\mu) - \frac{1}{ D(\mu)}\widehat{\overline{\phi}}(\mu)
( D(\lambda)- D(\mu))\right] \nonumber \\
\equiv \widehat{\overline{\phi}}(\lambda)\left[\int_{\Rr}\frac{\lambda-\mu}{(t-\lambda)(t-\mu)}dt
- \frac{\widehat{\overline{\phi}}(\mu) }{ D(\mu)}(\widehat{{\psi}}(\lambda)-\widehat{{\psi}}(\mu))
\right].\label{eq:mm6}
\end{align}
Performing the integral for the case in which $\Im\lambda\cdot\Im\mu<0$, we obtain
\be
(i\pi\mbox{sign}(\Im\lambda)-B)\left[\widehat{\overline{\phi}}(\lambda)
- \frac{ D(\lambda)}{ D(\mu)}\widehat{\overline{\phi}}(\mu)\right] 
\equiv \widehat{\overline{\phi}}(\lambda)\left[\pm 2\pi i
- \frac{\widehat{\overline{\phi}}(\mu) }{ D(\mu)}(\widehat{{\psi}}(\lambda)-\widehat{{\psi}}(\mu))
\right].
\ee
Fix $\lambda$ and let $\mu\rightarrow i\infty$, so that $ D(\mu)\rightarrow 1$
and $\widehat{\overline\phi}(\mu)\rightarrow 0$. This yields
\be (i\pi\mbox{sign}(\Im\lambda)-B)\widehat{\overline{\phi}}(\lambda) 
\equiv \pm2\pi i \widehat{\overline{\phi}}(\lambda). \label{eq:mm7} \ee
If, on the other hand, we consider $\Im\lambda\cdot\Im\mu>0$ in (\ref{eq:mm6}) then the value
of the integral is zero, and we obtain, upon letting $\mu\rightarrow i\infty$,
\be (i\pi\mbox{sign}(\Im\lambda)-B)\widehat{\overline\phi}(\lambda)\equiv 0. \label{eq:mm8} \ee
Equations (\ref{eq:mm7},\ref{eq:mm8}) together imply that $\widehat{\overline{\phi}}$ is identically
zero, and hence so is $\phi$. In this case the function $\psi$ is irrelevant and so our Friedrichs
model is trivial, a contradiction.
Thus (\ref{eq:mm2}) determines $\psi$ up to a constant multiple. We may choose this (non-zero)
multiple arbitrarily, since $\phi$ can be rescaled if necessary to obtain the correct Friedrichs model.

{\bf 2. Recovering the boundary condition parameter $B$.} Returning to the parameter
$c_f$ in (\ref{eq:mm10}) and using the notation (\ref{eq:mm9}), we have
\bea \left[i\pi \mbox{sign}(\Im\lambda)-B-\frac{1}{D(\lambda)}\widehat{\overline{\phi}}(\lambda)
\widehat{\psi}(\lambda)\right]c_f 
&=& \left[-\ll \frac{1}{t-\lambda},\overline{g}\rr \ian{+} \frac{1}{D(\lambda)}
\ll \frac{g}{t-\lambda},\phi\rr\ll
\frac{1}{t-\lambda},\overline{\psi}\rr\right] \\
&=& \left[-\ll \frac{1}{t-\lambda},\overline{g}\rr + O\left(\| g \|_2 |\Im\lambda|^{-3/2}\right)\right], 
\eea
as $\Im\lambda\rightarrow\infty$, and uniformly in $g$. Now choose an element 
\be g(x)\equiv g_\mu(x) := \frac{1}{x-\mu}-\sigma(\mu)\frac{\psi(x)}{x-\mu}, \label{eq:mm11}\ee
$\mu\in\Cc\setminus\Rr$, $D(\mu)\neq 0$, with some $\sigma(\mu) = O(|\Im\mu|^{-1/2})$.
We know that such $\sigma(\mu)$ exists, and indeed may be chosen as $\widehat{\overline \phi}(\mu)/
D(\mu)$, but we do not yet know $\phi$ and therefore do not claim that our particular choice of
$\sigma$ is given by this formula. 
%\textcolor{blue}{\em Uniqueness of the representation (\ref{eq:mm4}) and, based
%on it, another approach to the problem, to be analyzed later.} 
We fix some choice of $\sigma$,
so that $g=g_\mu$ is determined and $c_f$ is known as a function of $\lambda$ and $\mu$. We have
\bea &&
(i\pi\mbox{sign}(\Im\lambda)-B+O(|\Im\lambda|^{-1}))c_f \\ && \hspace{50pt} = \left[-\ll\frac{1}{t-\lambda},\frac{1}{t-\overline{\mu}}\rr
  + \sigma(\mu)\ll\frac{1}{t-\lambda},\frac{\overline{\psi}}{t-\overline{\mu}}\rr
   + O(|\Im\lambda|^{-3/2}) \| g_\mu \|_2 \right] \nonumber \\
    && \hspace{50pt} = -\int_{\Rr}\frac{1}{(t-\lambda)(t-\mu)}dt 
    + O(|\Im\mu|^{-3/2})O(|\Im\lambda|^{-1/2})\nonumber \\ && \hspace{65pt}
    +O(|\Im\lambda|^{-3/2})
    \left(O(|\Im\mu|^{-1/2})+\|\psi\|_2\frac{|\sigma(\mu)|}{|\Im\mu|}\right).
\eea
Assuming that $\Im\lambda\cdot\Im\mu<0$, we know that 
\[ -\int_{\Rr}\frac{1}{(t-\lambda)(t-\mu)}dt = \frac{\pm 2\pi i}{\lambda-\mu} . \]
Put $\lambda = -\mu$ and letting $\Im\mu\rightarrow\infty$, we obtain
\[ (i\pi\mbox{sign}(\Im\lambda)-B)c_f = \frac{\pm 2\pi i}{2\lambda} + O(|\lambda|^{-2}). \]
For one choice of $\mbox{sign}(\Im\lambda)$ at least, $i\pi\mbox{sign}(\Im\lambda)-B\neq 0$
and so we can recover $B$ from the asymptotic behaviour of $c_f$ as $
\Im\lambda\rightarrow\infty$.

{\bf 3. Recovering $\widehat{\overline{\phi}}(\lambda)/D(\lambda)$.}
Once again we choose $g=g_\mu$ of the form (\ref{eq:mm11}). Returning to (\ref{eq:mm2})
and indicating the $\mu$-dependence of $f$ by writing $f=f_\mu=(A_B-\lambda)^{-1}g_\mu$, we have
\[ (A_B-\lambda)^{-1}g_\mu - \frac{g_\mu(x)}{x-\lambda} - \frac{c_{f_\mu}(\lambda)}{x-\lambda}
 = -\frac{\psi(x)}{x-\lambda}\frac{1}{D(\lambda)}\left[\ll \frac{g_\mu}{t-\lambda},\phi\rr
  + c_{f_\mu}(\lambda)\ll\frac{1}{t-\lambda},\phi\rr\right]. \]
Since the left hand side of this equation is known and since $\psi$ is known, this implies that
\[ \frac{1}{D(\lambda)}\left[\ll \frac{g_\mu}{t-\lambda},\phi\rr + c_{f_\mu}(\lambda)
\ll\frac{1}{t-\lambda},\phi\rr\right] \]
is known. Substituting the known choice of $g_\mu$ we discover that
\[ \frac{1}{D(\lambda)}\left[ \ll \frac{1}{(t-\lambda)(t-\mu)},\phi\rr
 - \sigma(\mu)\ll \frac{\psi}{(t-\lambda)(t-\mu)},\phi\rr
  + c_{f_\mu}(\lambda)\ll \frac{1}{t-\lambda},\phi\rr\right](\lambda-\mu) \]
is known too. Using identity (\ref{eq:mm12}) this means that
\be \frac{1}{D(\lambda)}\left[\widehat{\overline\phi}(\lambda)-\widehat{\overline{\phi}}(\mu)
 - \sigma(\mu)(D(\lambda)-D(\mu)) + (\lambda-\mu)c_{f_\mu}(\lambda)
 \widehat{\overline{\phi}}(\lambda) \right]\label{eq:mm13}
 \ee
is known. We shall now fix $\lambda$ and let $\Im\mu\rightarrow \infty$, for which purpose
we need to know how $(\lambda-\mu)c_{f_\mu}(\lambda)$ will behave. From (\ref{eq:mm10}),
we have 
\begin{align}
 c_{f_\mu}(\lambda)(\lambda-\mu) = (\lambda-\mu)M_B(\lambda)\left[
  - \ll \frac{1}{t-\lambda},\frac{1}{t-\overline{\mu}}\rr + \sigma(\mu)\ll\frac{1}{t-\lambda},
   \frac{\overline{\psi}}{t-\overline{\mu}}\rr \right. \nonumber \\
   \left. + \frac{\widehat{\psi}(\lambda)}{D(\lambda)}
    \left\{ \ll \frac{1}{(t-\lambda)(t-\mu)},\phi\rr - \sigma(\mu)\ll\frac{\psi}{(t-\lambda)(t-\mu)},
    \phi\rr\right\} \right] .
    \end{align}
Choosing $\mu\neq \lambda$ with $\Im\lambda\cdot\Im\mu>0$ causes the integral term
$\ll \frac{1}{t-\lambda},\frac{1}{t-\overline{\mu}} \rr $ to vanish. This yields
\begin{align} c_{f_\mu}(\lambda) & = M_B(\lambda)\left[\sigma(\mu)(\widehat{\psi}(\lambda)-\widehat{\psi}(\mu))
   + \frac{\widehat{\psi}(\lambda)}{D(\lambda)}(\widehat{\overline{\phi}}(\lambda)-
   \widehat{\overline{\phi}}(\mu)) - \sigma(\mu)(D(\lambda)-D(\mu))\right] \nonumber \\
  &  \rightarrow M_B(\lambda)\frac{\widehat{\psi}(\lambda)}{D(\lambda)}\widehat{\overline{\phi}}(\lambda),
  \;\;\; \Im\mu\rightarrow\infty.
\end{align}
Letting $\Im\mu\rightarrow\infty$ in (\ref{eq:mm13}) therefore yields that
\be \frac{1}{D(\lambda)}\left[\widehat{\overline{\phi}}(\lambda) +
 M_B(\lambda) \frac{\widehat{\psi}(\lambda)}{ D(\lambda)}\widehat{\overline{\phi}}(\lambda)^2 \right]
 \label{eq:mm15}
 \ee
is known. However from (\ref{eq:9}) we have
\[ M_B(\lambda) = \left[ i\pi\mbox{sign}(\Im \lambda )-\frac{1}{ D(\lambda)}\widehat{\overline{\phi}}(\lambda)
\widehat{\psi}(\lambda)-B\right]^{-1} \]
and so the known quantity appearing in (\ref{eq:mm15}) is  
\[   M_B(\lambda)\frac{\widehat{\overline{\phi}}(\lambda)}{ D(\lambda)}\left[ i \pi \mbox{sign}(\Im\lambda)-B\right] .
\]
This means that 
\[ \alpha := M_B(\lambda)\frac{\widehat{\overline{\phi}}(\lambda)}{ D(\lambda)} \]
is known, and simple algebra shows that
\be \frac{\widehat{\overline{\phi}}(\lambda)}{ D(\lambda)}(1+\alpha\widehat{\psi}(\lambda)) = 
\alpha(i\pi\mbox{sign}(\Im\lambda)-B), 
\label{eq:mm16}
\ee
which determines $\frac{\widehat{\overline{\phi}}(\lambda)}{ D(\lambda)}$ and hence $M_B(\lambda)$ provided the factor 
$1+\alpha\widehat{\psi}(\lambda)$ is not identically zero; equivalently, provided
$i\pi \mbox{sign}(\Im\lambda)-B$ is not zero.

We are therefore left to rule out just one pathological case: the case in which $B=i\pi\mbox{sign}(\Im(\lambda))$
in one half-plane and $\widehat{\overline{\phi}}\widehat{\psi}\equiv 0$ in the same half-plane. This can only happen
if $M_B(\lambda)^{-1}$ is zero in this half-plane, which means that every point in the half-plane is an eigenvalue
of $A_B$ and the corresponding $g_\lambda$ given by
\[ g_\lambda(x) = \frac{1}{x-\lambda}-\frac{\widehat{\overline\phi}(\lambda)}{ D(\lambda)}\frac{\psi(x)}{
x-\lambda} = \frac{1}{x-\lambda} \]
belongs to $L^2(\Rr)$ and also satisfies the conditions to lie in the domain of $A_B$:
\[ i\pi \sign(\Im\lambda) = i\pi \mbox{sign}(\Im\lambda) - \frac{\widehat{\overline{\phi}}(\lambda)\widehat{\psi}
(\lambda)}{ D(\lambda)} = \Gamma_1 g_\lambda = B \Gamma_2 g_\lambda 
 = B \]
(see (6.16) in \cite{BHMNW09}). 
\end{proof}
\begin{Remark} ({\bf Uniqueness of $g_\mu$}). 
An alternative approach can be found by examining the uniqueness of the function $g_\mu$ in $\Sco$
defined in (\ref{eq:mm11}). If we know that the choice of $\sigma(\mu)$ is unique then we can immediately
determine $\widehat{\overline{\phi}}(\mu)/ D(\mu)$, which must be equal to $\sigma(\mu)$. This is determined
by $g_\mu$ if $g_\mu$ is unique with its required properties. We examine this now.
\end{Remark}
\begin{Definition}
The non-uniqueness set is the set
\begin{align}
 \Omega = \left\{ \mu\in\Cc\setminus\Rr\, \Big| \, \exists \sigma_1(\mu)\neq\sigma_2(\mu): \;\;
\frac{1}{x-\mu}+\sigma_j(\mu)\frac{\psi(x)}{x-\mu}\in\Sco, \;\; j = 1, 2\right\}.
\end{align}
Equivalently,
\[ \Omega = \left\{ \mu\in\Cc\setminus\Rr \, \Big| \, \frac{1}{x-\mu}\in\Sco \;\mbox{and}\; 
 \frac{\psi(x)}{x-\mu}\in \Sco \right\} .\]
 \end{Definition}
 We also let $\Omega_{\pm} = \Cc_{\pm}\cap\Omega$. We can ignore the condition
 $ D(\mu)\neq 0$ since it can be removed by taking a closure. We can also assume
 that $\overline{\mathcal{S}}\neq L^2(\Rr)$ since otherwise we know the whole resolvent
 $(A_B-\lambda)^{-1}$, which means we know $A_B$ and hence $M_B$. 
We consider two cases in $\Cc_+$ (the situation in $\Cc_{-}$ is similar):
\begin{description}
\item[(I)] $\Cc_{+}\setminus\Omega_+$ has measure 0;
\item[(II)] $\Cc_{+}\setminus\Omega_+$ has positive measure.
\end{description}
In case (II) the uniqueness set in $\Cc_+$, where we can recover $\widehat{\overline\phi}(\mu)/
 D(\mu)$ immediately from $g_\mu$, will have an accumulation point in $\Cc_{+}$
and thus $\widehat{\overline\phi}(\mu)/ D(\mu)$ is uniquely determined in $\Cc_{+}$, by 
analyticity.

In case (I) we have that for almost all $\mu\in\Cc_{+}$, the function $x\mapsto
(x-\mu)^{-1}$ lies in $\overline{\mathcal{S}}$. However
$ \bigvee_{\Im\mu>0}\frac{1}{x-\mu}$ is the Hardy space $H_2^{-}$, and hence
$\overline{\mathcal S}\supseteq H_2^{-}$. Consider the situation in $\Cc_{-}$. 
If we are in the case $|\Omega_{-}|>0$ then 
\[ \overline{\mathcal S} \supset \bigvee_{\mu\in\Omega_{-}}\frac{1}{x-\mu}
 = H_2^{+}, \]
and so we have proved the following.
\begin{Lemma}
\begin{enumerate}
\item If $\Cc_{\pm}\setminus\Omega_{\pm}$ has measure zero, then
 $\overline{\mathcal S}$ contains $H^{2}_{\mp}$, respectively.
\item If $\Cc_{\pm}\setminus\Omega_{\pm}$ has positive measure then one
can recover $\widehat{\overline{\phi}}(\mu)/ D(\mu)$ uniquely, for 
$\mu\in \Cc_{\pm}$.
\end{enumerate}
\end{Lemma}
\begin{Corollary}
Assume that the function $\widehat{\overline{\phi}}(\mu)/ D(\mu)$ in $\Cc_{+}$ coincides with the analytic continuation of $\widehat{\overline{\phi}}(\mu)/ D(\mu)$ in
$\Cc_{-}$. (This happens,
for instance, if $\phi$ has compact support or is zero on an interval.) Then
either $\overline{\mathcal{S}}=L^2(\Rr)$ or we can reconstruct
$\widehat{\overline{\phi}}(\mu)/ D(\mu)$ in $\Cc\setminus\Rr$ uniquely from
$(A_B-\lambda)^{-1}|_{\overline{\mathcal S}}$.
\end{Corollary}

\section{Determining $\overline{\Sc}$ for the Friedrichs model}\label{section:7}
This section is devoted to a detailed analysis of the space $\overline{\mathcal S}$ for the Friedrichs model. 
We shall demonstrate how different aspects of complex analysis are brought into the problem of determining $\overline{\mathcal S}$
and we compute the defect number 
\[ \mbox{def}(\overline{\mathcal S}) = \dim({\mathcal S}^\perp) \]
for various different choices of the functions $\phi$ and $\psi$ which determine the model. The proofs are almost all in the appendix,
as the calculations are sometimes elaborate.

 We first give a characterisation of the space $\overline{\Sc}$, or more precisely, its orthogonal complement.
 
 \begin{proposition} \label{sperpcrit}
Let $P_\pm$ be the Riesz projections defined in \eqref{Riesz} and $D(\lambda)$ be as in (\ref{eq:Ddef}). Denote by $D_\pm(\lambda)$ its restriction to $\C_\pm$ and by $D_\pm$ the boundary values of these functions on $\R$ (which exist a.e., cf.~\cite{Koosis,Pri50}).
\begin{enumerate}
	\item  Let $\phi,\psi\in L^2$.
	\ben\label{eq:65}
 g\in\overline{\Sc}^\perp & \Longleftrightarrow & \begin{cases} P_+ \og-\frac{2\pi i}{D_+}(P_+\ophi)P_+(\psi\og)=0,\\ 
 P_- \og+\frac{2\pi i}{D_-}(P_-\ophi)P_-(\psi\og)=0.  \end{cases}\\ \label{eq:66}
  & \Longleftrightarrow & \begin{cases} \mathrm{(i)}\ \frac{(P_+\ophi)P_+(\psi\og)}{D_+}\in H_2^+, \\
  \mathrm{(ii)}\ \frac{(P_-\ophi)P_-(\psi\og)}{D_-}\in H_2^-, \\
  \mathrm{(iii)}\ \og-\frac{2\pi i}{D_+}(P_+\ophi)P_+(\psi\og)+\frac{2\pi i}{D_-}(P_-\ophi)P_-(\psi\og)=0\ (a.e.).
   \end{cases}
   \een
   \item If $\phi\in L^2,\psi\in L^2\cap L^\infty$ or $\phi,\psi\in L^2\cap L^4$, then
   \ben\label{eq:67}
   g\in\overline{\Sc}^\perp & \Longleftrightarrow & \left[ D_+- 2\pi i(P_+\ophi)\psi\right] \og = 2\pi i\ophi[\psi P_-\og-P_-(\psi\og)]\ (a.e.), \\ \label{eq:68}
   & \Longleftrightarrow & \left[ D_+- 2\pi i(P_+\ophi)\psi\right] \og = 2\pi i\ophi[-\psi P_+\og+P_+(\psi\og)]\ (a.e.), \\ \label{eq:69}
   & \Longleftrightarrow & \left[ D_+- 2\pi i(P_+\ophi)\psi\right] \og = 2\pi i\ophi[P_+(\psi P_-\og)-P_-(\psi P_+\og)]\ (a.e.). 
   \een
\end{enumerate}
\end{proposition}

\begin{remark}
\begin{enumerate}
	\item The second part of the proposition allows us to replace all three conditions (i)-(iii) in \eqref{eq:66} with a single pointwise condition under mild extra assumptions on $\phi$ and/or $\psi$.
		\item Note that the operator $[P_+(\psi P_-\og)-P_-(\psi P_+\og)]$ in the last characterisation of $\overline{\Sc}^\perp$ is the difference of two  Hankel operators.
\end{enumerate}
\end{remark}

As an immediate consequence of \eqref{eq:67}, we get
\begin{theorem}\label{Schar} Assume $\phi\in L^2,\psi\in L^2\cap L^\infty$ or $\phi,\psi\in L^2\cap L^4$.
Define the operator $L$ on $L^2(\R)$ by 
\be\label{eq:70}
Lu = [-P_+(\psi\ophi)+P_+(\ophi)\psi] u +\ophi[\psi P_--P_-\psi] u
\ee 
with the maximal domain $D(L)=\{u\in L^2(\R): \ Lu\in L^2(\R)\}$.\footnote{Under our assumptions, for any $u\in L^2$ we have $Lu\in L^1$ (where we mean the expression $L$ in \eqref{eq:70} rather the operator $L$).} 
Then $\overline{\Sc}\neq L^2(\R)$ iff $1/(2\pi i)\in\sigma_p(L)$ and $\overline{\Sc}^\perp = \ker(L-1/(2\pi i))$.

Furthermore, let $\eta\in L^\infty(\R)$ be a function such that $\eta(k)\neq 0\ a.e.$ and $\eta[-P_+(\psi\ophi)+P_+(\ophi)\psi]$, $\eta\psi\ophi$, $\eta\ophi\in L^\infty(\R)$.  Define the operator $\cL$ on $L^2(\R)$ by
\be
\cL u = \eta\left[-\frac{1}{2\pi i}-P_+(\psi\ophi)+P_+(\ophi)\psi\right] u +\eta\ophi[\psi P_--P_-\psi] u
\ee
with dense domain $D(\cL)=\{u\in L^2(\R): \eta\ophi P_-(\psi u)\in L^2(\R)\}$.
Then $\overline{\Sc}\neq L^2(\R)$ iff $0\in\sigma_p(\cL)$. Moreover, $\overline{\Sc}^\perp = \ker \cL$. Note that if $\psi\in L^\infty$, then $D(\cL)=L^2(\R)$.
\end{theorem}

\begin{remark}\label{rem:Salpha}
Replacing $\psi$ by $\alpha\psi$, we denote the corresponding detectable subspace by $\Sco_\alpha$. Then,
under the conditions in the second part of Proposition \ref{sperpcrit}, we get $g\in \Sc_\alpha^\perp$ iff
		$$\frac{1}{2\pi i \alpha}\og =[-P_+(\psi\ophi)+P_+(\ophi)\psi]\og +\ophi [P_+\psi P_- -P_-\psi P_+]\og =L\og,$$
		where the right hand side is the sum of a multiplication operator and the difference of two Hankel operators multiplied by $\ophi$. As in the theorem, we then get  $\Sc_\alpha^\perp\neq \{0\}$ iff $1/(2\pi i\alpha)\in \sigma_p(L)$ and $\Sc_\alpha^\perp$ is given by the corresponding kernel.
\end{remark}

We now consider various special cases that illustrate the different situations that can arise. The proofs of the results can be found in the appendix.

\subsection{\change{Results depending on the support of $\phi$ and $\psi$}}
\subsubsection{\change{The case of disjoint supports}} In this part we assume that $\phi\cdot\psi=0$ almost everywhere, in particular, $D(\lambda)\equiv 1$, and that either
\be\label{starcondition}
 \phi\in L^2 \hbox{ and } \psi\in L^2\cap L^\infty  \hbox{ or } \phi,\psi\in L^2\cap L^4.
\ee
In some cases (which will be mentioned in the text), we will require the slightly stronger condition 
\be\label{epsiloncondition}
 \phi\in L^2 \cap L^{2+\eps} \hbox{ for some } \eps>0 \hbox{ and } \psi\in L^2\cap L^\infty \hbox{  or } \phi,\psi\in L^2\cap L^4.
\ee

\begin{theorem}\label{trivS} Define the sets  $\Omega_{\psi}=\{x\in\R: \psi(x)\neq 0\}$, $\Omega_{\phi}=\{x\in\R: \phi(x)\neq 0\}$\footnote{Note that the sets $\Omega_\psi$, $\Omega_\phi$ are only defined up to a set of Lebesgue measure zero, but this is sufficient for our purpose. Also, they may be much smaller than the support of the functions.} and $\Omega=\R\setminus(\Omega_\psi \cup \Omega_\phi)$.
Then $g\in \Sco^\perp$ iff $g\in L^2$ and
%\begin{enumerate}
	%\item %$k\in\Omega_{\phi}$: we have $0= \overline{g}(k)-\overline{\phi}(k)\widehat{\psi\chi_{\Omega_{\psi}} \overline{g}}(k+i0)$. Hence,
	\be\label{cond1}
	\chi_{\Omega_{\phi}}  \overline{g}(k)= \overline{\phi}(k)\widehat{\psi\chi_{\Omega_{\psi}} \overline{g}}(k+i0),
	\ee
	%and so $g\vert_{\Omega_{\psi}}$ completely determines $g\vert_{\Omega_{\phi}}$.
	%\item  %$k\in\Omega_{\psi}$: we have  $0= \overline{g}(k)-\widehat{\overline{\phi}}(k-i0)\psi(k)\overline{g}(k).$ Hence, almost everywhere we have
	\be\label{cond2} 
	\chi_{\Omega_{\psi}}\overline{g}(k)\left(1-\widehat{\overline{\phi}}(k-i0)\psi(k)\right)=0,
	\ee
	 %$k\in\Omega$: we have
	\be\label{cond3}
	g\vert_{\Omega}=0.
	\ee
Let  $\Omega_{\psi,0} =\{k\in \R:\ \widehat{\overline{\phi}}(k-i0)\psi(k)=1\}.$ Then, additionally,
\begin{enumerate}
	\item[(i)] if $\Omega_{\psi,0}$ has zero measure, then $\overline{\Sc}=L^2(\R)$.
	\item[(ii)] Assume \eqref{epsiloncondition}.  If $\Omega_{\psi,0}$ has non-zero measure, then  $\overline{\Sc}^\perp\neq\{0\}$. Moreover, we have that 
		$$\Sco \subseteq \{ f\in L^2(\R):\  f = \psi\left(\widehat{f\ophi}\right)_- \hbox{ on } \Omega_{\psi,0}\}.$$
	\item[(iii)] 	Assume $\phi,\psi\in L^2\cap L^\infty$. Then 
	$$\Sco = \{ f\in L^2(\R):\  f = \psi\left(\widehat{f\ophi}\right)_- \hbox{ on } \Omega_{\psi,0}\}.$$
\end{enumerate}
\end{theorem}

\begin{remark}
\begin{enumerate}
\item Note that, from \eqref{cond1} we have that $g\vert_{\Omega_{\psi}}$ completely determines $g\vert_{\Omega_{\phi}}$.
\item Condition \eqref{cond1} gives an additional restriction on  $g\vert_{\Omega_{\psi}}$, requiring $\overline{\phi}(k)\widehat{\psi\chi_{\Omega_{\psi}} \overline{g}}(k+i0)\in L^2(\Omega_\phi)$.
\item The condition \eqref{epsiloncondition} implies (via the H\"older inequality and boundedness of $P_+: L^p \to L^p$ for $1<p<\infty$) that for $g\in (L^2\cap L^\infty) (\Omega_\psi)$ we have $\overline{\phi}(k)\widehat{\psi\chi_{\Omega_{\psi}} \overline{g}}(k+i0)\in L^2(\Omega_\phi)$.
\end{enumerate}
\end{remark}

\begin{example}\label{ex:char}
Let $I$ and $I'$ be disjoint closed intervals in $\R$. Choose $\phi\in L^\infty$ with $\supp\ \phi\in I$ such that 
$$ \int_{\R}\frac{\overline{\phi(x)}}{x-k} dx\neq 0 \quad \hbox{  for  } \quad k\in I'.$$
Define $\psi$ by 
$$ \psi(k)=\begin{cases} \left( \int_{\R}\frac{\overline{\phi(x)}}{x-k} dx\right)^{-1}, & k\in I' \\ 0 & \hbox{ otherwise. }\end{cases}$$
Then $\psi\in L^\infty$,  $\phi\cdot\psi=0$ and $\widehat{\overline{\phi}}(k-i0)\psi(k)= 1$. Therefore, by Theorem \ref{trivS}, (ii),  $\overline{\Sc}^\perp\neq\{0\}$.
\end{example}

\begin{thm}\label{thm:index}
Assume \eqref{epsiloncondition}. Then $\overline \Sc =L^2 \Longleftrightarrow 1 -\psi 
 (\widehat{ \overline \phi})(k+i0) \neq 0$ for a.e. $k\in {\mathbb R}$.
Moreover,
$$
\mathrm{def}\ {\Sc}= \dim \Sc^\perp = \left \{ 
\begin{array}{cc}
0 & \mbox{ if }\overline \Sc=L^2,\\
\infty & \mbox{otherwise}.
\end{array}
\right .
$$
\end{thm}

We have the following results on complete detectability, ie.~$\Sco=L^2(\R)$. 
\begin{thm}\label{generic}
\begin{enumerate}
	\item Complete detectability is generic in the following sense. Replace $\psi$ by  $\alpha \psi$ for $\alpha\in \C$  (or $\phi$ by $\overline{\alpha}\phi$), then for all $\alpha$ outside a countable set $E_0$ we have $\overline \Sc_\alpha=L^2(\R)$, where $\Sco_\alpha$ is as defined in Remark \ref{rem:Salpha}.
 \item For small perturbations, we have
 that if $\psi\in L^\infty$ and $P_+\phi$ or $P_-\phi\in L^\infty$,  replacing $\psi$ by $\alpha\psi$ where $\alpha\in\C$, we get $\overline{\Sc}_\alpha=L^2(\R)$ for sufficiently small $|\alpha|$.
\end{enumerate}
\end{thm}

We next  analyse a specific example with disjoint supports where we consider some questions in more detail.

\begin{theorem}\label{unsymm} 

Assume $I$ and $I'$ are disjoint closed intervals such that $I'$ lies to the left of $I$. Let $$\phi=\chi_I \quad \hbox{ and } \quad 
\psi(x)=\chi_{I'}(x)\cdot\left(\int_I\frac{dt}{t-x}\right)^{-1}.$$

\begin{enumerate}
	\item In this case $\overline{\Sc}\neq L^2(\R)$ with $\mathrm{def}\ {\overline{\Sc}}=\infty$, while $\overline{\Sct} = L^2(\R)$.
\item 
The jump of the $M$-function across the real axis at $k$ is given by 
$$[M_B^{-1}(k)]=\begin{cases}2\pi i, & k\in\R\setminus(I\cup I'),\\ 2\pi i(1-\widehat{\psi}(k)), & k\in I, \\ 0, & k\in I'.\end{cases}$$
Moreover, the bordered resolvent $P_{\Scto}(A_B-\lambda)^{-1}P_\Sco$ jumps at $k\in\R$ iff $k\notin I'$, i.e.~the location of the jumps of the  bordered resolvent coincides with the jumps of $M_B$. (Compare to Theorem \ref{thm:5.2}, where we can only border the resovent by finite dimensional projections.)
\end{enumerate}
\end{theorem}

\subsubsection{\change{Results when $\phi$ and $\psi$ are not disjointly supported} }
We consider a more general case when $\phi$ and $\psi$ do not necessarily have disjoint supports.
\begin{theorem}\label{pos}
%Assume $\phi\cdot\overline{\psi}\geq0$. 
%Then 
\begin{enumerate}
	%\item $D(\lambda)= 1 + \int_{\Rr}(x-\lambda)^{-1}\psi\overline{\phi}dx$ is a Herglotz function and can only have zeroes on the real line.
	\item Let $\Omega^c=\{x: \phi(x)\neq 0\} \cup \{x: \psi(x)\neq 0\}$. Then  $g\in \Sc^\perp$ implies $\{x: g(x)\neq 0\}\subseteq\Omega^c$ (up to a set of measure zero). In particular, $\Sc^\perp \subseteq L^2(\Omega^c)$ and $\Sco\supseteq L^2(\Omega)$.
	\item Consider $\psi=\alpha\chi_I$ for some constant $\alpha$ and a set $I$ of finite measure and assume $ \phi\vert_{I^c}=0$ a.e. Then  $\overline{\Sc}=L^2(\R)$.
\end{enumerate}
\end{theorem}

We finish this subsection by showing that in a our situation we can improve on Theorems \ref{Munique}, \ref{thm:two} and \ref{thm:1} by  recovering the $M$-function from one bordered resolvent.

\begin{theorem}\label{oneborderedres}
As before, let  $\Omega_{\psi}=\{x\in\R: \psi(x)\neq 0\}$ and  $\Omega_{\phi}=\{x\in\R: \phi(x)\neq 0\}$. Let $\Omega=\R\setminus(\Omega_\phi \cup\Omega_\psi)$ and assume
that we know a set of non-zero measure $\Omega'\subseteq\Omega$. Then the $M$-function can be recovered from one bordered resolvent. 
\end{theorem}

\begin{remark}
The converse is not possible. The asymptotics of the $M$-function at $i\infty$ allow us to recover $B$ and thus $\widehat{\psi}(\lambda)\widehat{\overline{\phi}}(\lambda)$ for any $\lambda$. However, only knowing the product for example makes it impossible to distinguish the expression for $A$ from the operator expression obtained by replacing $\psi$ by $\ophi$ and $\phi$ by $\overline{\psi}$, respectively.
%, however the bordered resolvent seems to contain enough information to do this.
\end{remark}

\subsection{Results with $\phi, \psi \in H_2^+$}

We note that in the Fourier picture described in Remark \ref{Fourier}, the condition that $\phi, \psi \in H_2^+$ corresponds to $\cF\phi, \cF\psi$ being supported in $\R^-$ (by the Paley-Wiener Theorem \cite{Koosis}). A similar remark applies to the next subsection when $\overline{\phi},\psi\in H_2^+$ where the Fourier transforms will be supported on different half lines.
Moreover, similar results will hold if both $\phi, \psi \in H_2^-$.

\begin{theorem}\label{S++} Let $\phi, \psi \in H_2^+$. Then 
$$\Sco=\overline{\bigvee_{\mu \in \C^+} \dfrac1{x-\mu} + \bigvee_{\mu \in \C^-} \dfrac{  D(\mu) + 2 \pi i \bar \phi(\mu) \psi(x)}{x-\mu}}.$$

Moreover, if 
\be\label{eq:rationalpsi}\psi(x)= \sum_{j=1}^N \frac{c_j}{x-z_j}\ee
with $c_j\neq 0$, $\Im z_j<0$ and $z_i\neq z_j$ for $i\neq j$, then 
$$\ind(\change{\Sc})=N-P-M-M_0,$$
where $P=\sum p_k$ and $p_k$ is the order of poles of $\overline{\phi(\overline{\mu})}/D_+(\mu)$%\footnote{Here, $D_+(\mu)$ referes to the analytic continuation of $D_+$ to $C_-$.}
 in $\C_-\setminus\{z_j\}_{j=1}^N$, $M=\sum m_i$, where $m_i$ are the `order of the poles' of $\overline{\phi(x)}/D_+(x)$ in $\R$ (i.e.~$m_i$ is the minimum integer such that $\overline{\phi(x)}/D_+(x) (x-x_i)^{m_i}$ is square integrable), $M_0$ corresponds to a degenerated case and is given by
$$M_0=\left\vert\{j:\overline{\phi(\overline{z_j})}=0\hbox{ and } \lim_{\mu\to z_j}\dfrac{ 2 \pi i \bar \phi(\mu) c_j}{D_+(\mu)(\mu-z_j)} \neq 1\}\right\vert$$
and 
$$D_+(\mu):=1+2\pi i\sum_{j=1}^N\frac{c_j\overline{\phi(z_j)}}{\mu-z_j}$$
for $\mu\in\C$ is a meromorphic continuation of the rational function $D(\mu)$, $\mu\in\C_+$ to the lower half plane. (Note that this will not coincide with $D(\mu)$ in the lower half plane and that for generic $\phi\in H_2^+$ the continuation of $D_-$ to $\C_+$ will not even exist).

\end{theorem}

%\textcolor{red}{We need to put an outline proof in the Appendix.}

\begin{proposition} \label{proposition:8.13} Any values can be realized for the defect numbers of $\Sc$ and $\tilde{\Sc}$ by suitably choosing 
 rational $\phi$ and $\psi$ in $H_2^+(\R)$.
\end{proposition}
%\textcolor{red}{Proof to be inserted into the Appendix.}

%\textcolor{red}{

We conclude this part with an example. 
\begin{Example}\label{specialcase}
Let
$$\psi(x)=\dfrac \alpha {x-z_1} \hbox{ with } z_1\in \C_-, \alpha \in\C\setminus\{0\} \quad \hbox{ and }\quad\overline \phi(x) =\dfrac 1{x-w_1} \hbox{ with } w_1\in \C_+.
$$
The root of $D(\lambda)$ in $\C_+$ or its analytic continuation $D_+(\lambda)$ in $C_-$ is $\lambda_0= z_1+\frac{2\pi i\alpha }{w_1-z_1}$. We have three cases for $N,P,M,M_0$ as in Theorem \ref{S++}:
\begin{enumerate}
	\item If $\lambda_0\in\C_+$
	 then $N=1$, $P=M=M_0=0$,
	\item if $\lambda_0\in\C_-$ then $N=P=1$, $M=M_0=0$,
	\item if $\lambda_0\in\R$ then $N=M=1$, $P=M_0=0$.
\end{enumerate}
Therefore, $\Sc^\perp$ is non-trivial if and only if $\lambda_0=z_1+\frac{2\pi i\alpha }{w_1-z_1}\in\C_+$. In this case, $\Sc^\perp$ is one dimensional. Moreover,
\be 
\Sc^\perp=\left\{\frac{const}{(t-\overline{w_1})(t-\overline{z_1}+\frac{2\pi i\overline{\alpha} }{\overline{w_1}-\overline{z_1}})}\right\}\quad\hbox{ and }\quad \Sco=\{f\in L^2(\R):\ (P_+f)(w_1)=(P_+f)(\lambda_0)\}.
\ee
Similarly, $\Sct^\perp$ is non-trivial if and only if $\widetilde{\lambda_0}:= w_1+\frac{2\pi i \alpha }{w_1-z_1}\in\C_- $ (and therefore $D(\widetilde{\lambda_0})=0$). Note that if $\lambda_0\in\C_+$, then also $\widetilde{\lambda_0}\in \C_+$, whilst if $\widetilde{\lambda_0}\in\C_-$, then also $\lambda_0\in\C_-$. Therefore, at least one of $\Sco$ and $\Scto$ is the whole space. 

Moreover, we see that the bordered resolvent does not detect the singularities at the eigenvalues $\lambda_0\in\C_+$ or $\widetilde{\lambda_0}\in\C_-$: 

For $\lambda\approx \lambda_0\in\C_+$ we have from\eqref{eq:mm1c} and \eqref{eq:mm10} that
\be (A_B-\lambda)^{-1}=\hbox{regular part at }\lambda_0+\frac{\mathcal{P}_{\lambda_0}}{\lambda-\lambda_0},\ee
with the Riesz projection $\mathcal{P}_{\lambda_0}$ given by
\be
\mathcal{P}_{\lambda_0} = \llangle\cdot, u_1 \rrangle u_2,
\ee
where 
\be u_1=\frac{\phi}{x-\overline{\lambda_0}} \quad \hbox{ and } \quad u_2=\alpha (z_1-\lambda_0)\frac{\psi(x)-2\pi i M_B(\lambda_0)\psi(\lambda_0)}{x-\lambda_0}.
\ee
Since $u_1\in\Sc^\perp$, the singularity is cancelled by $P_\Sco$. $u_2$ is the eigenvector of $A_B-\lambda_0$ (see \cite{Kato}).

For $\lambda\approx \widetilde{\lambda_0}\in\C_-$ we have again from\eqref{eq:mm1c} and \eqref{eq:mm10} that
\be (A_B-\lambda)^{-1}=\hbox{regular part at }\lambda_0+\frac{\mathcal{P}_{\widetilde{\lambda_0}}}{\lambda-\widetilde{{\lambda_0}}},\ee
with 
\be
\mathcal{P}_{\widetilde{\lambda_0}} = \llangle\cdot, (\overline{\widetilde{\lambda_0}}-\overline{w_1})\frac{\phi(x)-2\pi i \overline{M_B(\widetilde{\lambda_0})}\ophi(\widetilde{\lambda_0})}{x-\widetilde{\lambda_0}} \rrangle \frac{\alpha\psi}{x-\overline{\widetilde{\lambda_0}}} ,
\ee
where $\frac{\psi}{x-\widetilde{\lambda_0}}$ is an eigenvector of $A_B$  for all (!) $B$ and lies in $\Sct^\perp$, so the singularity of the resolvent is cancelled by $P_\Scto$.

We note that this behaviour of the bordered resolvent is in accordance with Theorem 3.6 in \cite{BHMNW09}.

%Finally, we investigate the relation between $M_B$ and the bordered resolvent. Since $M_B(\lambda)=[\pi i \sign (Im (\lambda))-B]^{-1}$ it only contains information about the boundary condition. However, from our calcualtions above (be more specific), we have that if $\lambda\in\C_+$ with $D(\lambda)=0$, then 
%$$ u = \frac{c_u{\mathbf 1}-\tilde{c}\psi}{x-\lambda} \in \ker(\At^*-\lambda) $$
%for any $c_u,\tilde{c}\in\C$. Noting that
%$$\Gamma_1 u = \int_\R u = -\tilde{c}\int \frac{\psi}{x-\lambda} \quad \hbox{ and } \quad \Gamma_2 u =c_u,$$
%we have 
%$$(\Gamma_1-B\Gamma_2)u = -\tilde{c}\int \frac{\psi}{x-\lambda}-Bc_u.$$
%By choosing $c_u,\tilde{c}$ suitably, we see that such a $\lambda$ is an eigenvalue of any realisation $A_B$. 
%
%Since from \eqref{kernel} we know that elements of $\Sc^\perp$ and $\Sct^\perp$ are of the form 1/quadratic polynomial, the functions $u$ are not elements of $\Sc^\perp$ and $\Sct^\perp$ and will not be cancelled by the projections. It appears therefore that the bordered resolvent still detects these, while the $M$-function clearly does not.
\end{Example}

\begin{remark}\label{rem:Borg}
We note that in the case when $\phi,\psi\in H_2^+$ taking $\lambda,\mu\in \C_+$, the $M$-function and the ranges of the solution operators $S_{\lambda,B}$ and $\widetilde{S}_{\mu,B^*}$ do not depend on $\phi$ and $\psi$ (see \eqref{eq:9} and \eqref{eq:7}). In fact, $M_B(\lambda) = \left[ \mbox{sign}(\Im\lambda) \pi i - B\right]^{-1}$, $S_{\lambda,B}f= (\Gamma_2f) (x-\lambda)^{-1}$ and $\widetilde{S}_{\mu,B^*}f= (\Gammat_2 f) (x-\mu)^{-1}$.  In this highly degenerated case, 
only the boundary condition $B$ can be obtained. Therefore, in this case a Borg-type theorem allowing recovery of the bordered resolvent  from the $M$-function is not possible, even with knowledge of the ranges of the solution operators in the whole of the suitable half-planes.
On the other hand, knowledge of the ranges of the solution operators in both half-planes, together with the $M$-function at one point allows reconstruction by Theorem \ref{thm:recon}.
\end{remark}

\subsection{Analysis for the case $\overline{\phi},\psi\in H_2^+$.}

\begin{theorem}\label{notminuspi} Let $\overline{\phi},\psi\in H_2^+$.
Then  if $B\neq -\pi i$, we have $\ind({\mathcal S_B}) = 0$.
Similar results hold for $\Sct_B$ by taking adjoints.
\end{theorem}

%\textbf{[REMARK: I've now included the much shorter new proof of the fact here. Note that this seems to require the extra condition  $\phi\in L^2,\psi\in L^2\cap L^\infty$ or $\phi,\psi\in L^2\cap L^4$. The lengthy old proof without the extra assumption is currently still in the appendix. If we do want to keep it, we'd have to check that it still holds with the `new' characterisation of $\Sc$. One advantage of it is that it motivates the following remark.] }

\begin{remark}\label{minuspi}
We note that the space $\Sco_B$ as defined in \eqref{eq:mm1} can  depend on the boundary condition. In the degenerate case
$B=-\pi i$ we have
\[ \Sco_B = \bigvee_{\mu\in {\mathbb C}^+} \left(\frac{D(\mu) - 2\pi i \overline{\phi}(\mu)\psi(x)}{x-\mu} \right). \] 
If $\phi$ or $\psi$ additionally lies in $L^\infty$, then this gives $\mathrm{def}(\Sc_B)=+\infty$.
However, we consider this choice of $B$ as a degenerate case, since the hypotheses of Proposition \ref{prop:2.7} are not satisfied.
\end{remark}

\subsection{The general case $\psi, \phi\in L^2$}

We conclude this section by studying the general case. Without assumptions on the support, or the Hardy class of $\phi$ and $\psi$, the results are rather complicated. Therefore, in what follows we will not worry about imposing slightly stronger regularity conditions on $\phi$ and $\psi$. %In particular, we will always at least assume that $\phi\in L^2,\psi\in L^2\cap L^\infty$ or $\phi,\psi\in L^2\cap L^4$, so that $\overline{\Sc}^\perp$ can be characterised as in part (2) of Proposition \ref{sperpcrit}.
We first define the following set
 \ben \label{E0}
  E_0&:=& \{ \alpha\in\C :
 \exists \hbox{ a set of positive measure }E\subseteq\R  \\ 	&&  \hbox{ s.t. } 1 + 2\pi i \alpha (P_+(\overline \phi \psi)-\psi (P_+(\overline \phi)))= 0 \hbox{ on }E\}. \nonumber
	\een
Note that $E_0$ consists  of those $\alpha$ such that the factor $\left[ D_+- 2\pi i(P_+\ophi)\psi\right]$ appearing in \eqref{eq:67} - \eqref{eq:69} vanishes on some non-null set $E$ when $\psi$ is replaced by $\alpha\psi$.

\begin{theorem}\label{thm:E0}
Assume \eqref{epsiloncondition}. Let $\alpha\in E_0$, then $\mathrm{def}\  \Sc_\alpha =+\infty$.
\end{theorem}

\begin{remark}
When considering the corresponding $\Sct_\alpha$ note that the set
$$
  \widetilde{E}_0:= \left\{ \alpha: 1 + 2\pi i \overline{\alpha} (P_+(\overline \psi \phi)-\phi (P_+(\overline \psi)))= 0\hbox{ on a set of positive measure}\right\}$$ 
	does not need to coincide with $E_0$, so it is possible to have $\ind \Sc_\alpha \neq \ind\Sct_\alpha$ even for $\alpha\in E_0$. 
	
	This happens in Example \ref{ex:char}, where (if $\psi$ is multiplied by a suitable constant) $\ind \Sc_\alpha =\infty$, while $\ind\Sct_\alpha=0$. 
\end{remark}

 \begin{theorem}\label{thm:indexx} Let $\phi\in L^2\cap L^\infty$ and $\psi\in L^2\cap C_0(\R)$, where $C_0(\R)$ is the space of continuous functions vanishing at infinity, and assume $\alpha\not\in E_0$. 
\begin{enumerate}
	\item Then $\ind \Sc_\alpha>0$ iff $(2\pi i \alpha)^{-1}\in \sigma_p(\M+\K)$, where
\be
\M= \left((P_+\overline \phi)\psi -P_+(\psi \overline \phi )\right)
\ee is a possibly unbounded multiplication operator  and
\be
\K=\ophi\left[ P_+ \psi P_- - P_- \psi P_+  \right]
\ee is the difference of two compact Hankel operators multiplied by $\ophi$. Note that $D(\M+\K)=D(\M)$, where $D(\M)$ is the canonical domain of the multiplication operator.

Moreover, $$\Sc_\alpha^\perp=\ker\left(\M+\K-\frac{1}{2\pi i\alpha}\right),$$ 
so
$$\ind \Sc_\alpha =\dim \ker\left(\M+\K-\frac{1}{2\pi i\alpha}\right).$$
If $(2\pi i \alpha)^{-1}\notin\overline{\essran_{k\in\R} \M(k)}$, then 
$$\ind \Sc_\alpha =\dim \ker\left(I+\K\left(\M-\frac{1}{2\pi i\alpha}\right)^{-1}\right) < \infty.$$

\item Additionally assume $\M(k)$ is continuous. Then $\C\setminus\overline{\Ran\M(k)}$ is a countable union of disjoint connected domains. Set $\mu=(2\pi i \alpha)^{-1}$. Then in each of these domains we have either
\begin{enumerate}
	\item $\ind\Sc_\alpha=0$ whenever $\mu$ is in this domain except (possibly) a discrete set, or
	\item $\ind\Sc_\alpha\neq 0$ is finite and constant for any $\mu$ in the domain except (possibly) a discrete set.
\end{enumerate}
Moreover, for $\mu$ sufficiently large, we have $\ind\Sc_\alpha=0$.
\end{enumerate}
\end{theorem}

Although this theorem gives a description of $\Sco_\alpha$ for a rather general case of $\psi$ and $\phi$, for concrete examples as investigated in previous subsections it is useful to   determine the space explicitly rather than just give the description in terms of operators $\K$ and $\M$. However, this theorem shows the topological properties of the function $\ind \Sc_\alpha$ in the $\alpha$-plane.

\begin{example}\label{petal}
Let
$$\psi(x)=\alpha \left(\frac{c_1}{x-z_1}+\frac{c_2}{x-z_2}\right) \hbox{ with } z_1\neq z_2\in \C_-, \alpha \in\C\setminus\{0\}$$ and 
$$\overline \phi(x) =\dfrac 1{x-w_1} \hbox{ with } w_1\in \C_+.
$$
We wish to analyse the defect as a function of $\alpha$. By Theorem \ref{S++}, we need to determine the number of roots of the analytic continuation $D_+(\lambda)$ of $D(\lambda)$  in $\C_-$. 
Now,
\be\label{Dplus}
D_+(\lambda)= 1-2\pi i\alpha\left( \frac{c_1}{(z_1-w_1)(z_1-\lambda)} +  \frac{c_2}{(z_2-w_1)(z_2-\lambda)} \right).
\ee
After setting $\hat\mu:=\frac{2\pi i\alpha}{(z_1-w_1)(z_2-w_1)}$ a short calculation shows that the roots of $D_+(\lambda)$ solve
\be
\lambda^2+\lambda(d_1\hat\mu-z_1-z_2)+d_2\hat\mu+z_1z_2=0,
\ee
where 
$$d_1= c_1(z_2-w_1)+c_2(z_1-w_1) \quad\hbox{ and }\quad d_2= -c_1z_2(z_2-w_1)-c_2z_1(z_1-w_1).$$
In particular, for $\hat\mu=0$ the roots are $z_1,z_2\in \C_- $. By continuity, for small $|\alpha|$, by Theorem \ref{S++} we have $ \ind \Sco_\alpha=0 $.

For a polynomial $ \lambda^2+p\lambda+q=0$, an elementary calculation shows that it has a real root iff
\be\label{poly} (\Im q)^2 = (\Im p)\left(\Re p \Im q -\Re q\Im p\right)\quad\hbox{ and }\quad 4\Re q \leq |p|^2 .\ee

We now analyse the defect in a few examples. %The results presented always only hold up to a discrete set (as discussed in Theorem \ref{thm:indexx}).
\begin{enumerate}
	\item[(A)] We first make the specific choice
$$\psi(x)=\alpha \left(\frac{-2}{x+i}+\frac{3}{x+2i}\right) \quad \hbox{ and }\quad\overline \phi(x) =\dfrac 1{x-i}.
$$
Then $d_1=0, d_2=-6$ and the equation in \eqref{poly} becomes 
\be\label{eq:onepetal}
 (\Im \hat\mu)^2 = \frac12(1+3\Re\hat\mu).
\ee
All $\hat\mu$ satisfying \eqref{eq:onepetal} satisfy the inequality in \eqref{poly}.%, so  $\ind \Sco_\alpha=0$ for all $\alpha\in\C$
This gives a parabola in the $\alpha$-plane (or equivalently the $\hat\mu$-plane) with $\ind \Sco_\alpha=0$ inside or on the parabola and $\ind \Sco_\alpha=1$ outside. In the $1/\alpha$-plane this gives a curve whose interior is a petal-like shape  with   $\ind \Sco_\alpha=0$ for $1/\alpha$ outside or on the curve and $\ind \Sco_\alpha=1$ for  $1/\alpha$ inside the curve.
\item[(B)] We now return to the formula for $D_+$ in \eqref{Dplus}. Setting $\mu=(2\pi i \alpha)^{-1}$, we have 
\be
\mu =  \frac{c_1}{(z_1-w_1)(z_1-\lambda)} +  \frac{c_2}{(z_2-w_1)(z_2-\lambda)}.
\ee 
Clearly for $\lambda \to \pm\infty$, we have that $\mu=0$. 
We now choose $c_1,c_2$ to get another real root at $\lambda=0$.
Consider 
$$\psi(x)=\alpha \left(\frac{-1}{x+i}+\frac{3}{x+2i}\right) \quad \hbox{ and }\quad\overline \phi(x) =\dfrac 1{x-i}.
$$
In the $\mu$-plane this leads to one petal. As $\lambda$ runs through $\R$, this curve is covered twice (once for $\lambda<0$ and once for $\lambda>0$). We have $\ind \Sco_\alpha=0$ for $\mu$ outside the curve  and $\ind \Sco_\alpha=2$ for $\mu$ inside the curve. On the curve we have $\ind\Sco_\alpha =0$. The double covering of the curve allows the jump of $2$ in the defect when crossing the curve.% [INCLUDE PICTURE?]. 
\item[(C)] More generally,  if $\psi$ has $N$ terms, then the problem of finding real roots of $D_+(\lambda)$ leads to studying the real zeroes of
$$\xi(\lambda):=\sum_{k=1}^N \frac{a_k}{z_k-\lambda}, \quad \hbox{ where } \quad a_k=c_k\ophi(z_k).$$
Generically $\xi$ will not have real zeroes and we will only get one petal in the $\mu$-plane. However, we can arrange it that $\xi$ has $N-1$ real zeroes which leads to $N$ petals in the $\mu$-plane. Assume $a_N\neq0$. Then to do this, we need to  solve the linear system,
\be
Z\left(\begin{array}{c}a_1 \\ \vdots \\a_{N-1}\end{array}\right) = \left(\begin{array}{c}-\frac{a_N}{z_N-\lambda_1} \\ \vdots \\-\frac{a_N}{z_N-\lambda_{N-1}}\end{array}\right), 
\ee
where the matrix $Z$ has ${jk}$-component given by $z_{jk}=(z_k-\lambda_j)^{-1}$. $Z$ is invertible whenever all $z_k\in\C_-, \lambda_j\in\R$ are distinct.

\begin{figure}[ht]\vspace{-90pt}
    \includegraphics[width=2.2in,height=4in]{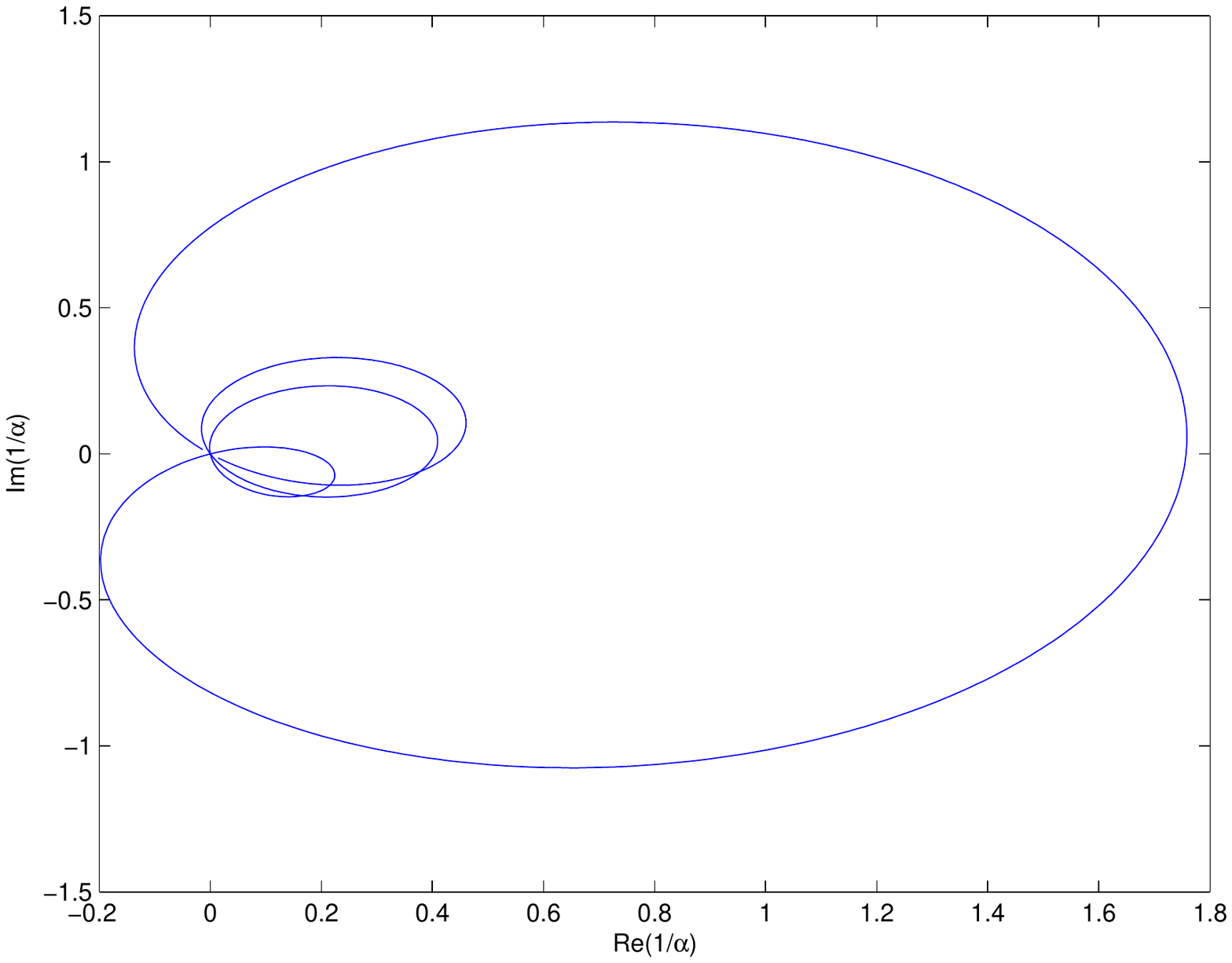} 
    \includegraphics[width=2.6in,height=4in]{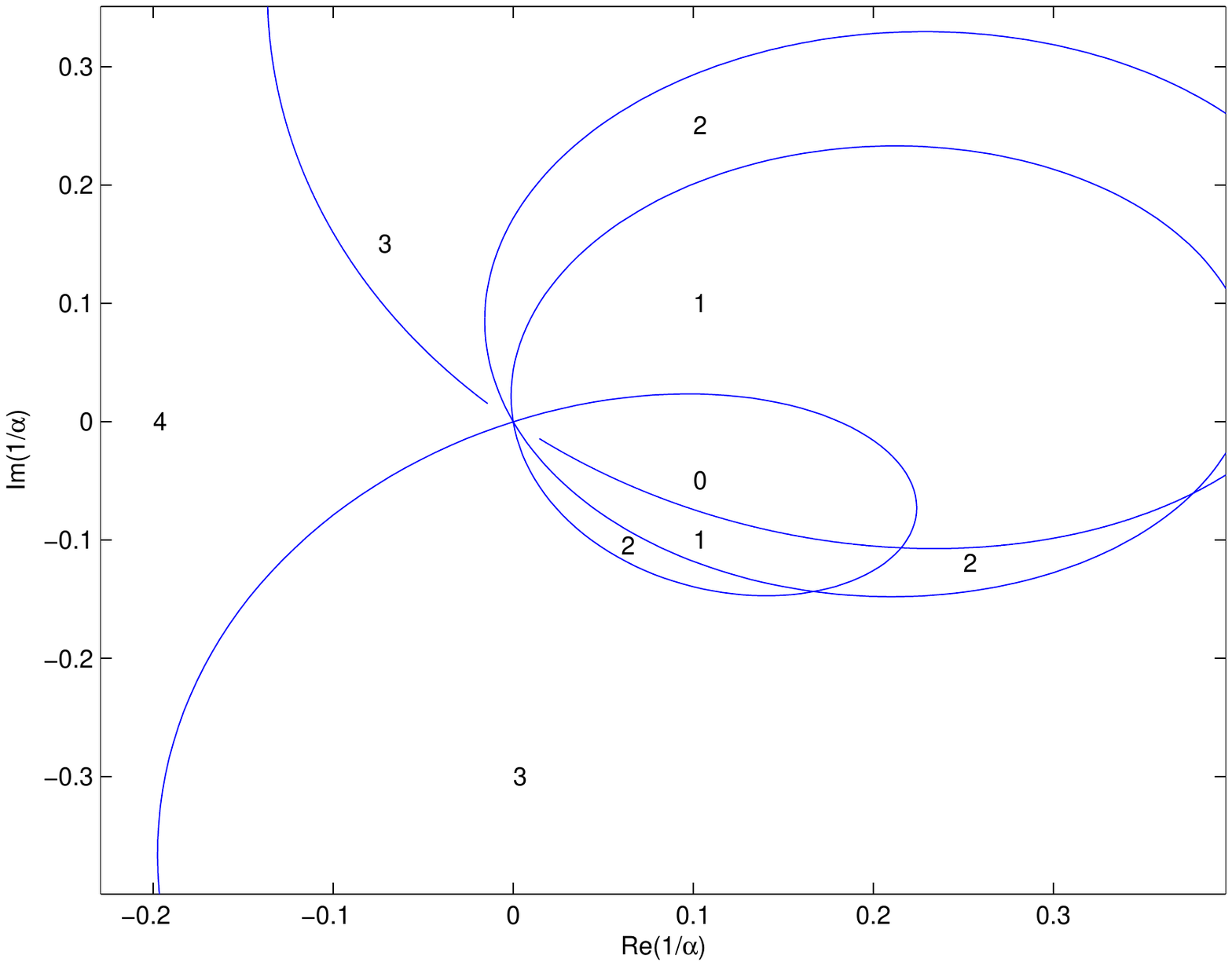} \vspace{-70pt}
  \caption{The curve in the $1/\alpha$-plane along which $D_+$ has a real root for the case $\lambda_1=0$, $\lambda_2=1$, $\lambda_3=-2$, $z_1=-i$, $z_2=1-i$, $z_3=-2-i$, $z_4=3-2i$ and $a_4=1$. On the right, zoom of part of the curve including the number of roots of $D_+$ in $\C_-$ in different components.}
  \label{fig:2}
\end{figure}

For the example in Figure \ref{fig:2}, the defect in each of the components is given by $4-\nu_-$ where $\nu_-$ denotes   the number of roots of $D_+$ in $\C_-$ (by Theorem \ref{S++}). At each curve precisely one of the roots crosses from the lower to the upper half-plane, thus increasing the defect by $1$. On the curve itself, one root is on the real axis and by Theorem \ref{S++}, the defect coincides with the smaller of the defects on the components on each side of the curve. By a similar reasoning at the three non-zero points of self-intersection of the curve the defect coincides with the smallest defect of the neighbouring components.

This example displays the analytical nature of finding the defect in terms of the location of roots of $D_+$ using  Theorem \ref{S++}. 
On the other hand, it also displays the topological nature of the same situation mentioned in  Theorem \ref{thm:indexx}. The complex $1/\alpha$-plane is separated into components in which the defect is constant everywhere (in this example the exceptional discrete set is empty). The curves are the range of $2\pi i \M (t)$ on the real axis.

%[INCLUDE PICTURE: Marco, what does the curve look like for $\lambda_1=0$, $\lambda_2=1$, $\lambda_3=-2$, $z_1=-i$, $z_2=1-i$, $z_3=-2-i$, $z_4=3-2i$ and $a_4=1$?]
\end{enumerate}

%For other choices of $c_1,c_2, w_1, z_1$ and $z_2$ generically (when $d_1\neq 0$) we get a flower with two petals in the $1/\alpha$-plane with   $\ind \Sco_\alpha=0$ for $1/\alpha$ outside the petals  and $\ind \Sco_\alpha=1$ for  $1/\alpha$ inside the petals [INCLUDE PICTURE?].
%
%One should expect generically that in the case of $N$-terms in the function $\psi$, as in \eqref{eq:rationalpsi}, we get a flower with $N$-petals in the $1/\alpha$-plane with   $\ind \Sco_\alpha=0$ for $1/\alpha$ outside the petals  and $\ind \Sco_\alpha\geq 1$ for  $1/\alpha$ inside the petals [INCLUDE PICTURE?].
%
%\textbf{[Request to Marco: It would be great if you could do a computation and produce a picture of the flower with four petals for the case $N=4$ and, say, $w_1=i$, $z_1=-i$, $z_2=-i+1$, $z_3=-1-i$, $z_4=-2i$, $c_i=1$ for $i=1,..4$ and check for a point inside a petal what the roots of $D_+$ are (and hence what the defect is).]}
\end{example}

\section{Spectra of Toeplitz operators and detectability}\label{section:Toep}
In Section \ref{section:7} Hankel operators played a special role in the analysis (see e.g.~Theorem \ref{thm:indexx}). However, for another class of examples of the Friedrichs model, the theory of Toeplitz operators is the main instrument of our analysis. We will discuss this type of examples and the related detectability problem here.
The proofs and some auxiliary results can be found in Appendix \ref{AppB}.

We will study the operator $T:=\overline \phi P_+\psi$ acting on $L^2$. It is closely related to  the Toeplitz operator $T_a:H^+_2\to H^+_2$ given by $T_a u =P_+au=P_+aP_+u$, where $a=\psi\overline \phi$.

\begin{assumptions}\label{AssAppB}
For this section, we will make the following assumptions on the functions $\phi,\psi\in L^2$:
(i) $a(z):=\psi(z)\overline \phi (z) \in H_1^-\backslash\{0\}$,
(ii) $\phi\in H_2^-$ and
(iii) $\phi,\psi\in L^\infty$.
\end{assumptions}

\begin{remark}\label{Rem:AppB}
\begin{enumerate}
	\item Under the above assumptions, we have $D_+(\lambda)\equiv 1$ and $P_+\ophi=\ophi$, $P_-\ophi=0$.
	\item Choosing $\phi=(x-i)^{-\frac12-\eps}$ for some $\eps>0$ and a suitable choice of the branch cut, we have $\phi\in H_2^-$. To satisfy the first assumption, we then require $\psi(x)=a(x)(x+i)^{\frac12+\eps}\in L^2$, or $a\in L^2(\R; (1+x^2)^{\frac12+\eps})$. Therefore, the assumption that $\phi\in H^-_2$ only imposes a mild extra condition on the decay of $a$ at infinity.
	\item The third assumption  is only introduced for the sake of convenience and may be significantly relaxed. However, this would introduce more technical details which would obscure the main results. It means that the operator $T$ is a bounded operator in $L^2$ and  $T_a$ in $H^+_2$, in particular, they are defined on the whole space.
\end{enumerate}
\end{remark}

In the next theorem we will make use of the canonical Riesz-Nevanlinna factorization theorem (see \cite[Pages 199-200]{Koosis}). For the reader's convenience, we state it here for functions in the lower half plane: If $f\not\equiv 0$, $f\in H_p^-$ for $p\geq 1$, then up to constant multiples, $f$ can be factorized uniquely as
\be
 f(z)=B(z)\Sigma(z)G(z),
\ee
where 
\begin{itemize}
	\item $B(z)$ is a Blaschke product with 
	$$B(z) = \prod_k\left( e^{i\theta_k}  \frac{z-z_k}{z-\overline{z_k}}\right) \hbox{ for } \Im z<0 , $$
	where $z_k$ are all the roots of $f$ in $\C_-$ and the real $\theta_k$ are chosen so that $$e^{i\theta_k}\frac{i-z_k}{i-\overline{z_k}}\geq 0;$$
	\item $\Sigma(z)$ is the singular factor given by
	$$\Sigma(z) = \exp\left(\frac{1}{\pi}\int_{-\infty}^\infty \left(\frac{i}{z-t}+\frac{it}{t^2+1}\right) \ d\sigma(t)  \right)$$
	with $d\sigma(t)\geq 0$ a singular measure (w.r.t.~Lebesgue measure) such that $\int_{-\infty}^\infty \frac{d\sigma(t)}{t^2+1}<\infty$;
	\item $G(z)$ is the outer factor given by
	$$G(z)=\exp\left(-\frac{i}{\pi}\int_{-\infty}^\infty \left(\frac{1}{z-t}+\frac{t}{t^2+1}\right)\log|f(t)| \ dt  \right).$$
\end{itemize}

\begin{theorem}\label{B2}
%Let $\phi,\psi\in L^\infty$ and $a(z):=\psi(z)\overline \phi (z) \in H^-_\infty\backslash\{0\}$, so that $P_+(\psi \overline \phi)=0$. 
Define the operators  $T$ on $L^2$ and $T_a:H^+_2\to H^+_2$ as above and let $\mu_\alpha =(2\pi i \alpha)^{-1}$. Then
\begin{enumerate}
\item  $\ind(\Sc_\alpha)= \dim\ker\left(T_a-\mu_\alpha\right)$. Moreover, 
$$\Sc^\perp_\alpha=\left(\ker(T-\mu_\alpha)\right)^*=\overline{\psi^{-1}}[(\mu_\alpha P_+ + P_-aP_+)\ker\left(T_a-\mu_\alpha\right)]^*.$$ Here, we use $\left(\ker(T-\mu_\alpha)\right)^*$ to denote the set of complex conjugates of functions in $\ker(T-\mu_\alpha)$. This will be used to distinguish it from the closure in cases where the meaning may be ambiguous.  \label{Toep_estimate} 
\item %Let $$ E_0= \left\{ \alpha: 1 - 2\pi i \alpha \psi (P_+(\overline \phi))= 0\right\} .$$
%For $\alpha \not \in E_0$ 
We have $\Sc^\perp_\alpha = \phi\left( H_2^-\ominus (a(z)-\mu_\alpha) H_2^-\right)$. In particular, $\phi\left( H_2^-\ominus (a(z)-\mu_\alpha) H_2^-\right)$ is a closed subspace.\label{Salpha}
\item Consider the canonical factorisation  in $\C_-$ of $a(z)-\mu_\alpha=B_\alpha(z) \Sigma_\alpha(z)G_\alpha(z)$ for a fixed $\alpha\in\C$ where $B_\alpha$ is a Blaschke product containing all zeros of  $a(z)-\mu_\alpha$ in $\C_-$, $\Sigma_\alpha$ is the singular part and $G_\alpha$ is an outer function. \label{index}
Then 
$$ \overline{(a-\mu_\alpha)H_2^-}=\overline{ B_\alpha(z) \Sigma_\alpha(z)H_2^-} \equiv B_\alpha(z) \Sigma_\alpha(z)H_2^-. $$
 and %for $\alpha \in \C\setminus E_0$, 
\be\label{Blaschkedefect}  \ind (\Sc_\alpha)=\left \{  
\begin{array}{cc}
 \infty& \mbox    {if } \Sigma_\alpha\not \equiv 1, \\
\mbox{ number of roots of } B_\alpha(z)
\mbox{ counted with multiplicity }  & \mbox{ if } \Sigma_\alpha \equiv 1.
\end{array}
\right . \ee
\end{enumerate}
\end{theorem}

\begin{Example} For fixed $\alpha\in\C$, we consider a canonical decomposition of $(a-\mu_\alpha)$ in $\C_-$ of the following form:
Choose any $B_a$ with zeroes in a finite box and the singular measure in $\Sigma_a$ with bounded support. Then $B_a\Sigma_a \to 1$ at infinity. %For fixed $\alpha\neq 0$, we want 
%$$(a-\mu_\alpha)(z) = B_a(z) \Sigma_a(z) G_a(z) \in H^-_\infty.$$%, \quad \psi,\phi\in L^2\cap L^\infty.$$
Choose 
$$G_a(z)= -\mu_\alpha\frac{(z-a_1)(z+\sigma)}{z^2+\tau z+\rho},$$
where $a_1\in\C_+,\sigma\in\C_-,\tau\in\C_-$ and $\rho<<-1$.

Being contractive in $\C_-$ the product $B_a\Sigma_a$ behaves like $e^{\frac{b}{z}}$ with $b\in\C_+$ at $\infty$. We can choose the constants above such that $a_1-\sigma=b-\tau\in \C_+$. Therefore $a(z)=O(1/z^2)$ at $\infty$. Defining $\phi(z)=(z-i)^{-1}\in H_2^-$, we have $\phi\sim 1/z$ at $\infty$ , so $\psi:=a/\ophi = O(1/z)$ belongs to $L^2$.

Choosing $\rho$ sufficiently negative, we get that both roots of $z^2+\tau z+\rho$, approximately $-\tau/2\pm\sqrt{-\rho}$, lie in $\C_+$. This gives $G_a\in H^-_\infty$ and 
since both its roots $a_1$ and $-\sigma$ lie in $\C_+$, $G_a$ is outer in $\C_-$
By Theorem \ref{B2}, the number of roots in $\C_-$ equals $\ind\Sc$ provided $\Sigma_a\equiv 1$ and therefore all natural numbers are possible as indices.
\end{Example}

\appendix
\section{Proofs of the results of Section \ref{section:7}} 
%\section*{Appendix A: }

We start with an elementary uniqueness lemma, whose proof we include for completeness.

\begin{lemma}\label{lem:unique}
Let $$G(\mu)= \begin{cases} G_+(\mu),\ \mu\in\C_+, \\  G_-(\mu),\ \mu\in\C_-. \end{cases}$$ with $G_\pm\in H_1^\pm + H_2^\pm:=\{G_1+G_2:\ G_1\in H_1^\pm, G_2\in H_2^\pm\}$. Then the jump across the real axis $[G]\equiv 0$ iff and only iff $G\equiv 0$.
\end{lemma}

\begin{proof} We show that if the jump vanishes, then $G$ vanishes, the other implication is trivial.

By $G(x\pm i0)$ we denote the boundary values of the functions $G_\pm$ which exist a.e.~on $\R$ (see \cite{Koosis}). Since $[G]\equiv 0$, we have $G(x):=G(x+i0)=G(x-i0)$. Now, $G(x+ i0)\in H_1^+ + H_2^+$, so for all $\la\in\C_-$ we have 
$$\int_\R\frac{G(t)}{t-\la}\equiv 0.$$ Similarly, $G(x- i0)\in H_1^- + H_2^-$, so for all $\la\in\C_+$ we have 
$$\int_\R\frac{G(t)}{t-\la}\equiv 0.$$
Combining this gives 
$$\int_\R\frac{G(t)}{t-\la}\equiv 0\quad\hbox{ for all } \la\not\in\R.$$

We want to show that this implies $G(x)=0$ for almost every $x\in\R$. Clearly, taking complex conjugates,
$$ \int_\R\frac{\Re G(t)}{t-\la}\equiv 0 \equiv \int_\R\frac{\Im G(t)}{t-\la},$$
so we may now assume without loss of generality that $G$ is real. Then
$$ 0=\Im \int_\R\frac{G(t)}{t-\la} = \Im\la \int_\R\frac{G(t)}{|t-\la|^2} = \pi P_\eps(G)(k),$$
where $\la=k+i\eps$ and $P_\eps$ denotes the Poisson transformation.

Now, let $I$ be any bounded open interval in $\R$ and $k\in I$. As $G=G_1+G_2\in L^1(\R) + L^2(\R)$, we have $\chi_IG\in L^1(\R)$  and
$$\eps\int_{I^c}\frac{G_1(x)+G_2(x)}{|x-\la|^2}\ dx \leq \eps\int_{I^c}\frac{|G_1(x)|+|G_2(x)|}{|x-k|^2}\ dx = o(1),$$
where we have used that $G_2(x)/|x-k|^2\in L^1(I^c)$. Therefore, 
$$P_\eps(G)=P_\eps(\chi_IG+\chi_{I^c}G)=P_\eps(\chi_IG) + o(1) \to \chi_IG \hbox{ as } \eps\to 0.$$
This shows that $\chi_IG(x)=0$ and as $I$ was arbitrary, we get $G(x)=0$ for almost every $x\in\R$.
As the boundary values of $G_\pm$ vanish identically, we have $G_\pm\equiv 0$.
\end{proof}

We can now prove the characterisation of $\Sc^\perp$.

\begin{proof}[Proof of Proposition \ref{sperpcrit}]
 \begin{enumerate}
	 \item Noting the form of elements from $\Ran S_{\lambda,B}$ from \eqref{eq:7}, we get that
 \bea 
 \Sc^\perp&=&\Big\{ g\in L^2(\R): \ \hbox{ for all } \mu\notin\R \\
&&\ F_\pm(\mu):=\llangle \frac{1}{x-\mu},g\rrangle-\frac{1}{D(\mu)}\llangle \frac{1}{x-\mu},\phi\rrangle \llangle \frac{1}{x-\mu}\psi,g\rrangle=0\Big\}.
\eea
Here the index $\pm$ indicates which half plane $\mu$ lies in. Then
\be
F_\pm(\mu)=\widehat{\overline{g}}(\mu)-\frac{1}{D(\mu)}\widehat{\overline{\phi}}(\mu)\widehat{\psi\overline{g}}(\mu).
\ee
First assume $g\in \Sc^\perp$.
Let $\mu\in \C_+\setminus\{D(\mu)=0\}$. Then $F_+(\mu)=0$ is equivalent to
\be\label{eqplus} (P_+ \og)(\mu)=\frac{2\pi i}{D_+(\mu)}(P_+\ophi)(\mu)P_+(\psi\og)(\mu). \ee
Similarly, if $\mu\in \C_-\setminus\{D(\mu)=0\}$, then $F_-(\mu)=0$ is equivalent to
\be\label{eqminus} (P_- \og)(\mu)=-\frac{2\pi i}{D_-(\mu)}(P_-\ophi)(\mu)P_-(\psi\og)(\mu).\ee
This proves the first implication in the statement. 

Now, as the zeroes of $D(\mu)$ in $\C_\pm$ form a discrete set, the right hand sides of \eqref{eqplus}  and \eqref{eqminus} must lie in $H_2^+$ and $H_2^-$, respectively, showing (i) and (ii). Since $\og=P_+\og+P_-\og$, we also get (iii).
For the reverse implications, simply apply $P_+$ and $P_-$ to $\og$ as given in (iii). 

\item We first show the equivalence in \eqref{eq:67}. Let $g\in \Sc^\perp$ and apply $P_\pm$ to \eqref{eq:66}, part (iii), keeping  \eqref{eq:66} parts (i) and (ii) in mind. Then 
$$P_+ \og-\frac{2\pi i}{D_+}(P_+\ophi)P_+(\psi\og)=0 \quad\hbox{ and }\quad 
 P_- \og+\frac{2\pi i}{D_-}(P_-\ophi)P_-(\psi\og)=0.$$
Since $D_{\pm}$ only have discrete zeroes, this is equivalent to
$$ D_\pm P_\pm (\og)\mp 2\pi i(P_\pm\ophi)P_\pm(\psi\og)=0.$$
In particular, on $\R$ we have
$$ D_+ P_+ (\og) - 2\pi i(P_+\ophi)P_+(\psi\og) = - D_- P_- (\og) - 2\pi i(P_-\ophi)P_-(\psi\og)\ (a.e).$$
Since for the boundary values on $\R$ we have $D_-=1-2\pi i P_-(\psi\ophi) = D_+ -2\pi i \psi\ophi$ a.e., this is equivalent to 
$$  D_+ (P_+ (\og)+P_-(\og)) - 2\pi i(P_+\ophi)P_+(\psi\og) = 2\pi i \psi\ophi P_- (\og) - 2\pi i(P_-\ophi)P_-(\psi\og)\ (a.e.)$$
which can be rewritten as
$$  D_+ \og - 2\pi i(P_+\ophi)\psi\og = 2\pi i \psi\ophi P_- (\og) - 2\pi i \ophi P_-(\psi\og)\ (a.e.),$$
giving the right hand side of \eqref{eq:67}.

On the other hand, assume the right hand side of \eqref{eq:67} holds. Retracing our steps above, this gives 
\be\label{jump}\widetilde{F}_+:= D_+ P_+ (\og) - 2\pi i(P_+\ophi)P_+(\psi\og) = - D_- P_- (\og) - 2\pi i(P_-\ophi)P_-(\psi\og)=:\widetilde{F}_-\ (a.e).\ee
Using H\"{o}lder's inequality and boundedness of the Riesz projections $P_\pm:L^p\to L^p$ for $1<p<\infty$ for the cases $p=\frac43$, $p=2$ and $p=4$ (see proof of Proposition \ref{ev} for more details), our conditions on $\psi$ and $\phi$ guarantee that $\widetilde{F}_\pm\in H_1^\pm + H_2^\pm$. Moreover, \eqref{jump} states  that $[\widetilde{F}]=0$.
By Lemma \ref{lem:unique} we have $\widetilde{F}=0$ and so 
$$D_+ P_+ (\og) - 2\pi i(P_+\ophi)P_+(\psi\og)=0 \hbox{ and } D_- P_- (\og) + 2\pi i(P_-\ophi)P_-(\psi\og)=0.$$
Therefore all conditions on the right hand side of \eqref{eq:65} are satisfied and $g\in \Sc^\perp$.

Using the identity $P_-=I-P_+$ then gives \eqref{eq:68} and a similarly simple calculation leads to \eqref{eq:69}.
\end{enumerate}
\end{proof}

\subsection{\change{Results depending on the support of $\phi$ and $\psi$}} 
%Here, $\widehat{f}(k\pm i0)=\lim_{\eps\to 0}\int_\R\frac{f(x)}{x-(k\pm i\eps)}dx$.

\begin{proposition}\label{ev}
	Let  either $\phi\in L^2$ and $\psi\in L^2\cap L^\infty$  or $\phi,\psi\in L^2\cap L^4$ be such that $\phi\psi=0$. Then 
	\ben\label{disjointcond}
	 g\in\overline{\Sc}^\perp & \Longleftrightarrow &  \og-2\pi i(P_+\ophi)P_+(\psi\og)+2\pi i(P_-\ophi)P_-(\psi\og)=0\ (a.e.)\\ \label{disjointcond2}
	& \Longleftrightarrow &  \og- 2\pi i\overline{\phi}P_+(\psi\og)+2\pi i(P_-\overline{\phi})\psi\og  =0\ (a.e.)
	\een
Define $L_0=2\pi i\left(\overline{\phi}P_+\psi-(P_-\overline{\phi})\psi\right)\equiv 2\pi i\left(\overline{\phi}P_+\psi+(P_+\overline{\phi})\psi\right)$ on its maximal domain $D(L_0)= \{u\in L^2: L_0u\in L^2\}$.
 Then we have $\overline{\Sc}=L^2(\R)$ if and only if $1$ is not an eigenvalue of the operator $L_0$ on $L^2(\R)$.
\end{proposition}

\begin{proof}%[Proof of Proposition \ref{ev}]
In this case, $D(\lambda)\equiv 1$, so \eqref{disjointcond} is equivalent to condition (iii)  on the right hand side of  \eqref{eq:66} and one implication is trivial. 
By  Proposition \ref{sperpcrit} it is sufficient to show that conditions (i) and (ii) on the right hand side of  \eqref{eq:66} hold. In our case, this simplifies to 
$$\mathrm{(i')}\ (P_+\ophi)P_+(\psi\og)\in H_2^+ \hbox{ and }
  \mathrm{(ii')}\ (P_-\ophi)P_-(\psi\og)\in H_2^-.$$
	
	In the case when $\phi\in L^2$ and $\psi\in L^2\cap L^\infty$, we have that $P_+\ophi\in L^2$ and $P_+(\psi\og)\in L^2$ by boundedness of the Riesz projection $P_+:L^2\to L^2$, so the product lies in $H_1^+$ (see \cite{Koosis}). On the other hand, if $\phi,\psi\in L^2\cap L^4$, then $\psi\og\in L^{4/3}$ by H\"{o}lder's inequality, so $P_+(\psi\og)\in L^{4/3}$. Also $P_+\ophi\in L^4$, so by H\"{o}lder's inequality, we again get that  the product lies in $H_1^+$. Similarly $(P_-\ophi)P_-(\psi\og)\in H_1^-$.

From \eqref{disjointcond}, $\og= 2\pi i(P_+\ophi)P_+(\psi\og)-2\pi i(P_-\ophi)P_-(\psi\og)$ which gives a decomposition of $\og\in L^2$ into its unique $H^+$ and $H^-$ components, whence we obtain (i') and (ii').

Now,
\bea
\og&=&2\pi i(P_+\ophi)P_+(\psi\og)-2\pi i(P_-\ophi)P_-(\psi\og)\\
&=& 2\pi i \left[(P_+\ophi)P_+\psi-(P_-\ophi)P_-\psi\right]\og\\
&=& 2\pi i \left[\ophi P_+\psi-(P_-\ophi)(P_+\psi+P_-\psi)\right]\og\ =\ L_0\og,
\eea
which shows \eqref{disjointcond2} and the statement about $L_0$.
\end{proof}

\begin{remark}
Note that $D(L_0)$ is dense in $L^2$ if $\phi,\psi\in L^2\cap L^4$ or if $\phi\in L^{2+\eps}$ for some $\eps>0$ and $\psi\in L^2\cap L^\infty$, as it contains $L^2\cap L^\infty$.
\end{remark}

\begin{proof}[Proof of Theorem \ref{trivS}]
We recall that $\R=\Omega_{\phi}\cup\Omega_{\psi}\cup\Omega$. Without loss of generality, we may assume $\Omega_{\phi}\cap\Omega_{\psi}=\emptyset$.  Suppose  $g\in \overline{\Sc}^\perp$. Using \eqref{disjointcond2}, we then have three cases:
\begin{enumerate}
	\item $k\in\Omega_{\phi}$: we have $0= \overline{g}(k)-\overline{\phi}(k)\widehat{\psi\chi_{\Omega_{\psi}} \overline{g}}(k+i0)$. Hence,
	\be%\label{cond1}
	\chi_{\Omega_{\phi}}  \overline{g}(k)= \overline{\phi}(k)\widehat{\psi\chi_{\Omega_{\psi}} \overline{g}}(k+i0)
	\ee
	and so $g\vert_{\Omega_{\psi}}$ completely determines $g\vert_{\Omega_{\phi}}$.
	\item  $k\in\Omega_{\psi}$: we have  $0= \overline{g}(k)-\widehat{\overline{\phi}}(k-i0)\psi(k)\overline{g}(k).$ Hence, almost everywhere we have
	\be%\label{cond2} 
	\chi_{\Omega_{\psi}}\overline{g}(k)\left(1-\widehat{\overline{\phi}}(k-i0)\psi(k)\right)=0.
	\ee
	\item $k\in\Omega$: we have
	\be%\label{cond3}
	g\vert_{\Omega}=0.
	\ee
\end{enumerate}
This gives three necessary and sufficient conditions for $g$ to lie in $\overline{\Sc}^\perp$.

We now prove the statements (i)-(iii).
\begin{enumerate}
	\item[(i)] If $\widehat{\overline{\phi}}(k-i0)\psi(k)\neq 1$ for almost every $k\in\Omega_{\psi}$, \eqref{cond2} implies that $g\vert_{\Omega_{\psi}}=0$, and so by \eqref{cond1} we have $g\vert_{\Omega_{\phi}}=0$ whence $g\equiv 0$ and $\overline{\Sc}=L^2(\R)$.
	\item[(ii)] Choose $g$ on $\Omega_{\psi,0}$ to be an arbitrary non-zero $(L^2\cap L^\infty)(\Omega_{\psi,0})$-function (in the case when $\phi\in L^{2+\eps}, \psi \in L^\infty$, we may even choose $g$ arbitrary in $(L^2\cap L^{\frac{2(2+\eps)}{\eps}})(\Omega_{\psi,0})$). Extending $g$ by zero to $\Omega_\psi$ and then using H\"older's inequality and boundedness of $P_+$, we automatically get  that $\overline{\phi}\left(\widehat{\psi\chi_{\Omega_{\psi}} \overline{g}}\right)_+\in L^2$. This then determines $g$ on $\Omega_\psi$  from \eqref{cond1} and extending $g$ to $\Omega$ by $0$, we have $g\in \Sc^\perp$.
	
	Now let $f\in \Sco$ and choose $g\in (L^2\cap L^\infty)(\Omega_{\psi,0})$. Then
	\bea
	0 &=& \int_\R f\og \ =\ \int_{\Omega_{\psi,0}} f \og + \int _{\Omega_\phi} f \og \ = \ \int_{\Omega_{\psi,0}} f \og + \int _{\Omega_\phi} f \overline{\phi}\left(\widehat{\psi \overline{g}}\right)_+  \\ & = &  \int_{\Omega_{\psi,0}} f \og + \int _{\R} f \overline{\phi}\left(\widehat{\psi\overline{g}}\right)_+ 
	\ =\  \int_{\Omega_{\psi,0}} f \og - \int _{\R} \left(\widehat{ f \ophi}\right)_- \psi \overline{g}\\  &= &  \int_{\Omega_{\psi,0}} f \og - \int _{\Omega_{\psi,0}} \left(\widehat{ f \ophi}\right)_- \psi \overline{g} \ =\ \ \int_{\Omega_{\psi,0}} \left(f -\psi\left(\widehat{ f \ophi}\right)_-\right)\overline{g},
	\eea
	from which it follows that  $\og\mapsto \int_{\Omega_{\psi,0}} \left(f -\psi\left(\widehat{ f \ophi}\right)_-\right)\overline{g}$ 
	is a bounded linear functional for all $g$ in the dense set $ (L^2\cap L^\infty)(\Omega_{\psi,0})$. Hence, $\psi\left(\widehat{ f \ophi}\right)_- \in L^2(\Omega_{\psi,0})$  and we see that $f\in \Sco$ must satisfy $f=\psi\left(\widehat{f \ophi}\right)_-$  on $ L^2(\Omega_{\psi,0})$.
	\item [(iii)] If $\psi$ and $\phi$ are both bounded, then choosing $g\in L^2\cap L^\infty(\Omega_{\psi,0})$ and determining $g$ on $\Omega_\phi$ from \eqref{cond1} will give a dense set of $g$ in $\Sc^\perp$, as $g\mapsto \overline{\phi}\left(\widehat{\psi \overline{g}}\right)_+$ is bounded. The calculation in (ii) can then be repeated for a dense set of $g$ in $\Sc^\perp$ which gives the needed equality of the two sets.
\end{enumerate}
\end{proof}

%\begin{remark}
%In fact, as $\psi\chi_{\Omega_{\psi}} \overline{g}\in L^p$ for any $p$, we have $\widehat{\psi\chi_{\Omega_{\psi}} \overline{g}}\in L^p$ for any $p$ by the Theorem of Riesz. Hence, for $g\vert_{\Omega_{\phi}}\in L^2$, it is sufficient to have $\phi\in L^{2+\eps}$ for any $\eps>0$.
%\end{remark}

\begin{proof}[Proof of Theorem \ref{thm:index}]

From \eqref{disjointcond2}, we have
\begin{align}\label{disjointcrit} % requires amsmath; align* for no eq. number
 g \in \Sc^\perp& \Longleftrightarrow \overline g(k)-\overline \phi(k)(\widehat{\psi \overline g})_+-(\widehat{\overline \phi})_-\psi\overline g=0 \;a.e. \ 
\Longleftrightarrow\ (1-\psi\widehat{(\overline{\phi})}_+)(\overline g - \overline \phi (\widehat{ \psi \overline{ g}})_+)=0 \;a.e.
\end{align}

We have two cases: In the first case,  $\overline g -\overline{ \phi} (\widehat{\psi\overline g})_+=0\ a.e.$ Then multiplying by $\psi$ and using the condition on the supports, we get $\psi \overline g=0$ and 
hence $\overline g =0$.

In the second case,
$\overline g -\overline \phi \cdot 2\pi i P_+(\psi\overline g) \not\equiv 0. $
Then there exists a set $E$ of positive measure such that 
$\overline{g} - \overline {\phi} \cdot 2\pi i P_+(\psi\overline g)\neq 0$ on $E$  and
 $\left(1-\psi   (\widehat {\overline{  \phi}})\right)\Big\vert_E=0\ a.e.$  

We now show that 
if there exists a set $E$ of positive measure such that $(1-\psi   (  \widehat  { \overline \phi} )_+)|_E=0$ a.e., then $\Sc^\perp\neq\{0\}$.
Note first that $E\subset \Omega_\psi$.  Choose any non-zero $\tilde g \in (L^2\cap L^\infty)(E)$, $\tilde g \neq 0$ and continue it to $\R$ by zero.

Define $$\overline g(k)=\left \{  \begin{array}{cc}
\overline {\tilde g(k)}, & k\in E,\\
0, & k\in \Omega_\psi\backslash E, \\
\overline{ \phi} 2\pi i P_+(\psi \overline{\tilde g}), & k\notin\Omega_\psi.
\end{array}\right .
$$
By \eqref{disjointcrit}, to show $g\in \Sc^\perp$ we only require $g\in L^2$. This follows immediately from the condition \eqref{epsiloncondition}. From the freedom in the choice of $\tilde{g}$, it is clear that $\mathrm{def}\ \Sc$ is infinite.
\end{proof}

\begin{lemma}\label{lem:count} Let $f:\R\to\C$ be measurable.
Consider the set $E_\alpha \subset \R$ given by $E_\alpha = \{ x\in\R : f(x) =\alpha\},\quad \alpha \in \C $.
Then $\{ \alpha \in \C : |E_\alpha |>0 \}$ is countable.
\end{lemma}

\begin{proof}
$\{ \alpha \in \C: |E_\alpha | >0 \}= \bigcup_{n=1}^\infty \{ \alpha \in \C : | E_\alpha |> 1/n \}$.
Assume the set is not countable. Then there exists $\eps_0$ such that 
$J:=\{\alpha \in \C: |E_\alpha | > \eps_0 \}$ is uncountable.
Moreover if $\alpha \neq \alpha'$, $E_\alpha \cap E_{\alpha'}=\emptyset$.
Now $\C \supset \cup_{\alpha \in J} E_\alpha$  is a disjoint union of uncountably many sets of measure greater than $\eps_0$.
Choose $\widetilde{f}_\alpha = \chi_{\tilde E_\alpha}$
where $\tilde  E_\alpha \subseteq E_\alpha$ with $|\tilde E_\alpha | =\eps_0$.  Then $\norm{\widetilde{f}_\alpha}=\sqrt{\eps_0}$ and $\widetilde{f}_\alpha\perp \widetilde{f}_{\alpha'}$ for $\alpha \neq \alpha'$.  This contradicts separability of $L^2(\R^2)$. Therefore, $\{ \alpha \in \C : |E_\alpha |>0 \}$ must be countable.
\end{proof}

\begin{proof}[Proof of Theorem \ref{generic}]
\begin{enumerate}
	\item This is immediate from Theorem \ref{thm:index} with the help of Lemma \ref{lem:count}.
		\item From \eqref{disjointcond2}, we have $g\in\Sc^\perp$ if and only if
\be\label{eq:diff2}  \overline{g}(k)- \overline{\phi}(k)\widehat{\psi \overline{g}}(k+i0)-\widehat{\overline{\phi}}(k-i0)\psi(k)\overline{g}(k) = 0. 
\ee
Multiplying by $\psi (k)$, and setting $g_\psi:=\psi \overline{g}$ we get
\be\label{eq:gpsi}  g_\psi(k)=  \psi(k)\widehat{\overline{\phi}}(k-i0)g_\psi(k). \ee
Hence, unless $ \psi(k)\widehat{\overline{\phi}}(k-i0)=1$ on a set of non-zero measure, we have $g_\psi\equiv 0$ which by \eqref{eq:diff2} gives $g\equiv 0$ and therefore $\overline{\Sc}=L^2(\R)$.

Now, as $\ophi\psi=0$, 
$$\alpha\psi(k)\widehat{\overline{\phi}}(k-i0) = -2\pi i \alpha\psi(k) P_-\ophi = 2\pi i \alpha\psi(k) P_+\ophi.$$
Assuming, that $\psi\in L^\infty$ and $P_+\phi$ or $P_-\ophi\in L^\infty$, we see that $\overline{\Sc}_\alpha=L^2(\R)$ for sufficiently small $|\alpha|$, as $|\alpha\psi(k)\widehat{\overline{\phi}}(k-i0)|<1$ and \eqref{eq:gpsi} implies $g_\psi\equiv 0$.
\end{enumerate}
\end{proof}

\begin{proof}[Proof of Theorem \ref{unsymm}]
\change{Clearly, $\phi$ and $\psi$ are both bounded and compactly supported, so \eqref{epsiloncondition} is satisfied.
We have $$\widehat{\overline{\phi}}(k+i0)\psi(k) = \left(\int_I\frac{dt}{t-k}\right)\chi_{I'}(k)\left(\int_I\frac{dt}{t-k}\right)^{-1}\ =\ \chi_{I'}(k)$$  and by Theorem \ref{thm:index}, we have $\overline{\Sc}\neq L^2(\R)$ and $\mathrm{def}\  \Sc=\infty$. On the other hand, 
$$\widehat{\overline{\psi}}(k+i0)\phi(k)= \chi_{I}(k) \int_{I'}\frac{1}{t-k}\left(\int_I\frac{du}{u-t}\right)^{-1}\ dt\neq 1\ a.e.\ \hbox{on}\ I,$$
since the inner integral is positive and the outer one negative. Hence, by Theorem \ref{thm:index},} $\overline{\Sct} = L^2(\R)$, and the first part of the theorem is proved.

For part (ii) we investigate the relationship between the $M$-function and the bordered resolvent in this example. As a first step, we investigate jumps of the $M$-function across the real axis.
Recall that 
\be M_B(\lambda) = \left( \mbox{sign}(\Im\lambda) \pi i -\widehat{\psi} (\lambda) \widehat{\overline{\phi}}(\lambda) - B\right)^{-1}
\ee
and that here we have $$ \widehat{\phi} (\lambda) = \int_I\frac{dt}{t-\lambda}\quad\hbox{ and }\quad \widehat{\psi} (\lambda)= \int_{I'}\frac{1}{t-\lambda}\left(\int_I\frac{du}{u-t}\right)^{-1}\ dt.$$

Suppose $k\in \R\setminus(I\cup I')$. Then both $\widehat{\phi}$ and $\widehat{\psi}$ are analytic near $k$ and  the jump of $M_B$ is given by the jump coming from the $\mbox{sign}(\Im\lambda) \pi i$ term.

Now suppose $k\in I$. Then the difference 
$\widehat{\phi}(k+i\eps)-\widehat{\phi}(k-i\eps)$ can be found using a contour integral and we get 
$$\widehat{\phi}(k+i0)-\widehat{\phi}(k-i0)= 2\pi i \chi_I(k).$$ 
Moreover, 
$$\lim_{\eps\to 0}\widehat{\psi}(k\pm i\eps)=\widehat{\psi}(k)=\int_{I'}\frac{1}{t-k}\left(\int_I\frac{du}{u-t}\right)^{-1}\ dt < 0 $$
is analytic on $I$.
Therefore, the jump of $M_B^{-1}$ at $k$ is $2\pi i(1-\widehat{\psi}(k))$ and since $\widehat{\psi}(k)<0$, the $M$-function will jump at $k$.

Finally, let $k\in I'$. Then $\widehat{\phi} (\lambda) = \int_I\frac{dt}{t-\lambda}$ is analytic in $k$, so does not jump, while
$$\left[ \widehat{\psi}(k) \right] = 2\pi i \left(\int_I\frac{du}{u-k}\right)^{-1}.$$
Therefore, $\left[\widehat{\overline{\phi}}(k) \widehat{\psi}(k) \right] = 2\pi i .$ This  cancels the jump from the $\mbox{sign}(\Im\lambda) \pi i$ term, and the formula for the jump of $M_B^{-1}$ in the second part is proved.

We  next examine the bordered resolvent. Note that, since $\overline{\Sct} = L^2(\R)$, only one projection is necessary here. 
We first consider the resolvent. Let $(A_B-\lambda)u=v$. Then (noting that $ D(\lambda) =1$ in our situation) we have from \eqref{eq:mm10}
\be u(x)  = \frac{v(x)}{x-\lambda}
                                 -\frac{\psi(x)}{x-\lambda}
 \left\langle \frac{v}{t-\lambda},\phi\right\rangle +
  c_u\left(\frac{1}{x-\lambda}-\frac{\psi(x)}{x-\lambda}
 \widehat{\overline{\phi}}(\lambda)\right). \label{FriedrichsResolvent}
\ee
We see immediately from \eqref{FriedrichsResolvent}  that the non-bordered resolvent (as the sum of a multiplication operator with a jump and finite rank operators) has a non-trivial jump across the whole real axis. Moreover, we get
\be \Gamma_1 u= \widehat{v}(\lambda) - \widehat{\psi}(\lambda) \left[ \left\langle \frac{v}{t-\lambda},\phi\right\rangle +
  c_u \widehat{\overline{\phi}}(\lambda)\right] + c_u\lim_{R\to\infty}\int_{-R}^R \frac{ dx}{x-\lambda},
\ee
where the value of the limit of the integrals in the last term is given by $\pi i \sign\Im(\lambda)$.
From \eqref{eq:mm10} we have 
\be
c_u = M_B(\lambda)\left(-\widehat{v}(\lambda)
 + \left\langle \frac{v}{t-\lambda},\phi\right\rangle
 \widehat{\psi}(\lambda)\right).
\ee

Since we are only interested in the restricted resolvent, we now fix $v\in\overline{\Sc}$. In this example both $\phi,\psi\in L^2\cap L^\infty$, so we may apply the final part of Theorem \ref{trivS} to obtain that 
%
%
 %and
%consider the three conditions \eqref{cond1}, \eqref{cond2}, \eqref{cond3} for a function $g$ to lie in $\overline{\Sc}^\perp$:
%$$ g \equiv 0 \quad \hbox{ on } \quad\R\setminus(I\cup I'), \quad g\vert_{I'} \quad\hbox{ is arbitrary and }\quad \overline{g}(k)= \widehat{\psi\chi_{I'} \overline{g}}(k+i0)\quad\hbox{ on } I.$$
%
%For any such $g$ we have 
%\bea
 %0&=& \int_I v\overline{g}+\int_{I'} v\overline{g}+\int_{\R\setminus(I\cup I')} v\overline{g} \\ 
 %&=&\int_I v\overline{g}+\int_{I'} v\overline{g}\\ 
 %&=&    \int_I v(k) \widehat{\psi\chi_{I'} \overline{g}}(k+i0)+\int_{I'} v\overline{g}\\
 %&=&  -  \int_{I'}  \widehat{v\chi_{I}}(k+i0)\psi \overline{g}+\int_{I'} v\overline{g} .
%\eea
%Here, we have used that $\int f\widehat{h} = -\int\widehat{f} h$.
%As $g$ is arbitrary on $I'$, this implies
\be\label{eq:v}
v\vert_{I'}= \widehat{v\chi_{I}}(k+i0)\left(\int_I\frac{dt}{t-x}\right)^{-1}\Big\vert_{I'},
\ee 
while $v$ is an arbitrary $L^2$-function outside $I'$.

Now consider the jump of $u(x,\lambda)=(A_B-\lambda)^{-1}v(x)$ for $k\in I'$:
\bea &&
u(x,k+i0)-u(x,k-i0)\\ && \hspace{50pt}= \underbrace{v(x) 2\pi i\delta(x-k)-\psi(x) 2\pi i \delta(x-k)\widehat{v\overline{\phi}}(k+i0)}_{\hbox{cancel due to \eqref{eq:v}}}-\frac{\psi(k)}{x-k}2\pi i v(k)\underbrace{\overline{\phi}(k)}_{=0}\\
&& \hspace{65pt}\ +\ c_u(k+i0) \underbrace{\left(2\pi i\delta(x-k)- 2\pi i\delta(x-k)\psi(x)\widehat{\overline{\phi}}(k)\right)}_{=0,\ \hbox{ as } \psi(k)\widehat{\overline{\phi}}(k)=1\ \hbox{ in } I'} \\
&&\hspace{65pt} \ +\ \left(\frac{1}{x-k}-\frac{\psi(x)}{x-k}\widehat{\overline{\phi}}(k)\right) M_B(k)\underbrace{\left(-2\pi iv(k)+ \left\langle \frac{v}{t-k},\phi\right\rangle 2\pi i\psi(k)\right)}_{=0\ \hbox{ due to \eqref{eq:v}}}\\
&&\hspace{50pt}= 0 \hbox{ on } I'.
\eea
Here, we have used that $[M_B(k)]=0$ for $k\in I'$.

We now consider the jump for $k\in I$, where $\psi$ is zero.
\ben \label{resjump}
u(x,k+i0)-u(x,k-i0)&=& v(x) 2\pi i\delta(x-k)
\ +\  \frac{[c_u(k)]}{x-k-i0}\\&&  \ +\ c_u(k-i0) 2\pi i\delta(x-k). \nonumber
\een
Next, we use that for any function $f$,
$$[(1/f)(k)]=-\frac{[f(k)]}{f(k+i0)f(k-i0)}$$
to obtain
\ben
[c_u(k)]&=& -[(1/c_u)(k)]c_u(k+i0)c_u(k-i0)\label{cjump} \\ \nonumber
&=& -[M_B^{-1}]\frac{c_u(k+i0)c_u(k-i0)}{-\widehat{v}(k-i0)+\widehat{\psi}(k-i0)\widehat{v\overline{\phi}}(k-i0) } \\
&&
    -\frac{c_u(k+i0)c_u(k-i0)}{M_B(k+i0)}\left[\frac{1}{-\widehat{v}(\lambda)
 + \left\langle \frac{v}{t-\lambda},\phi\right\rangle
 \widehat{\psi}(\lambda)}\right](k)
\nonumber \\
&=& [M_B^{-1}]M_B(k+i0)M_B(k-i0)\left(\widehat{v}(k+i0)-\widehat{\psi}(k)\widehat{v\overline{\phi}}(k+i0) \right)
\nonumber \\ && + \ M_B(k-i0)\left(2\pi i v(k)-\widehat{\psi}(k) 2\pi i v(k)\right). \nonumber
\een
Here, we have used the notation $[F(\lambda)](k)$ for the jump of the function $F$ of $\lambda$ at the point $k$.

To show that the restricted resolvent has a jump, it is sufficient to show that there exists $v\in\Sco$ such that the righthand side of \eqref{resjump} does not vanish identically. 
Since $\frac{1}{x+i0}=\pi i \delta(x) + v.p.(\frac1x)$ and the delta-functions cannot cancel the principal value, it is sufficient to show that there exists $v\in\overline{\Sc}$ such that $[c_u(k)]\neq 0$. To do this, choose $v=0$ on $I\cup I'$, but $v\in L^2$ not identically zero and sufficiently smooth. Then $v\in\overline{\Sc}$ and \eqref{cjump} becomes
$$[c_u(k)]=[M_B^{-1}(k)]M_B(k+i0)M_B(k-i0)\widehat{v}(k+i0)\neq 0,$$
as all terms of the product are non-zero. $M_B(k+i0)$ and $M_B(k-i0)$ are non-zero a.e. in $k$, as they are boundary values of a non-zero analytic function (see \cite{Koosis}).

Now, let $k \in \R\backslash I\cup I'$. Then
\begin{align*} % requires amsmath; align* for no eq. number
      & u(x,k+i0)-u(x,k-i0)=   2\pi i v\delta(x-k)-\underbrace{\psi(x)2\pi i \delta(x-k)}_0 v \hat{  \bar{ \phi}} (k+i0) \\
  & -\underbrace{\dfrac{ \psi(k)}{x-k}}_0  2 \pi i v(k)\bar \phi (k)+c_u(k+i0) \left ( 2\pi i\delta (x-k)-\underbrace{ 2\pi i \delta  (x-k)\dfrac{  \psi (x)}{x-k}\hat{\bar{ \phi}} (k)}_0 \right )\\
 &+\left [  M_B (k+i0) \left( -2\pi i v(k)\delta (x-k)+\llangle  \dfrac v{x-k},\phi \rrangle \underbrace{2\pi i \psi(k)\delta(x-k)}_0\right)\right. \\ 
 &+[M_b(k)]\left ( -\hat{v}(k-i0)+\widehat{v\bar{\phi}}(k-i0)\hat{\psi} (k-i0)\right )\Bigg]\left ( \dfrac1{x-k-i0}-\dfrac{ \psi(x)}{x-k-i0}\hat{ \bar{ \phi}} (k-i0)\right )\\
 =& 2\pi i v(x)\delta (x-k)+c_u(k+i0) 2\pi i \delta (x-k)-2\pi i v(x)\delta (x-k)\dfrac 1{x-k-i0}M_B(k+i0)\\
 &-M_B(k+i0)M_B(k-i0)2\pi i (1-\hat \psi (k))(-\hat{v}(k)+\widehat{v\bar{ \phi}}(k)\hat{ \psi}(k)).
 \end{align*}
 Therefore, the most singular term is $-2\pi i v(x) M_B(k+i0)\dfrac{\delta (x-k)}{x-k-i0}$
 which cannot be cancelled.  Clearly the coefficient does not vanish for any $v$ supported outside $I\cup I'$ with  $k\in \{x\in\R: v(x)\neq 0\}$.
 \end{proof}
%\begin{remark}
%Some of this should follow immediately from the abstract results, so much of the proof can be taken out. We could reduce it to just giving the formula for the jump. Stress that it's not obvious without the abstract result.
%\end{remark}

%\subsection{Results when $\phi$ and $\psi$ are not disjointly supported}

%\textcolor{red}{STILL NEEDS MORE DETAILED WORK FROM HERE ON.}

\begin{proof}[Proof of Theorem \ref{pos}]
\begin{enumerate}
	%\item This follows immediately from the representation theorem for Herglotz functions (see, e.g.~\cite[Appendix B]{Tes00}).
	\item Let \be\label{eq:Fgen}
F_\pm:=	
	P_\pm \og\mp\frac{2\pi i}{D_\pm}(P_\pm\ophi)P_\pm(\psi\og)=0. 
	\ee
	We consider the condition \eqref{eq:65}. Then $g\in \Sc^\perp$ implies $F_\pm\equiv 0$ and $[F]=0$. On $\Omega$ we  have
	$$\frac{\llangle \frac{1}{x-\mu},\phi\rrangle \llangle \frac{1}{x-\mu}\psi,g\rrangle}{1+\llangle \frac{\psi}{x-\mu},\phi\rrangle}$$ is analytic a.e., so its jump is zero. Therefore
	$$2\pi i g(k) = \llangle \frac{1}{x-\mu},g\rrangle\vert_{\Omega} = 0\ a.e.$$
	and $g$ vanishes a.e. on $\Omega$. 
	Moreover, since our conditions are symmetric in $\phi$ and $\psi$, we immediately get that also $\Sct^\perp\subseteq L^2(\Omega^c)$.
\item Consider $F_\pm$ from \eqref{eq:Fgen}. We need to show that if $F$ vanishes, then so does $g$. We have
$$D_\pm(\mu)=1+\alpha\int_I\frac{\ophi}{x-\mu}\ dx,$$
so 
$$F_\pm(\mu)=\llangle \frac{1}{x-\mu},g\rrangle -\frac{\alpha\int_I\frac{\ophi}{x-\mu}\ dx\int_I\frac{\og}{x-\mu}\ dx}{1+\alpha\int_I\frac{\ophi}{x-\mu}\ dx}.$$ 
Since $g\in\Sc^\perp$, the first part of the theorem implies $g\vert_{I^c} = 0 $ a.e., so 
$$F_\pm(\mu)= \int_I\frac{\og}{x-\mu}\ dx \frac{1}
{1+\alpha\int_I\frac{\ophi}{x-\mu}\ dx}.$$ 
Clearly, $
\left(1+\alpha\int_I\frac{\ophi}{x-\mu}\ dx\right)^{-1}\neq 0$ for a.e. $\mu$, so we have that for all $\mu\not\in\R$
$$\int_I\frac{\og}{x-\mu}\ dx = \llangle \frac{1}{x-\mu},g\rrangle =0.$$
As in the proof of Lemma \ref{lem:unique}, this implies $g\equiv 0$.
\end{enumerate}
\end{proof}

\begin{proof}[Proof of Theorem \ref{oneborderedres}]
From Theorem \ref{pos}, we know that $L^2(\Omega')\subseteq \Sco\cap\Scto$. Choose $v,\widetilde{v}\in L^2(\Omega')$. By assumption, we know $\llangle (A_B-\lambda)^{-1}v,\widetilde{v}\rrangle$. Noting that $v\phi=0$ and $\widetilde{v}\psi=0$,  we get from \eqref{eq:mm1c} and \eqref{eq:mm10} that
\be
\llangle (A_B-\lambda)^{-1}v,\widetilde{v}\rrangle - \llangle \frac{v}{x-\lambda}, \widetilde{v}\rrangle = - M_B(\lambda) \llangle \frac{v}{x-\lambda},  {\bf 1} \rrangle \llangle \frac{1}{x-\lambda}, \widetilde{v}\rrangle. 
\ee   
Choosing $v,  \widetilde{v} \geq 0$ and not identically zero, we can divide $\lambda$-a.e.~and obtain
\be
M_B(\lambda) = \frac{\llangle \frac{v}{x-\lambda}, \widetilde{v}\rrangle- \llangle (A_B-\lambda)^{-1}v,\widetilde{v}\rrangle}{\llangle \frac{v}{x-\lambda},  {\bf 1} \rrangle \llangle \frac{1}{x-\lambda}, \widetilde{v}\rrangle}.
\ee

\end{proof}

\subsection{Results with $\phi, \psi \in H_2^+$}

\begin{proposition}\label{Schar++}
Let $\phi, \psi \in H_2^+$. Then 
$$ g\in\Sc^\perp
  \Longleftrightarrow  \begin{cases} \mathrm{(I)}\ g\in H_2^+, \\
  \mathrm{(II)}\  \og = - \frac{2\pi i}{D_-}\ophi P_-(\psi\og) \ (a.e.).
   \end{cases}$$
\end{proposition}

\begin{proof}
We consider the conditions in \eqref{eq:65}. As $\ophi\in H_2^-$, we have $P_+\ophi=0$, giving $P_+\og=0$, hence $\og\in H_2^-$ and $g\in H_2^+$. Since $P_-\og=\og$ and $P_-\ophi=\ophi$, the second condition in \eqref{eq:65} becomes (II).
%
%
%
%We compare the conditions (I) and (II) 
%to the conditions (i)-(iii) given in \eqref{eq:66}. As $\ophi\in H_2^-$, condition (i) is trivially true and (iii) is equivalent to (II). (I) is equivalent to $\og\in H_2^-$ which by (II) is equivalent to $\frac{2\pi i}{D_-}\ophi P_-(\psi\og)\in H_2^-$ , which is precisely (ii).
\end{proof}

\begin{proof}[Proof of Theorem \ref{S++} (outline)]
%For the sake of brevity we give the proof only for the generic case $M=0=M_0$. 
%\textcolor{red}{Sergey has a new proof to insert here: delete blue stuff.}
We use the fact that $\overline{\mathcal S}=\overline{\mathcal T}$ where $\mathcal T$ is as defined
in (\ref{eq:mm1b}): the elements of ${\mathcal T}$ are found by solving $(\tilde{A}^*-\mu)u = 0$ and varying
$\mu$ over the resolvent set of some appropriate operators $A_B$. 
We therefore start by solving
$$
(\At ^* -\mu)u=(x-\mu)u-c_u  {\bf 1} +  \langle u,\phi \rangle \psi=0
$$
where $\phi, \psi \in H_2^+$. Dividing by $(x-\mu)$ we find that 
$$ 
u=\dfrac{ c_u {\bf 1} - \langle u,\phi \rangle \psi}{x-\mu}.
$$
Taking the inner product with $\phi$ we get
$$
D(\mu)\langle u,\phi \rangle - \llangle \frac{c_u}{x-\mu}, \phi\rrangle =0.
$$

There are two cases to consider.
\begin{enumerate}
\item
$\mu \in \C_+$.  This means $\llangle \dfrac{1}{x-\mu},\phi\rrangle =0$, and therefore $D(\mu)\langle u,\phi \rangle=0$.
%from which it follows that
%$ u=\dfrac{ c_u {\bf 1} - \langle u,\phi \rangle \psi}{x-\mu}$ and $\langle u,\phi \rangle = \llangle \dfrac{c u}{x-\mu},\phi \rrangle -\llangle u,\phi \rrangle \llangle \dfrac\psi {x-\mu},\phi \rrangle $.
There are two subcases to consider.
\begin{enumerate}
\item[(1a)]
$
D(\mu)\neq   0$ which implies $\langle u,\phi \rangle =0$, giving $u=\dfrac{{\bf 1}}{x-\mu}$ up to arbitrary constant multiples.
\item[(1b)]
$D(\mu)=0$ giving $u=\dfrac{ c_u {\bf 1} -\tilde c \psi}{ x-\mu}$ for arbitrary values $c_u$ and $\tilde c$.
%In particular those $u$ with $c_u=0$ are eigenfunctions with boundary condition $\Gamma_2u=0$; hence we exclude values of $\mu$ for which $D(\mu)=0$.
For any boundary condition $B$, by  suitable choice of the two constants we see that $\mu$ belongs to the spectrum of $A_B$. Therefore these functions need not be added to the space $\Sco$. However, functions $\dfrac{{\bf 1}}{x-\mu}$ are in $\Sco$ due to being able to approximate them using neighbouring values of $\mu$. 
\end{enumerate}
\item
We take $\mu\in \C_-$. Then
$$ \langle u,\phi\rangle D(\mu)=\llangle \dfrac{ c_u}{x-\mu},\phi \rrangle = -2 \pi i c_u \bar\phi(\mu).$$
There are  some subcases to consider.
\begin{itemize}
\item[(2a)]
$D(\mu) \neq 0$ which implies $u=c_u \dfrac{  1+(2\pi i \bar \phi (\mu)/D(\mu))\psi}{x-\mu}$ for arbitrary  $c_u$;
\item[(2b)]
$D(\mu)=0,$ $\bar \phi(\mu)=0$ giving by explicit calculation a two dimensional kernel: $u=\dfrac{ c_u {\bf 1} -\tilde c \psi}{x-\mu}$ for arbitrary values $c_u$ and $\tilde c$;
\item[(2c)]
$D(\mu)=0, \; \bar \phi(\mu)\neq 0$ giving $c_u=0$ and $u=\tilde{c}\dfrac\psi{x-\mu}$ for any $\tilde c$.
\end{itemize}
\end{enumerate} 
In the case (2b) for any boundary condition $B$, by  suitable choice of the two constants we see that $\mu$ belongs to the spectrum of $A_B$. Therefore these functions need not be added to the space $\Sco$. In the case (2c) the function $\dfrac\psi{x-\mu}$ should be included in $\Sco$. There is only one $B$ for which it is an eigenfunction (formally $B=\infty$), but even for this choice of $B$ it can be approximated by elements from neighbouring kernels with $D(\mu)\neq 0$.
%If $D(\mu)=0$ then $\mu$ lies in the spectrum of the operator $A_B$ for some $B$. For instance, in case (2c) we have $c_u=0$ which corresponds to $\Gamma_2u=0$, or $\Gamma_1 u -B\Gamma_2 u=0$ with $B=\infty$.  Given our definition of $\mathcal T$, these cases can be omitted. 
%In case (2b) we get $\Gamma_1u=0$  which corresponds to $B=0$. 
%Together with Proposition \ref{prop:2.7} 
Note that this means that $\Sco$ is independent of $B$ as it should be by Proposition \ref{prop:2.7}.
This proves the formula
for $\overline{\mathcal T} = \overline{\mathcal S}$ in the theorem.

We now obtain the expression for the dimension of $\Sc^\perp$, in the generic case $M=0=M_0$, when
$\psi(x) = \sum_{j=1}^n c_j/(x-z_j)$, where the $z_j$ are distinct, lie in $\C_-$ and the $c_j$ are all non-zero. We know that 
$g\in {\mathcal S}^\perp$ if and only if $g$ satisfies both conditions (I) and (II) in Proposition \ref{Schar++}:
Using the fact that $P_{-}=I-P_{+}$ the second condition becomes
\be (1-2\pi i P_{-}(\psi\overline{\phi})+2\pi i \overline{\phi} \psi)\overline{g}-2\pi i\ophi P_{+}(\psi\overline{g}) = 0. \label{eq:fedupnow}\ee
The first bracket on the left gives $D_+$ and by (I) we know that $\overline{g}\in H_2^-$ and so (\ref{eq:fedupnow}) becomes
\[ D_+(x)\overline{g}-2\pi i \overline{\phi}\sum_{j=1}^N c_j P_{+}\left(\frac{1}{x-z_j}\overline{g}\right) = 0, \quad g\in H_2^+, x\in\R\]
in which $D_+(x)$ are the boundary values on the real line of the function $D_+(\mu) = 1 + \int_{\mathbb R}\frac{\psi(x)\overline{\phi(x)}}{x-\mu}dx$, $\mu\in {\mathbb C}_+$. Thus by the Residue Theorem,
\be g\in H_2^+, \quad \overline{g}(x) = \frac{2\pi i \overline{\phi}(x)}{D_+(x)} \sum_{j=1}^N\frac{c_j \overline{g}(z_j)}{x-z_j}, \;\;\; x \in {\mathbb R}. \ee
Therefore, by unique continuation of the meromorphic function to the lower half plane (see \cite{Koosis}) $\og$ is given by
\be g\in H_2^+, \quad \overline{g}(\mu) = \frac{2\pi i \overline{\phi}(\mu)}{D_+(\mu)} \sum_{j=1}^N\frac{c_j \overline{g}(z_j)}{\mu-z_j}, \;\;\; \mu \in {\mathbb C}_-,\label{eq:grep} \ee
from which it is immediately clear that the space of all such $g$ is at most $N$-dimensional. Note that the expression on the right
hand side of the equality sign in (\ref{eq:grep}) is not clearly an element of $H_2^-$; to deal with this we
substitute the particular $\psi$ under consideration into the formula for $D_+ $ and use residue calculations to obtain the following expression for its analytic continuation to $\C$:
\be D_+ (\mu) = 1 - 2\pi i \sum_{j=1}^N \frac{\overline{\phi}(z_j)}{z_j-\mu}, \quad \mu\in {\mathbb C}. \label{eq:Dpoles} \ee

If $D_+ (\mu)$ has no zeros in $\overline{\mathbb C}_-$ and if $\overline{\phi}(z_j)\neq 0$ for all $j$ then we get
\[ \overline{g}(\mu) = 2\pi i \overline{\phi}(\mu)\sum_{j=1}^N \frac{c_j \overline{g}(z_j)}{D_+ (\mu)(\mu-z_j)},\quad \mu\in\C_- \]
and the condition that $\lim_{\mu\rightarrow z_j}\overline{g}(\mu) = \overline{g}(z_j)$ gives no additional restrictions, as can be confirmed by a simple explicit calculation.
In this case, therefore, the defect of $\overline{\Sc}$ is  $N$.

Now suppose $D_+ $ has zeros in $\overline{\C_{-}}$; for simplicity we are assuming that they all lie
strictly below the real axis. We let $\mu_1,\ldots, \mu_\nu$ be the distinct poles of $\overline{\phi}/D_+ $, with orders
$p_1,\ldots, p_\nu$ and set $P=\sum_{j=1}^\nu p_j$. In order to ensure that $g$ given by (\ref{eq:grep}) lies in $H_2^+$ we need that the conditions
\be \sum_{j=1}^N \frac{c_j}{(\mu_k-z_j)^{n}}\overline{g}(z_j) = 0, \;\;\; n = 1,\ldots, p_k, \;\;\; k = 1,\ldots, \nu , \ee
all hold - a total of $P$ linear conditions on the numbers $\overline{g}(z_1),\ldots,\overline{g}(z_N)$. We now check
that this is a full-rank system. Suppose for a contradiction that there is a non-trivial set of constants $\alpha_{i,k}$ such 
that 
\[ \sum_{k=1}^\nu \sum_{n=1}^{p_k} \frac{\alpha_{i,k}}{(\mu_k-z_j)^n} = 0, \;\;\; j = 1,\ldots,N. \]
Define a rational function by $F(z) = \sum_{k=1}^\nu \sum_{n=1}^{p_k} \frac{\alpha_{i,k}}{(\mu_k-z)^n}$ so that
$F$ has zeros at $z_1,\ldots,z_N$. Observe that $Q(z) := F(z)\prod_{k=1}^\nu  (\mu_k-z)^{p_k}$ is a polynomial
of degree strictly less than $P = \sum_{k=1}^\nu  p_k$, having $N$ zeros. Now $D_+ (\mu)\rightarrow 1$ as $\Im(\mu)\rightarrow\infty$,
so $D_+ $ has the same number of zeros as poles. In particular, $D_+ $ has at least as many poles in $\mathbb C$ as it has zeros
in ${\mathbb C}_-$, giving $N\geq P$. Thus $Q$ is a polynomial of degree $<P\leq N$ having $N$ zeros. 
This means $Q\equiv 0$, so $F\equiv 0$, and the constants $\alpha_{i,k}$ must all be zero. This contradiction
shows that the set of linear constraints on the $N$ values $\overline{g}(z_j)$ has full rank $P$, and so the set of
allowable values for $(\overline{g}(z_1),\ldots,\overline{g}(z_N))$ has dimension  $N-P$. 

The degenerated case leading to non-zero $M$ and $M_0$ can be analysed similarly by considering the local behaviour of $\ophi/D_+$ around zeroes of $D_+(x)$ on the real axis.
\end{proof}

\begin{proof}[Proof of Proposition \ref{proposition:8.13}]

We follow the construction of Theorem \ref{S++}, assuming additionally now that $\overline{\phi}$ has the same
form as $\psi$:
\[ \overline{\phi}(x) = \sum_{k=1}^{\tilde{N}}\frac{d_k}{x-w_k}, \quad w_k\in\C_+ \hbox{ distinct and } d_k\neq0. \]
We shall construct $\psi$ and $\phi$ so that the defect number $N-P$ of Theorem \ref{S++} takes any value between
$0$ and $N-1$, while the corresponding defect number $\tilde{N}-\tilde{P}$ for $\tilde{\mathcal S}$ takes any value between
$0$ and $\tilde{N}-1$, independently of the value of $N-P$.

In addition to the function $D(\mu)$ appearing in the proof of Theorem \ref{S++} we now have a function $\tilde{D}$ which,
following (\ref{eq:Dpoles}), has the form
\be \tilde{D}(\mu) = 1-2\pi i\sum_{k=1}^{\tilde{N}}\frac{d_k\overline{\psi}(w_k)}{w_k-\mu}. \label{eq:DTpoles} \ee
The expressions (\ref{eq:Dpoles},\ref{eq:DTpoles}) can be further developed using the explicit formulae for $\phi$ and $\psi$ to
obtain
\be D(\mu) = 1 - 2\pi i \sum_{j=1}^N\frac{c_j}{z_j-\mu}\sum_{k=1}^{\tilde{N}}\frac{\overline{d_k}}{z_j-\overline{w_k}}, \ee
\be \tilde{D}(\mu) = 1 - 2\pi i \sum_{k=1}^{\tilde{N}}\frac{d_k}{w_k-\mu}\sum_{j=1}^{N}\frac{\overline{c_j}}{w_k-\overline{z_j}}. \ee
We choose the points $z_j$ and $w_k$ so that $z_j-w_k\gg 1$ and $|z_j|\gg |w_k$ for all $j,k$. Then
\[ D(\mu) \approx 1 - 2\pi i \sum_{j=1}^N \frac{c_j/z_j}{z_j-\mu}\sum_{k=1}^{\tilde{N}} \overline{d_k}; \quad\quad \tilde{D}(\mu) \approx 1 + 2\pi i \sum_{k=1}^{\tilde{N}} \frac{d_k}{w_k-\mu}\overline{\sum_{j=1}^{N} \frac{c_j}{z_j}}. \]
From these expressions we can find approximations to the zeros $\mu_j$ of $D$ and $\tilde{\mu}_k$ of $\tilde{D}$,
\[ \mu_j \approx z_j - 2\pi i \frac{c_j}{z_j}\sum_{k=1}^{\tilde{N}}\overline{d_k}, \quad\quad
 \tilde{\mu}_k \approx w_k + 2\pi i d_k\overline{\sum_{j=1}^{N} \frac{c_j}{z_j}}; \]
these expressions can be written in the form
\[ \mu_j \approx z_j - 2\pi i \alpha_j\overline{d}, % \sum_{k=1}^{\tilde{N}}\overline{d_k}, 
\quad\quad \tilde{\mu}_k \approx w_k + 2\pi i d_k\overline{a}; \quad\quad a = \sum_{j=1}^N\alpha_j, \;\;\; d = \sum_{k=1}^{\tilde{N}}d_k, \;\;\;
 \alpha_j = c_j/d_j.\]
 
%and we can try to arrange, for given $1\leq P \leq N$ and $1\leq \tilde{P}\leq \tilde{N}$, that $P$ of the $\mu_j$ and
%$\tilde{P}$ of the $\tilde{\mu}_k$ lie in the lower half-lane.}

Inspecting these expressions one may deduce that it is possible to assign independently any value in $\{1,\ldots,N\}$ to
the number $P$ of $\mu_j$ in ${\mathbb C}_{-}$ and any value in $\{ 1,\ldots, \tilde{N}\}$ to the number $\tilde{P}$ of ${\tilde \mu}_k$ in
${\mathbb C}_{-}$. Thus $N-P$ may take any value in $\{0,\ldots,N-1\}$ and
$\tilde{N}-\tilde{P}$ may take any value in $\{0,\ldots,\tilde{N}-1\}$. Since $N$ and $\tilde{N}$ are arbitrary, the proof is
complete.
\end{proof}

%
%\begin{remark}
%Alternatively, one can use the explicit form of $\Sc$ and the orthogonality condition directly to calculate $g$.
%\end{remark}

\begin{proof} (Statements in Example \ref{specialcase}.)
In this example, for $\lambda\in \C^+$ we have by the residue theorem 
\be
D_+(\lambda) = 1+\alpha \int\left(\frac{1}{x-z_1}\cdot\frac{1}{x-w_1}\right) \frac{1}{x-\lambda} \ dx\ =\ 1+\frac{2\pi i \alpha}{(z_1-w_1)(\lambda-z_1)} = \frac{\lambda_0-\lambda}{z_1-\lambda}.
\ee
Clearly, this formula also gives the meromorphic continuation of $D_+$ to the lower half plane. We remark that this differs from $D_-$ which is given by
\be
D_-(\lambda)=1+\frac{2\pi i \alpha}{(z_1-w_1)(w_1-\lambda)}.
\ee
We now calculate the numbers $N,P,M,M_0$ from Theorem \ref{S++}. $\psi$ has a simple pole at $z_1\in\C_-$, hence $N=1$. As $\phi$ has no zeroes, $M_0=0$. The function $D_+$ has one pole at $z_1\in\C_-$, $\ophi$ has a simple pole at $w_1\in\C_+$. Thus all poles of $\ophi/D_+$ in $\overline{\C_-}$ stem from zeroes of $D_+$. The only zero of this function is at $$\lambda_0=z_1+\frac{2\pi i\alpha }{w_1-z_1}.$$ Thus, if 
 $\lambda_0\in\C_+$
	 then $P=M=0$; 
	if $\lambda_0\in\C_-$ then $P=1$, $M=0$;
	 if $\lambda_0\in\R$ then  $P=0$, $M=1$.
	
We next show the form of $\Sc^\perp$ and $\Sco$ in the case $\lambda_0\in\C_+$. Using $\ophi \in H_2^+$, from \eqref{eq:65}, we have $\og\in H_2^-$ and 
\be
\og = - \frac{2\pi i}{D_-}\ophi P_-(\psi\og). % = - \frac{2\pi i}{1+\frac{2\pi i \alpha}{(z_1-w_1)(\lambda-w_1)}}\ophi P_-\left(\frac{\alpha\og} {x-z_1}\right).
\ee
Hence, 
\be
\left(1+\frac{2\pi i \alpha}{(z_1-w_1)(\lambda-w_1)}\right) \og = -\frac{2\pi i \alpha}{\lambda-w_1} P_-\left(\frac{\og} {\lambda-z_1}\right) = -\frac{2\pi i \alpha}{\lambda-w_1} \left(\frac{\og-\og(z_1)} {\lambda-z_1}\right).
\ee
Noting that $\og(z_1)$ is a free parameter, a short calculation shows that
$$\og=\frac{-2\pi i \og(z_1)}{(\lambda-w_1)(\lambda-\lambda_0)}\quad \hbox{ or }\quad g(x)=\frac{const}{(x-\overline{w_1})(x-\overline{\lambda_0})}.$$
Now, $f\in\Sco$ iff
\be
0=\int f\og = const \int f(t)\left(\frac{1}{t-w_1}-\frac{1}{t-\lambda_0}\right).
\ee
This is equivalent to $(P_+f)(w_1)=(P_+f)(\lambda_0)$.
\end{proof}

\subsection{Analysis for the case $\overline{\phi},\psi\in H_2^+$}

\begin{proof}[Proof of Theorem \ref{notminuspi}]
We use the characterisation of $\Sc^\perp$ given in \eqref{eq:65}:
\bea
g\in\overline{\Sc}^\perp & \Longleftrightarrow & \begin{cases} P_+ \og-\frac{2\pi i}{D_+}\ophi P_+(\psi\og)=0, \\ 
 P_- \og=0.  \end{cases}\\
& \Longleftrightarrow & \og\in H_2^+ \hbox{ and } \og=\frac{2\pi i}{D_+}\ophi \psi\og.
\eea
Since  $D_+=1+2\pi i \psi\ophi$ on $\R$ we have
$$  \og\in H_2^+ \hbox{ and } (1+2\pi i \ophi\psi) \og=2\pi i\ophi \psi\og,$$
so $\og=0$.
\end{proof}

\subsection{The general case $\psi, \phi\in L^2$}

%As in the case of disjoint supports of $\psi$ and $\phi$, we have that generically the first factor on the r.h.s.~of \eqref{genfac} is non-zero:
      %Replacing $\psi$ by $\alpha \psi$ and denoting the corresponding detectable subspace by $\Sc_\alpha$, the condition to lie in $\Sc_\alpha^\perp$ becomes
    %$$ g \in \Sc_\alpha ^\perp \Longleftrightarrow [ 1 + 2\pi i \alpha (P_+(\overline \phi \psi)-\psi (P_+(\overline \phi)))][\overline g +2\pi i \alpha (P_+(\psi \overline \phi)\overline g -\overline \phi P_+(\psi \overline g))] = 0\ a.e.
    %$$

\begin{proposition} \label{enought}
 The set $E_0$ defined in \eqref{E0} is countable.
\end{proposition}

\begin{proof}[Proof of Proposition \ref{enought}] Let $\alpha\in E_0\setminus\{0\}$ and $E$ be the set of positive measure on which
$1 + 2\pi i \alpha (P_+(\overline \phi \psi)-\psi (P_+(\overline \phi)))=0$. Set $f=2\pi i (P_+(\overline \phi \psi)-\psi (P_+(\overline \phi)))$.  
As $1+\alpha f|_E=0$ then $f|_E=-1/\alpha$; this can only be true for a countable set of $\alpha$, as Lemma \ref{lem:count} shows.
\end{proof}

 \begin{proof}[Proof of Theorem \ref{thm:E0}] Without loss of generality, we asssume $\alpha=1$. 
Let $E$ be the set of positive measure from \eqref{E0}. 
For $\phi\in L^{2+\eps}$, choose $h\in L^2(E) \cap L^\infty(E)$, while if $\phi\in L^4$, then choose $h\in L^2(E)\cap L^4(E)$.
Now, set 
\be
\og=(P_+\ophi)\chi_E h - \ophi P_-(\chi_E h). 
\ee
By our assumptions on $h$ and in \eqref{epsiloncondition}, we have $g\in L^2$.

We next show that $\og$ satisfies the right hand side of \eqref{eq:67} pointwise. Note that here and in several other places in this proof we use that $P_-P_+f=0$. This is justified as our assumptions  on $h$ and in \eqref{epsiloncondition} guarantee that the functions $f$ we apply this to are in appropriate function classes.
 We have
\ben \nonumber
P_-\og &=& P_- ((P_+\ophi)\chi_E h - \ophi P_-(\chi_E h))\\ \nonumber
&=& P_- ((P_+\ophi)P_-(\chi_E h) - \ophi P_-(\chi_E h))\\ \nonumber
&=& P_- ((P_+\ophi - \ophi) P_-(\chi_E h))\\ 
&=& P_- (-(P_-\ophi) P_-(\chi_E h)) = -(P_-\ophi) P_-(\chi_E h).
\een
 Multiplying by $2\pi i\psi$ and using that $D_+-D_-=2\pi i\psi\ophi$ on the real axis by the Sohocki-Plemelj Theorem (see \cite{Koosis}), gives
\ben\nonumber
2\pi i\psi P_-\og  &=& -2\pi i\psi (P_-\ophi) P_-(\chi_E h) \\ \nonumber
&=& 2\pi i\psi (-\ophi+(P_+\ophi)) P_-(\chi_E h) \\ 
&=& (-(D_+-D_-)+2\pi i\psi (P_+\ophi)) P_-(\chi_E h). \label{eq:121}
\een
We rewrite the $D_-$-term as follows.
\ben
D_- P_-(\chi_E h) &=& P_- (D_- P_-(\chi_E h))\ =\ P_- ((D_--D_+) P_-(\chi_E h))+ P_- (D_+ P_-(\chi_E h))\\ \nonumber
&=& P_- ((D_--D_+) P_-(\chi_E h))+ P_- (D_+ \chi_E h).
\een
Inserting this in \eqref{eq:121}, and rearranging gives the identity
$$ 2\pi i\psi P_-\og - P_- ((D_--D_+) P_-(\chi_E h))- P_- (D_+ \chi_E h) = (-D_++2\pi i\psi (P_+\ophi)) P_-(\chi_E h).$$
Multiplying by $\ophi$ and using that on $E$ we have $D_+=2\pi i\psi(P_+\ophi)$ this gives
$$ 2\pi i \ophi\left( \psi P_-\og + P_- (\psi\ophi P_-(\chi_E h))- P_- (\psi(P_+\ophi) \chi_E h)\right) =- (D_+- 2\pi i\psi (P_+\ophi))\ophi P_-(\chi_E h),$$
which, noting that $(D_+- 2\pi i\psi (P_+\ophi))\chi_E h=0$, is 
 the equation on the right hand side of \eqref{eq:67}.

We now need to chose $h\in L^2(E)$ suitably to obtain an infinite dimensional subspace for the corresponding $\og$. Choose $E'\subset E$ with $|E'|>0$ and sufficiently small such that $\Omega_\phi\not\subseteq E'$ (as $E$ has positive measure and $\phi$ is not identically zero this is always possible). Consider
\be
\og=(P_+\ophi)\chi_{E'} h - \ophi P_-(\chi_{E'} h). 
\ee
By the above arguments, $g\in\Sc^\perp$. Moroever,
$$\og\vert_{(E')^c} = - \chi{((E')^c)} \ophi P_-(\chi_{E'} h).$$
As $\chi{((E')^c)} \ophi\not\equiv 0$ and $ P_-(\chi_{E'} h)$ are the boundary values of an analytic function and therefore non-zero a.e.~on $\R$, we have $\og\not\equiv 0$ whenever $ P_-(\chi_{E'} h)\not\equiv 0$ (see \cite{Koosis}), which gives an infinite dimensional set of such functions.
 \end{proof}

%[FOR THE REST OF THIS SECTION THE NOTES HAVE NOT YET BEEN TIDIED UP OR COMPLETED!]

 \begin{proof}[Proof of Theorem \ref{thm:indexx} (outline)]
This follows easily from \eqref{eq:69} in Proposition \ref{sperpcrit} and standard results on compact operators. The compactness of the difference of Hankel operators follows from \cite[Corollary 8.5]{Peller}.

 Consider the analytic operator-valued function
 $$
 I+({\mathcal M} -\mu I)^{-1} \K
 $$
 which is a compact perturbation of $I$. We need to know the values $\mu \in \C$ for which this operator has non-trivial kernel.
 Each connected component of $\C\backslash \overline{\mbox{essran}\ \M} $  either contains only discrete (countable) spectrum or else lies entirely in the spectrum. However for large $\mu$, $\{ 0 \} =   \ker(I+(\M-\mu)^{-1}\K)$, so by the Analytic Fredholm Theorem (see \cite{RS05}),  outside some bounded set   
 there is no spectrum of $\M+\K$.
 \end{proof}

\section{Proofs of the results of Section \ref{section:Toep}} \label{AppB}
%\section{ Detectability and spectral properties of Toeplitz operators}

%\begin{remark}
%We need to think about how much of this should go into main part - probably at least the results and the examples.
%\end{remark}

%\section*{Appendix B: Results on the spectrum of operators $T:=\overline \phi P_+\psi$}
In this appendix, we collect spectral results for an operator $T:=\overline \phi P_+\psi$ acting on $L^2$. These are closely related to the corresponding results for the Toeplitz operator $T_a:H^+_2\to H^+_2$ given by $T_a u =P_+au=P_+aP_+u$ which can be found, e.g.~in \cite{BS}. 

Assumptions \ref{AssAppB} are assumed to hold throughout.

\begin{proposition}\label{PropAppB}
%Let $\phi,\psi\in L^\infty$ and $a(z):=\psi(z)\overline \phi (z) \in H^-_\infty\backslash\{0\}$, so that $P_+(\psi \overline \phi)=0$. 
Define the operators  $T$ on $L^2$ and $T_a:H^+_2\to H^+_2$ as above. Then
\begin{enumerate}
%\item  $\ind(\Sc_\alpha)= \dim\ker\left(T_a-\frac{1}{2\pi i\alpha}\right)$. \label{Toep_estimate} 
\item $\sigma_p(T)\supseteq \{a(z)|z\in \C_-\}$. \label{pointspec}
\item $\sigma_p(T)=\{\mu\in \C \ |\ (a-\mu) \mbox{  is not an outer function in } \C_-\}\cup\{0\}$. \label{outer}
\item   $\mu \not \in \overline{  \Ran_{z\in \C_-}  a(z)}$ implies $\mu \in \rho(T)$. \label{T_res}
\item $\sigma(T)=\overline{\sigma_p(T)}=\overline{ \Ran_{z\in \C_-}   a(z)}.$ \label{spectrum}
\item $\sigma_p(T)\setminus\{0\}= \sigma_p(T_a)\setminus\{0\}$. \label{Toeplitz}
\end{enumerate}
\end{proposition}

\begin{remark}
The proofs of (i)-(iv) are very similar to the standard proofs for Toeplitz operators which can be found, e.g.~in \cite{BS}.
\end{remark}

\begin{proof}
Proof of \ref{pointspec}:
Take $u(k)=\tfrac{ \overline \phi(k)}{k-z_1},\;\;z_1\in \C_-$.
Then 
\bea % requires amsmath; align* for no eq. number
   Tu&=& \overline \phi P_+ \Big ( \dfrac{\psi\overline \phi }{k-z_1} \Big )\ = \overline \phi P_+\Big ( \dfrac{a(k)}{k-z_1} \Big ) \\
  & =& \overline \phi P_+\Big ( \dfrac{ a(k)-a(z_1)}{k-z_1}+\dfrac{a(z_1)}{k-z_1}\Big )\
   = \dfrac{a(z_1)}{k-z_1}\overline \phi(k)=a(z_1)u(k)
\eea
since the first term  acted on by $P_+ $ is analytic in $\C_-$ and in $L^2(\R)$ and the second is in $H_2^+$ since $z_1\in \C_-$.

Proof of \ref{outer}: 
We first consider $\mu=0$. Choosing $\og=\overline{\phi}h$ for $h\in H_2^-$, we get
$$T\og = \overline{\phi} P_+ \psi \overline{\phi}h =0,$$ since $\psi \overline{\phi}h \in H_2^-$.
Hence all functions in $\overline{\phi}H_2^-$ are  eigenfunctions to the eigenvalue $0$.

%Then
%\bea
% T\overline{g}=0 \ a.e. &\Longleftrightarrow & P_+\psi \overline{g} = 0 \ a.e \ \Longleftrightarrow\ \psi \overline{g}\in H^-_2
% %\ \Longleftrightarrow\ a \overline{g}\in \overline{\psi} H^-_2 
% \ \Longleftrightarrow\  \overline{g}\in \frac{1}{\psi} H^-_2.
%\eea

Now let $\mu\neq 0$ and assume that $(a-\mu)$ is an outer function in $\C_-$. We use that if $f\in H^-_\infty$, then $f$ is outer iff $\overline{fH_2^-}=H_2^-$ (Beurling Theorem, see \cite{Koosis}) and that the functions $(k-z_0)^{-1}$ for $z_0\in\C_+$ span $H_2^-$. Therefore 
$$\bigvee_{z_0\in\C_+} (a(k)-\mu)\frac{1}{k-z_0}=H^-_2.$$
Now assume there exists $g\in L^2\setminus\{0\}$ with $T\overline{g}=\mu \overline{g}$ and set $h=\psi\overline{g}$. Then $h\in L^1\setminus\{0\}$ and $aP_+h=\mu h$, or $(a-\mu)P_+h=\mu P_-h$. Let $z\in \C_+$, then
$$\int_\R (a-\mu)P_+h \frac{dk}{k-z} =\mu \int_\R P_-h \frac{dk}{k-z}\equiv 0.$$
Therefore, $P_+h\perp\frac{\overline{a}-\overline{\mu}}{k-\overline{z}}$ for all $z\in \C_+$, i.e.
$$P_+h\perp \bigvee_{\overline{z}\in\C_-} \frac{\overline{a}-\overline{\mu}}{k-\overline{z}} = H^+_2.$$ This implies $P_+h=0$, giving $P_-h=0$, so $h=0$ and, as $\psi$ is non-zero a.e.~(due to condition (i) in Assumptions \ref{AssAppB}), we have $g=0$, so $\mu$ is not an eigenvalue of $T$.

Next let $\mu\neq 0$ and assume that $(a-\mu)$ is not an outer function in $\C_-$. This implies that there exists $h\in H^-_2\setminus\{0\}$ such that $h\perp (a-\mu)H^-_2$. Now 
\bea
h\perp (a-\mu)H^-_2 &\Longleftrightarrow & \overline{h}\perp (\overline{a}-\overline{\mu}) H^+_2 
%\ \Longleftrightarrow\ \int_\R \overline{h}\overline{\frac{\overline{a}-\overline{\mu}}{k-z_0}} \hbox{ for all } z_0\in\C_-
%\ \Longleftrightarrow\ \int_\R \overline{h}\frac{a-\mu}{k-\overline{z_0}} \hbox{ for all } \overline{z_0}\in\C_+
\Longleftrightarrow\ (a-\mu) \overline{h}\in H^-_2, 
\eea
so $P_+(a-\mu)\overline{h}=0$ and $P_+a\overline{h}=\mu P_+ \overline{h} = \mu\overline{h}$ (as $\overline{h}\in H^+_2$). This implies 
$$\underbrace{\overline{\phi}P_+\psi}_{T}\overline{\phi}\overline{h}=\mu \overline{\phi}\overline{h}.$$
As $\phi\in L^\infty$, $\overline{\phi}\overline{h}\in L^2$ and it is not identically zero (as $\phi\not\equiv 0$ and $h$ is non-zero a.e.~by the uniqueness theorem, see \cite{Koosis}), so $\mu\in \sigma_p(T)$.

Proof of \ref{T_res}: We first note that $\mu \not \in \overline{  \Ran_{z\in \C_-}  a(z)}$ implies $\mu\neq 0$ and that $\inf_{z\in\C_-}|a(z)-\mu|>0.$ We want to calculate the resolvent of $T$ at $\mu$. Consider $(T-\mu)g=v$. Since $\psi\neq 0$ a.e.~and $(a-\mu)\vert_{\R}$ is invertible we get (all equalities hold a.e.)
\ben \nonumber
(T-\mu)g=v &\Longleftrightarrow & \psi\overline{\phi}P_+\psi g - \mu\psi g= \psi v 
\ \Longleftrightarrow\ (a-\mu)P_+(\psi g)- \mu P_-(\psi g) = \psi v \\ \label{eq:star}
&\Longleftrightarrow & P_+(\psi g) =\frac{\psi}{a-\mu}v + \frac{\mu}{a-\mu} P_-(\psi g).
\een
Note that, as $\frac{\mu}{a-\mu}\in H^-_\infty$, the last term lies in $H^-_2$.
Applying $P_+$ and $P_-$ to \eqref{eq:star}, we get 
$$P_+(\psi g)= P_+\frac{\psi v}{a-\mu}, \quad\quad 0=P_-\frac{\psi v}{a-\mu}+ \frac{\mu}{a-\mu} P_-(\psi g).$$
Thus, 
$$\psi g = P_+(\psi g) +P_-(\psi g)= P_+\frac{\psi v}{a-\mu} - \frac{a-\mu}{\mu}P_-\frac{\psi v}{a-\mu}$$
and 
$$g =\frac{1}{\psi}  P_+\frac{\psi v}{a-\mu} - \frac{a-\mu}{\psi\mu}P_-\frac{\psi v}{a-\mu}= \frac{v}{a-\mu}-\frac{a}{\mu\psi}P_-\frac{\psi v}{a-\mu}.$$
Formally, we have
\be
g=(T-\mu)^{-1} v =  \frac{v}{a-\mu}-\frac{\overline{\phi}}{\mu}P_-\frac{\psi v}{a-\mu}.
\ee
Since $\phi,\psi\in L^\infty$, the linear operator defined by the r.h.s.~is bounded in $L^2(\R)$.
We check the formal calculation of the resolvent:
\bea
(T-\mu)\left(  \frac{v}{a-\mu}-\frac{\overline{\phi}}{\mu}P_-\frac{\psi v}{a-\mu}\right) &=& \overline{\phi} P_+\frac{\psi v}{a-\mu} - \frac{\mu v}{a-\mu} - \frac{\overline{\phi}}{\mu} \underbrace{P_+ a P_- \frac{\psi v}{a-\mu}}_{=0} +\overline{\phi} P_- \frac{\psi v}{a-\mu}\\
&=&\overline{\phi} \frac{\psi v}{a-\mu} - \frac{\mu v}{a-\mu} \ =\ v.
\eea
Similarly, 
$$ \frac{(T-\mu) v}{a-\mu}-\frac{\overline{\phi}}{\mu}P_-\frac{\psi (T-\mu) v }{a-\mu} = v,$$
so $T-\mu$ is invertible.

Now, \ref{spectrum} follows from \ref{pointspec} and \ref{T_res}, as
$$\sigma(T)\subseteq \overline{  \Ran_{z\in \C_-}  a(z)} \subseteq \overline{\sigma_p(T)} \subseteq \sigma(T),$$
so all three sets must coincide.

Proof of \ref{Toeplitz}:
% We know that (should reference this)
%$$\ind S_\alpha := \dim S_\alpha^\perp = \dim \ker (T-\mu_\alpha).$$
We again solve $T\og =\mu\og.$ As $\psi\neq 0$ a.e., this is equivalent to $aP_+\psi\og=\mu\psi\og$. Setting $h=\psi\og$ this gives $aP_+h=\mu h$. 

Note that if $h\in L^2$ and $\mu\neq 0$ with $aP_+h=\mu h$, then $h=\psi\frac{\overline{\phi}P_+h}{\mu}\in\psi L^2$, so $\og=h/\psi\in L^2$. Thus $T\og =\mu\og$ is equivalent to $aP_+h=\mu h$. This reduces the problem to considering Toeplitz operators:
\bea
aP_+h=\mu h &\Longleftrightarrow& \left\{\begin{array}{rcl}  (P_+aP_+) P_+ h &=& \mu P_+h, \\ P_-aP_+h &=& \mu P_-h  \end{array}\right.
 \ \Longleftrightarrow\ \left\{\begin{array}{rcl}  (P_+aP_+) P_+ h &=& \mu P_+h, \\ P_-h &=& \frac1\mu P_-aP_+h  \end{array}\right.
\eea
Thus, $P_+h$ determines $P_-h$ uniquely and we only need to consider the first equation in $H^+_2$ which shows equality of the point spectra of $T$ and $T_a$ away from $0$.
\end{proof}

\begin{Example} We illustrate Proposition \ref{PropAppB} \ref{spectrum} with an example: We consider a case where we have no non-zero eigenvalues on the boundary of $\sigma_p(T)$, as $(a-\mu)$ is outer.
Let $\overline{\phi}(k)=(k+i)^{-1}$ and $\psi(k)=(k+i)(k-i)^{-4}$. Then $\phi,\psi\in L^2\cap L^\infty$, $\ophi\in H^-_2$ and $a(z)=(z-i)^{-4}\in H^-_1$, so Assumptions \ref{AssAppB} are satisfied. To determine $\Ran_{z\in\C_-} a(z)$, we consider $a(t), t\in\R$ and take the inside of the curve.

\begin{figure}[ht] %  figure placement: here, top, bottom, or page
    \centering\vspace{-180pt}
    \includegraphics[width=4in,height=8in]{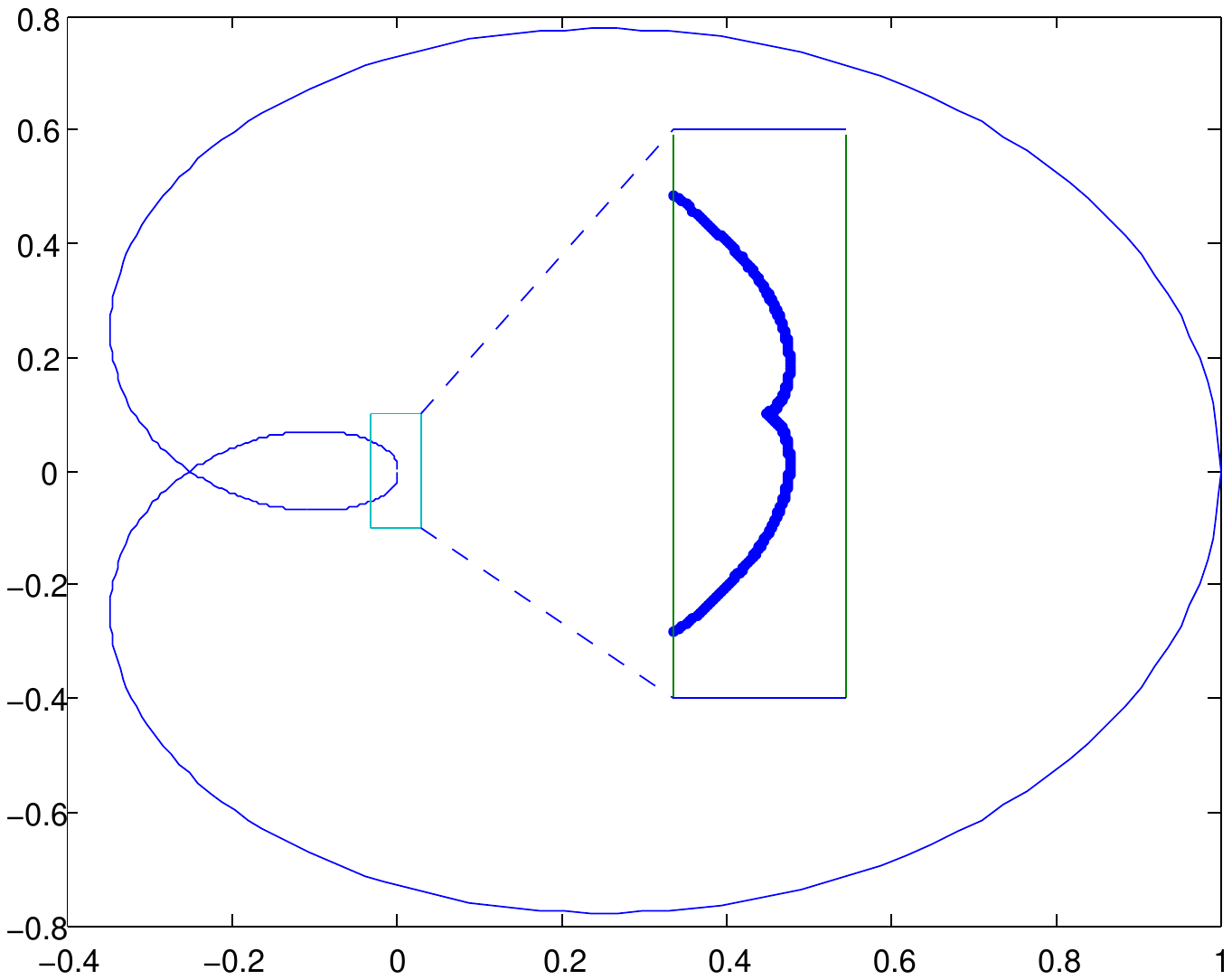}  \vspace{-180pt}
    \caption{The range of $a(t)=(t-i)^{-4}$ for $t\in\R$ with the section around the origin enlarged.}
    \label{fig:1}
\end{figure}

  Let 
$$x+iy:=\frac{1}{(t-i)^4}=\frac{(t+i)^4}{(t^2+1)^4} = \frac{t^4-6t^2+1}{(t^2+1)^4} + i \frac{4t(t^2-1)}{(t^2+1)^4}.$$
We first check that all non-zero points inside the inner curve are in the range.
Now, if $x$ is small and negative and $y=0$, then 
$$\frac{1}{(z-i)^4} = x \Longleftrightarrow (z-i)^4=\frac1x \Longleftrightarrow z-i=\frac{1}{\sqrt[4]{-x}}e^{\frac{i(\pi+2\pi m)}{4}}, \ m\in\Z,$$
so, e.g.~for $m=2$, $z=i+\frac{1}{\sqrt[4]{-x}}e^{\frac{5i\pi}{4}}\in\C_-$, so the point $x$ lies in $\Ran_{z\in\C_-} a(z)$. Similarly, we see that all points between the inner and outer curve lie in $\Ran_{z\in\C_-} a(z)$.
Next, we check the points on the inner curve (corresponding to $|t|>1$): Let $t>1$ and $(t-i)^{-4}=(z-i)^{-4}$. Then $z-i=(t-i)e^{i\frac{2\pi m}{4}}$, $m\in\Z$. With $m=3$, $z=i+(t-i)(-i)=-1+(1-t)i\in\C_-.$ 

Therefore, the boundary of the set $\Ran_{z\in\C_-} a(z)$ consists of the outer curve ($|t|\leq 1$) together with the isolated point 0. For these $\mu$, the function $(a-\mu)$ is outer in $\C_-$. For $\mu=0$ this is clear. For all other such $\mu$, $(a-\mu)$ takes values outside a cone. This implies that $(a-\mu)^k$ is outer for some sufficiently small positive $k$ which implies that $(a-\mu)$ is outer, since any Herglotz function (i.e.~analytic functions on $\C_+$ with positive imaginary part) is outer.

We finally consider the behaviour at $t=\pm 1$ and $t=\pm\infty$. As $t\to\pm\infty$, $x\sim t^{-4}$ and $y\sim 4t^{-5}$, so $y\sim 4x^{5/4}$. At $t=1$, $x\sim\frac{-4-8(t-1)}{16}$, $y\sim\frac{8(t-1)}{16}$, so $y\sim -x-\frac14$, and using symmetry of the range of $a$ w.r.t.~complex conjugation, we get a cone of angle $\pi/2$ at this point. 
\end{Example}

\begin{Example}
The next example shows that in statement \ref{outer} of Proposition \ref{PropAppB}, it is necessary to add the point $\{0\}$, as it is not always contained in the set $\{\mu: (a-\mu) \hbox{ is outer}\}$: Let $\alpha_0\in\R$ and consider 
$$\psi (z) = \frac{(z-\alpha_0)(z+i)}{(z-i)^3}e^{\frac{i}{z}}, \quad \overline{\phi}(z)=\frac{1}{z+i}, \quad \Im z\leq 0.$$
Then $\phi\in H_2^-$, %$\psi, \overline{\phi}\in H^-_2\cap H^-_\infty$, so 
$$ a(z)= \frac{z-\alpha_0}{(z-i)^3}e^{\frac{i}{z}}\in H_1^-\cap H^-_\infty$$
and $0\notin\Ran_{z\in\C_-}a(z)$. Due to the singular exponential factor, $a=(a-0)$ is not an outer function.
\end{Example}

\begin{proof}[Proof of Theorem \ref{B2}] 
Proof of \ref{Toep_estimate}:
As $a(z) \in H^-_1\backslash\{0\}$, we have $\phi,\psi\neq 0$ a.e. Moreover
for $g\in L^2$ we have from \eqref{eq:65} and using $\ophi\in H^+_\infty$
\ben g\in\Sc_\alpha^\perp \quad \Longleftrightarrow \quad
P_+ \og = 2\pi i\alpha\ophi P_+(\psi\og)\ \hbox{ and }\  P_- \og=0 \quad \Longleftrightarrow \quad \og = 2\pi i\alpha\ophi P_+(\psi\og).
\een
We rewrite this as
\be\label{kerchar}
T\overline{g}=\overline{\phi} P_+ (\psi \overline{g})=\mu_\alpha \overline{g},\ee
 giving $\Sc^\perp_\alpha=\left(\ker(T-\mu_\alpha)\right)^*$.

Next let $g\in \Sc_\alpha^\perp$ and
set $h=\psi\overline{g}$. Then, as $\psi \in L^\infty$ we have $h\in L^2$ and
\bea
g\in \Sc_\alpha^\perp &\Longleftrightarrow & \psi\overline{\phi} P_+ h = \mu_\alpha h, h\in L^2 \ \Longleftrightarrow aP_+h=\mu_\alpha h 
  \ \Longleftrightarrow \ \left\{\begin{array}{ccc} P_+aP_+h&=&\mu_\alpha P_+h, \\ P_-aP_+h&=&\mu_\alpha P_-h. \end{array}\right.
\eea
For the first equivalence, we note that any $L^2$-solution of $\psi\overline{\phi} P_+ h = \mu_\alpha h$ with $ \mu_\alpha\neq 0$  is divisible by $\psi$ and $h/\psi\in L^2$.

This shows that $P_+h$ uniquely determines $P_-h$ via $P_-h=\frac1\mu(P_-aP_+)P_+h$ and it is sufficient to consider $P_+aP_+h=\mu_\alpha P_+h$ which gives $\Sc^\perp_\alpha\subseteq\overline{\psi^{-1}}[(\mu_\alpha P_+ + P_-aP_+)\ker\left(T_a-\mu_\alpha\right)]^*$. On the other hand, given $h_+\in\ker\left(T_a-\mu_\alpha\right)$, set $\og=\overline{\phi}h_+$.
Then
$$T\og= \overline{\phi} P_+ (\psi\og)= \overline{\phi} P_+ a h_+ = \mu_\alpha \overline{\phi} h_+ = \mu_\alpha\og$$
gives the reverse inclusion, since $\Sc^\perp_\alpha=\left(\ker(T-\mu_\alpha)\right)^*$.
%Since $h\in L^2$ does not necessarily imply $\overline{g}=h/\psi\in L^2$ this only gives the estimate claimed in \eqref{Toep_estimate}.

Proof of \ref{Salpha}: Using the characterisation \eqref{kerchar}, we need to study the equation
$$ (T-\mu_\alpha)\og=\left(\overline{\phi} P_+\psi -\mu_\alpha\right)\og =0.$$
We consider the equation pointwise and multiply by $\psi$. Setting $h=\psi\og$, we get $h\in L^2$ and
\be\label{eq:psi}
aP_+h-\mu_\alpha h =0.
\ee
By virtue of \eqref{eq:psi}, the fact that $a$ is divisible by $\psi$ and $\mu_\alpha\neq 0$,
$h/\psi=\og\in L^2$. Now, using $h=P_+h+P_-h$, we find
$$(a-\mu_\alpha)P_+h=\mu_\alpha P_-h.$$
Thus $(a-\mu_\alpha)P_+h\perp H^+_2$, or $P_+h\perp (\overline{a}-\overline{\mu_\alpha})H^+_2$, which implies $P_+h\in H^+_2\ominus (\overline{a}-\overline{\mu_\alpha})H^+_2$. From \eqref{eq:psi}, this implies
$$ h \in \frac{a}{\mu_\alpha} \left(H^+_2\ominus (\overline{a}-\overline{\mu_\alpha})H^+_2\right)$$
and dividing by $\psi$ (which is non-zero a.e.), we get
$$ \og \in \frac{\overline{\phi}}{\mu_\alpha}\left( H^+_2\ominus (\overline{a}-\overline{\mu_\alpha})H^+_2\right) =\overline{\phi} \left(H^+_2\ominus (\overline{a}-\overline{\mu_\alpha})H^+_2\right).$$ 
Taking complex conjugates and using $(H_2^+)^*=H_2^-$ implies one set inclusion. Conversely, let $\og\in \overline{\phi} \left(H^+_2\ominus (\overline{a}-\overline{\mu_\alpha})H^+_2\right).$ Then $\og=\overline{\phi}f_+$ for some $f_+\in H^+_2\ominus (\overline{a}-\overline{\mu_\alpha})H^+_2$. Then
\bea
\left(\overline{\phi} P_+\psi -\mu_\alpha\right)\og &=& \overline{\phi} P_+\psi\overline{\phi}f_+ -\mu_\alpha\overline{\phi}f_+\ =\ \overline{\phi} (P_+a -\mu_\alpha)f_+ \ =\ \overline{\phi} P_+(a -\mu_\alpha)f_+ \ =\ 0,
\eea
as $(a -\mu_\alpha)f_+ \in H^-_2$. Hence $g\in \Sc_\alpha^\perp$ by part (i).

%By part \eqref{Toeplitz} we're looking for functions $h_+\in H^+_2$ such that $(a-\mu)h_+ \in H^-_2$. This is equivalent to $h_+ \perp (\overline{a}-\overline{\mu})H^+_2$. Therefore, 
%$$\ind \Sc_\alpha = \dim \left( H^+_2\ominus (\overline{a}-\overline{\mu})H^+_2 \right) = \dim \left( H^-_2\ominus (a-\mu_\alpha) H^-_2 \right).$$
%Moreover,
%\bea
%\Sc_\alpha^\perp &=&\left\{ g\in  L^2: T\og = \mu_\alpha \og \right\}\\
%&=& \left\{ g\in  L^2: \og = \frac{P_+h+\frac{1}{\mu_\alpha}P_-aP_+h}{\psi},\ P_+h \perp (\overline{a}-\overline{\mu_\alpha})H^+_2\right\}\\
%&=& \left\{ g\in  L^2: \og \in   \frac1\psi\left(I+\frac{1}{\mu_\alpha}P_-aP_+\right) \left(H^+_2\ominus (\overline{a}-\overline{\mu_\alpha})H^+_2   \right) \right\}\\
%&=&   \frac{1}{\overline{\psi}}\left(I+\frac{1}{\overline{\mu_\alpha}}P_+\overline{a}P_-\right)\left(H^-_2\ominus (a-\mu_\alpha) H^-_2 \right)
%\eea

Proof of \ref{index}: Since $(a-\mu_\alpha)\in H^-_\infty$ we have the canonical factorisation 
$$a(z)-\mu_\alpha=B_\alpha(z) \Sigma_\alpha(z)G_\alpha(z).$$ In $\C_-$,   $B_\alpha \Sigma_\alpha$ is an inner function and $G_\alpha$ is an outer function.
As $G_\alpha$ is outer, by Beurling's Theorem, the closure
$$ \overline{(a-\mu_\alpha)H^-_2}= B_\alpha(z) \Sigma_\alpha(z)H^-_2.$$
%For $\alpha \in \C\setminus E_0$ 

Thus, by  part \ref{Salpha}, $\Sc^\perp_\alpha = \phi\left( H_2^-\ominus B_\alpha(z) \Sigma_\alpha(z)H^-_2\right)$.
This gives 
\eqref{Blaschkedefect}, since $\phi\neq 0$ a.e.
\end{proof}

\end{document}